\documentclass[oneside,english,11pt]{amsart}
\usepackage{lmodern}
\usepackage[T1]{fontenc}       			 
\usepackage[latin9]{inputenc}  			 
\usepackage[english]{babel}			 
\usepackage[top=3cm, bottom=3cm, left=3.5cm, right=3.5cm, heightrounded, marginparwidth=2.5cm, marginparsep=6mm]{geometry}

\usepackage{setspace} 
\usepackage{verbatim}
\usepackage{mathrsfs}	
\usepackage{amstext}
\usepackage{amsthm}
\usepackage{amssymb}
\usepackage{amsopn} 
\usepackage{bbold}
\usepackage{stix}
\usepackage{enumerate}					
\usepackage[all]{xy}    				
\usepackage{mathdots}   				
\usepackage{aliascnt}   				
\usepackage{esint}						
\usepackage{tikz}  \usetikzlibrary{matrix}	
\usepackage{hyperref}
\usepackage{mathdots}
\hypersetup{
    colorlinks,
    linkcolor={red!50!black},
    citecolor={green!50!black},
    urlcolor={blue!80!black},
    linktocpage
}
\usepackage{soul}  
\usepackage{longtable} 

\numberwithin{equation}{section}
\numberwithin{figure}{section}

\theoremstyle{theorem}
  \newtheorem*{cor*}{Corollary}
  \newtheorem*{thm*}{Theorem}
  \newtheorem*{lem*}{Lemma}
  \newtheorem*{claim*}{Claim}
  \newtheorem*{mclaim*}{Main Claim}

  \newtheorem{thm}{Theorem}[section]

  \newaliascnt{lem}{thm}
  \newtheorem{lem}[lem]{Lemma}
  \aliascntresetthe{lem}

  \newaliascnt{klem}{thm}
  \newtheorem{klem}[klem]{Key Lemma}
  \aliascntresetthe{klem}

  \newaliascnt{cor}{thm}
  \newtheorem{cor}[cor]{Corollary}
  \aliascntresetthe{cor}

  \newaliascnt{prop}{thm}  
  \newtheorem{prop}[prop]{Proposition}
  \aliascntresetthe{prop}

\theoremstyle{definition}
  \newaliascnt{defn}{thm}
  \newtheorem{defn}[defn]{Definition}
  \aliascntresetthe{defn}

  \newaliascnt{exmpl}{thm}
  \newtheorem{exmpl}[exmpl]{Example}
  \aliascntresetthe{exmpl}

\theoremstyle{remark}
  \newaliascnt{rem}{thm}
  \newtheorem{rem}[rem]{Remark}
  \aliascntresetthe{rem}

\newcommand{\bbC}{\mathbb{C}}

\newcommand{\bbE}{\mathbb{E}}

\newcommand{\bbH}{\mathbb{H}}

\newcommand{\bbN}{\mathbb{N}}

\newcommand{\bbP}{\mathbb{P}}
\newcommand{\bbQ}{\mathbb{Q}}
\newcommand{\bbR}{\mathbb{R}}

\newcommand{\bbZ}{\mathbb{Z}}       

\newcommand{\clC}{\mathcal{C}}

\newcommand{\clG}{\mathcal{G}}
\newcommand{\clH}{\mathcal{H}}

\newcommand{\clL}{\mathcal{L}}

\newcommand{\clP}{\mathcal{P}}

\newcommand{\frC}{\mathfrak{C}}

\newcommand{\frS}{\mathfrak{S}}

\newcommand{\scL}{\mathscr{L}}

\newcommand{\kay}{\mathscr{k}}

\DeclareMathOperator{\dd}{{\rm d}}
\DeclareMathOperator{\Gr}{Gr}    
\DeclareMathOperator{\id}{id}          
\DeclareMathOperator{\HH}{H}

\DeclareMathOperator{\LL}{{\rm L}}
\DeclareMathOperator{\Prob}{Prob}
\DeclareMathOperator{\Span}{span}

\newcommand{\Hb}{{\rm H}_{\rm b}}
\newcommand{\Hc}{{\rm H}_{\rm c}}
\newcommand{\Hcb}{{\rm H}_{\rm cb}}
\newcommand{\Linfty}{L^\infty}

\newcommand{\Binfty}{\scL^\infty}

\newcommand{\bcdot}{{\scriptscriptstyle \bullet}}

\newcommand{\IE}{{}^{\rm I}{\rm E}}
\newcommand{\Id}{{}^{\rm I}{\rm d}}
\newcommand{\IIE}{{}^{\rm II}{\rm E}}
\newcommand{\IId}{{}^{\rm II}{\rm d}}

\newcommand{\BX}{\bar{X}}

\DeclareMathOperator{\GL}{GL}         
\DeclareMathOperator{\SL}{SL}         
\DeclareMathOperator{\SO}{SO}             
\DeclareMathOperator{\SU}{SU}             
\DeclareMathOperator{\Sp}{Sp}              

\newcommand{\qand}{\quad \mathrm{and} \quad}

\AtBeginDocument{\addtocontents{toc}{\protect\setlength{\parskip}{0pt}}}

\begin{document}

\title[Stability in bounded cohomology]{Stability in bounded cohomology for classical groups, I:\\ The symplectic case}
\author{Carlos De la Cruz Mengual}
\address{Departement Mathematik, ETH Z\"urich, R\"amistrasse 101, 8092 Z\"urich, Switzerland}
\email{carlos.delacruz@math.ethz.ch}
\thanks{}

\author{Tobias Hartnick}
\address{Mathematisches Institut, Arndtstra\ss e 2, 35392 Gie\ss en, Germany}
\email{tobias.hartnick@math.uni-giessen.de}
\thanks{}

\begin{abstract}
We show that continuous bounded group cohomology stabilizes along the sequences of real or complex symplectic Lie groups, and deduce that bounded group cohomology stabilizes along sequences of lattices in them, such as $(\Sp_{2r}(\bbZ))_{r \geq 1}$ or $(\Sp_{2r}(\bbZ[i]))_{r \geq 1}$. Our method is based on a general stability criterion which extends Quillen's method to the functional analytic setting of bounded cohomology. 
This criterion is then applied to a new family of complexes associated to symplectic polar spaces, which we call symplectic Stiefel complexes; similar complexes can also be defined for other families of classical groups acting on polar spaces.
\end{abstract}

\maketitle
\tableofcontents

\section{Introduction}
\subsection{Statement of results}
Given a sequence $\HH^\bcdot = (\HH^q: \clC \to \clC')_{q \geq 0}$ of functors, one says that $\HH^\bcdot$ \emph{stabilizes} along a sequence $(G_r)_{r \geq 1}$ of objects in $\mathcal C$ provided 
\[
\forall\; q \geq 0\quad\exists\; r_0 = r_0(q)\colon \quad \HH^q(G_{r_0}) \cong  \HH^q(G_{r_0+1}) \cong \HH^q(G_{r_0+2}) \cong \cdots 
\]
Any function $q \mapsto r_0(q)$ satisfying this property will be called a \emph{stability range} for $\HH^\bcdot$ along $(G_r)_{r \geq 0}$. The present article is concerned with stability results of bounded group cohomology along sequences of lattices in simple Lie groups. Here bounded group cohomology $\Hb^\bcdot$ is understood in the sense of Gromov \cite{Gromov}; we refer the reader to\ \cite{Frigerio} for background concerning bounded cohomology. Our main result is as follows.

\begin{thm}\label{ThmLatticeIntro} Let $\kay \in \{\bbR, \bbC\}$ and for every $r \in \mathbb N$, let $\Gamma_r$ be a lattice in $\Sp_{2r}(\kay)$. Then $\Hb^\bcdot$ is stable along $(\Gamma_r)_{r \geq 1}$. In particular, $\Hb^\bcdot$ is stable along $(\Sp_{2r}(\bbZ))_{r \geq 1}$ and  $(\Sp_{2r}(\bbZ[i]))_{r \geq 1}$.
\end{thm}

This result should be compared to results of Monod, which imply stability of $\Hb^\bcdot$ along the families $(\SL_r(\bbZ))_{r \geq 1}$ and  $(\SL_r(\bbZ[i]))_{r \geq 1}$, or more generally, families of lattices in $\SL_r(\kay)$ for a local field or skew-field $\kay$ (e.g. $\bbR, \bbC, \bbQ_p$ or $\bbH$). To derive these results from \cite{Monod-sot, Monod-Stab}, one first deduces from \cite[Cor. 1.4]{Monod-sot} that if $(G_r)$ is a sequence of simple real Lie groups with real ranks ${\rm rk}_\bbR G_r \to \infty$ and for every $r \in \bbN$ we choose a lattice $\Gamma_r < G_r$, then bounded cohomology stabilizes along $(\Gamma_r)_{r \geq 1}$ provided \emph{continuous bounded cohomology} $\Hcb^\bcdot$ stabilizes along $(G_r)_{r \geq 1}$. Since $\Hcb^\bcdot$ is stable along $(\SL_r(\bbR))_{r \geq 1}$ by \cite[Thm. 1.1]{Monod-Stab}, one then obtains stability of $\Hb^\bcdot$ along $(\SL_r(\bbZ))_{r \geq 1}$. The case of $(\SL_r(\bbZ[i]))_{r \geq 1}$ can be handled similarly via stability of the family $(\SL_r(\bbC))_{r \geq 1}$. By the same argument our main theorem is implied by the following.

\begin{thm}\label{ThmIntro} $\Hcb^\bcdot$ is stable along $(\Sp_{2r}(\kay))_{r \geq 1}$ for $\kay \in \{\bbR, \bbC\}$.
\end{thm}

So far, this is the only other stability result for families of simple Lie groups beyond the case of $\SL_r(\kay)$. The only other related result that we are aware of concerns the rank-one families $(\SO(r,1))_{r \geq 1}$ and $(\SU(r,1))_{r \geq 1}$. For these, Pieters \cite{Pieters} proved that for $k+1 \leq r$ there are injections
\[
\Hcb^k(\SO(r,1)) \hookrightarrow \Hcb^k(\SO(r+1,1)) \qand \Hcb^k(\SU(r,1)) \hookrightarrow \Hcb^k(\SU(r+1,1)).
\]
Surjectivity of these is not known for any non-trivial range. There is, however, a much richer conjectural picture: according to conjectures usually attributed to Dupont \cite{Dupont} and Monod \cite{Monod-Survey}, the continuous bounded cohomology of a simple Lie group with finite center should be isomorphic to its continuous group cohomology. The latter can be computed explicitly by virtue of theorems by van Est and Cartan, and existing computations of the cohomology of symmetric spaces; see e.g. \cite{Stas} and \cite{GHV}. Thus, the conjecture would imply stability for $\Hcb^\bcdot$ along many classical families, including in particular the remaining classical families of split Lie groups over $\bbR$ and $\bbC$, namely $(\SO_{2r}(\bbC))$, $(\SO(r,r))$, $(\SO_{2r+1}(\bbC))$ and $(\SO(r, r+1))$, and with a stability range which is essentially linear in the respective real ranks. It seems that the conjectures of Dupont and Monod are currently out of reach, whereas we suggest here a general method to obtain stability results for all of the split families above, thus providing new evidence for the conjectures. So far, our method has only been fully implemented for the real and complex symplectic groups, but we believe that it can be extended (with some additional measure-theoretic difficulties) to other classical families over any local field or skew-field.

\subsection{Bounded-cohomological Quillen's method and symplectic Stiefel complexes}
The meth-od we propose here comprises two parts: Firstly, we give a general criterion that guarantees stability of $\Hcb^\bcdot$ along a given family of Lie groups $(G_r)$. This criterion is essentially a very general version of the argument used by Monod in his stability proof. To apply this criterion one needs to construct a suitable family of complexes for the groups $G_r$. In Monod's case, the desired complexes can be readily constructed using the fact that $\SL_r(\kay)$ acts essentially multiply transitively on the corresponding projective space. It is known that no other family of classical groups admits essentially highly multiply transitive actions on flag varieties \cite{Popov, Devyatov}; therefore, Monod's method does not apply beyond the special linear groups. In order to establish our main theorem, we construct more complicated complexes related to symplectic groups, which we call \emph{symplectic Stiefel complexes}. Establishing that these complexes have the required properties will occupy most of the current article, and is significantly much harder to check than in the case of special linear groups. 

Let us comment briefly on both parts of our proof. Our stability criterion is based on a functional-analytic version of a classical method in group cohomology, sometimes known as \emph{Quillen's method}; see \cite{Bestvina} and the references mentioned therein. In the purely algebraic setting, the method works as follows: Given a nested sequence $G_0 \subset G_1 \subset G_2 \subset G_3 \subset \cdots$ of groups, suppose one constructs for every index $r$ a semi-simplicial set\footnote{Roughly speaking, a set endowed with a structure that enables the definition of its simplicial (co)homology. Semi-simplicial sets are also known as $\Delta$-complexes in the literature, though we will refrain from the use of this terminology here. For the precise definition of a \emph{semi-simplicial object} in a category, see the first paragraph of Section \ref{sec:abstract_stability}.} $X_{r,\bcdot}$ such that
\begin{enumerate}[(Q1)]
	\item $X_{r,\bcdot}$ is \emph{increasingly connected} in function of $r$, i.e. the reduced simplicial homology $\tilde{\HH}_\bcdot(X_{r,\bcdot};\bbZ)$ vanishes up to a certain degree $\gamma(r)$, and $\gamma(r) \to \infty$ as $r \to \infty$;
	\item $X_{r,\bcdot}$ is \emph{increasingly transitive} in function of $r$, i.e. the group $G_r$ acts transitively on the collection $X_{r,k}$ of $k$-simplices for all $k \in \{0,\ldots,\tau(r)\}$, and $\tau(r) \to \infty$ as $r \to \infty$; and
	\item the semi-simplicial sets $X_{r,\bcdot}$ are \emph{compatible} with the sequence $(G_r)$, i.e.\ the point stabilizer of the transitive $G_r$-action on $X_{r,k}$ is isomorphic to $G_{r-k-1}$ (or at least up to a finite kernel). 
\end{enumerate} 
Then, it is possible to show by a spectral sequence argument that the existence of these semi-simplicial sets implies $\HH^\bcdot$-stability up to a range that relates directly to how much connectivity and transitivity one is able to show for them. This idea has been applied successfully in a number of contexts, see e.g. \cite{vdK, Harer, HatVogt, HatVogtWahl, Essert}. Our bounded-cohomological version of Quillen's method will replace conditions (Q1)-(Q3) above by corresponding measurable versions. Our complexes will be semi-simplicial objects in a suitable category of measure spaces; we will define a notion of measurable connectivity based on measurable bounded functions; transitivity will be understood up to null sets; and point stabilizers have to be isomorphic to smaller groups in the series only up to amenable kernels. For a precise statement, see \autoref{thm:abstract_stability} below.

The complexes to which we will apply our criterion are going to be defined as follows: Recall that the non-compact Stiefel variety $X_{r,k}$ of bases of $k$-dimensional subspaces of $\kay^r$ is a fiber bundle over the corresponding Grassmannian $\Gr_k(\kay^r)$. Symplectic Stiefel varieties can be defined similarly, by considering first the ``symplectic Grassmannian'' consisting of $k$-dimensional \emph{isotropic} subspaces of a symplectic vector space $(\kay^{2r}, \omega)$. Over these, one then has corresponding ``symplectic Stiefel varieties'' consisting of ordered bases, and these varieties can be arranged into a complex in which the face maps are given simply by forgetting one of the basis vectors. It is easy to see that the symplectic groups act increasingly transitively on these symplectic Stiefel complexes. The hardest part of the proof is to establish increasing measurable connectivity; for this we introduce a new method of constructing ``random homotopies'' in measure spaces, which may be of independent interest. 

In our proofs we have to fight some additional technical difficulties which are intrinsic to bounded cohomology. Namely, bounded cohomology is not exact along arbitrary long exact sequences of coefficient modules, but only along so-called \emph{dual sequences}. In practice, this means that one has to construct explicit probability measures in the canonical invariant measure class of symplectic Stiefel varieties, such that integration against these measures is dual to the boundary maps in the associated complex of $\Linfty$-function classes. This is not an issue in the case of special linear groups, where the Stiefel varieties are measurably isomorphic to products of projective spaces, and any choice of product measures is appropriate. In the case of symplectic groups there does not seem to be a similarly simple work-around; we resolve the problem by applying the co-area formula from geometric measure theory. There are also additional problems related to the fact that one is always working with function classes rather than actual functions, and finally there are a number of measurability questions; these technical issues are responsible for the length of the current article.

\subsection{Outlook and future work}
The definition of symplectic Stiefel complexes can be given in terms of the associated symplectic polar spaces, and similar Stiefel complexes can be constructed for more general polar spaces. We observe that all the other classical families of split groups mentioned above can also be realized as automorphism groups of polar spaces. This suggests a strategy to establish stability for these families by applying our general criterion to the corresponding Stiefel complexes. We believe that this approach does indeed work, but since some of the necessary technical verifications become substantially harder in these cases, we leave them for another time. Another problem not addressed in the current article concerns finding the optimal stability range in the symplectic case. The stability range we obtain here for symplectic groups is exponential in the real rank, compared to a conjecturally linear stability range. It is possible that the symplectic Stiefel complexes have better connectivity properties than what we prove here, but establishing better bounds would most likely require a deeper understanding of the finer combinatorial properties of symplectic Stiefel complexes. Finding the optimal connectivity result concerning symplectic Stiefel complexes is an interesting combinatorial problem even for symplectic groups (and other classical groups acting on polar spaces) over finite fields.

\subsection{Structure of the article}
The structure of this article is as follows. In Section \ref{sec:abstract_stability}, we state \autoref{thm:abstract_stability}, the bounded-cohomological Quillen's method for sequences of nested lcsc groups, giving precise definitions of the category in which our semi-simplicial objects will lie and of the measure-theoretical versions of Properties (Q1)--(Q3); the proof of our criterion is deferred to Section \ref{sec:proof_abstract_stability}. We then demonstrate our criterion by recovering Monod's $\Hcb^\bcdot$-stability theorems.\footnote{The reason for writing this example in detail is that there is an inaccuracy in the induction step that yields the original stability bounds in \cite{Monod-Stab}; this is mentioned by its author in the Note before Lemma 10 in \cite{Monod-Vanish}. The bounds are amended in \autoref{MonodRange} and \autoref{MonodRange2} below.} We then explain the proof of our main theorem, assuming the existence of a sequence of semi-simplicial objects $X_{r,\bcdot}$ with suitable properties. In Section \ref{sec:stiefel}, we introduce formally the symplectic Stiefel complex $X_{\bcdot}$ associated to a symplectic vector space $(V,\omega)$ and establish all of their desired properties except for ``measurable connectivity''. Its proof is tackled in Sections \ref{sec:contractibility} and \ref{sec:magic}. In the former we construct the required contracting homotopy at the level of bounded measurable functions, and in the latter we show that our homotopy descends to function classes.

We include three appendices with background material. Appendix \ref{sec:monodcoeff} summarizes the properties of Monod's category of coefficient modules, which underlies the functorial approach to continuous bounded cohomology of locally compact groups. Appendix \ref{sec:symplectic} summarizes basic topological properties of symplectic Grassmannians, and Appendix \ref{sec:measuretheory} collects some basic facts from measure theory that we use throughout.

{\bf Notational conventions.} We adopt the convention that $0 \in \bbN$. For any $R\in \bbN$, we denote by $[R]$ the subset $\{0,\ldots,R\}$ of $\bbN$, and we set $[\infty] := \bbN$. The symmetric group on $k$ letters will be denoted by $\frS_k$. Given a measure space $(X, \mu)$ and $0 \leq p < \infty$, we denote by $\scL^p(X, \mu)$ (or simply by $\scL^p(X)$ when the measure $\mu$ is understood from the context) the space of $\mu$-measurable functions $f :X \to \bbR$ such that $\mu(|f|^p) < \infty$. For $p = \infty$, $\Binfty(X)$ denotes the space of bounded measurable functions on $X$. For distinction, the corresponding spaces of function classes will be denoted by $L^p(X, \mu)$ or $L^p(X)$ for $1 \leq p \leq \infty$. We will denote by $\clH_X^d$ the $d$-dimensional Hausdorff measure on a metric space $X$, omitting mention of $X$ when the space is clear from the context. 

{\bf Acknowledgements.} We would like to thank Marc Burger for suggesting the question that led to the main result of this paper and for his valuable input throughout the development of the project. We are particularly indebted to Alexis Michelat for suggesting the use of the co-area formula in the proof of \autoref{thm:symmetry_muk}. We also thank Nicol\'as Matte Bon for profitable conversations concerning the probabilistic language used in Section \ref{sec:contractibility}, and Alessandra Iozzi, Nir Lazarovich and Maria Beatrice Pozzetti for their help and encouragement. Finally, we are grateful for the hospitality of the following institutions where different parts of this work were carried out: the Technion in Haifa, MSRI in Berkeley, Paderborn University, ETH Zurich, and JLU Gie\ss en.

\section{The bounded-cohomological Quillen's method} \label{sec:abstract_stability}
\subsection{Admissible $G$-objects and associated $\Linfty$-complexes}\label{subsec:GObjects} 
A \emph{semi-simplicial object} $X_\bcdot$ in a category $\clC$ is a sequence of objects $(X_k)^\infty_{k = 0}$, together with morphisms $\delta_{i,k}:\, X_{k+1} \to X_k$  for all $k$ and $i \in [k]$, called the \emph{face maps} of the semi-simplicial object, such that 
\[
	\delta_{i,k-1} \, \circ \, \delta_{j,k} = \delta_{j-1,k-1} \, \circ \, \delta_{i,k} \qquad \mbox{whenever } i<j\text{.}
\]
If $k$ is clear from the context we will usually denote the face map $\delta_{i,k}$ simply by $\delta_{i}$. In the situations we are interested in, $\mathcal C$ will always be a concrete category. Elements of the set underlying $X_k$ will then be referred to as \emph{$k$-simplices} of the semi-simplicial object. If the set underlying $X_k$ is empty for all $k \geq n+1$, then we say that $X_\bcdot$ is \emph{$n$-dimensional} and write $\dim(X_\bcdot)=n$. 

We are going to consider semi-simplicial objects in the following category (cf. \autoref{defn:regspace}):
\begin{defn} \label{defn:regspace:Intro}
A \emph{regular $G$-space} is a standard Borel space $X$, endowed with a Borel $G$-action and a Borel probability measure $\mu$ on $X$ that is $G$-quasi-invariant in the sense that $g_\ast \mu \sim \mu$ for all $g\in G$. We denote by ${\rm Reg}_G$ the category whose objects are regular $G$-spaces and morphisms are Borel $G$-maps. 
\end{defn}
\begin{defn}\label{DefGObject} A semi-simplicial object $X_\bcdot$ in ${\rm Reg}_G$ is called a \emph{measured $G$-object}.
\end{defn}
Let $G$ be a lcsc group, let $X_\bcdot$ be an $n$-dimensional measured $G$-object (with possibly $n=\infty$), and for every $k \in [n]$, let $\mu_k$ denote the choice of a $G$-quasi-invariant probability measure on $X_k$. By \autoref{LInfty}, the pairs $(L^1(X_k),\Linfty(X_k))$ are then coefficient $G$-modules in the sense of \autoref{DefCoefficientModule}. Note that the face maps $\delta_i:\, X_{k+1} \to X_k$ of the $G$-object $X_\bcdot$ induce operators $\delta^i:\Linfty(X_k) \to \Linfty(X_{k+1})$. Since the operators $\dd^k = \sum_{i=0}^{k+1} \, (-1)^i \delta^i$ satisfy $\dd^{k+1} \circ \dd^{k} = 0$ for all $k \geq 0$, the sequence
\begin{equation} \label{eq:linfty}
	0 \to \Linfty(X_0) \xrightarrow{\dd^0} \Linfty(X_1) \xrightarrow{\dd^1} \Linfty(X_2) \xrightarrow{\dd^2} \cdots \xrightarrow{\dd^{r-2}} \Linfty(X_{n}) 
\end{equation}
is a complex of vector spaces that we call the \emph{$\Linfty$-complex associated to $X_\bcdot$}. We also denote by $\dd^{-1}: \bbR \to \Linfty(X_0)$ the inclusion of constants and refer to
\begin{equation} \label{eq:augmentedlinfty}
	0 \to \bbR \xrightarrow{\dd^{-1}} \Linfty(X_0) \xrightarrow{\dd^0} \Linfty(X_1) \xrightarrow{\dd^1} \Linfty(X_2) \xrightarrow{\dd^2} \cdots \xrightarrow{\dd^{r-2}} \Linfty(X_n) 
\end{equation}
as the \emph{augmented} $\Linfty$-complex associated to $X_\bcdot$. 

In general, the functors $L^\infty(G^{n}; -)^G$ are only well-behaved in terms of exactness as in \autoref{Exactness} along complexes of coefficient modules, i.e.\ complexes in which all differentials are weak-$*$-continuous. We thus have to require this as an additional condition on the augmented $\Linfty$-complex above.
\begin{defn} \label{DefAdmissibility}
A measured $G$-object is called \emph{admissible} if its associated augmented $\Linfty$-complex is a complex of coefficient $G$-modules, i.e.\ if each of the morphisms $\dd^k$ is dual to a morphism $L^1(X_{k+1}) \to L^1(X_k)$. 
\end{defn}
We emphasize that while the augmented $\Linfty$-complex of $X_\bcdot$ only depends on the underlying $G$-invariant measure classes on the standard Borel spaces $X_k$, the notion of admissibility depends on the specific choice of probability measures on these spaces. 

\subsection{Measured Quillen families and stability}
Our next goal is to formulate versions of Quillen's conditions (Q1)--(Q3) from the introduction in our measurable context.
 
\begin{defn} \label{def:conn+trans}
Let $\BX_\bcdot$ be an $n$-dimensional measured $G$-object and let $q \in [n]$ be an integer. We say that $\BX_{\bcdot}$ is:  
\begin{enumerate}[(i)]
	\item \emph{measurably $q$-connected} if the homology of its augmented $\Linfty$-complex \eqref{eq:augmentedlinfty} vanishes up to degree $q$. 
	\item \emph{$q$-transitive} if for all $k \in [q]$, the $G$-action on $\BX_k$ is transitive. 
	\item \emph{essentially $q$-transitive} if for all $k \in [q]$, the $G$-action on $\BX_k$ has a Borel orbit $X_k$ of full measure and such that $\delta_i(X_k) \subset X_{k-1}$ for all $i \in [k]$ whenever $k \geq 1$. 
\end{enumerate}
As a convention, we will say that a measured $G$-object $\BX_\bcdot$ is measurably $(-\infty)$-connected if $\BX_\bcdot$ is not measurably $q$-connected for any $q \in [n]$. If $n = \infty$ and $\BX_\bcdot$ is measurably $q$-connected for all $q$ we also say that $\BX_\bcdot$ is measurably $\infty$-connected.
\end{defn}

\begin{rem} If $\BX_\bcdot$ is measurably $q_1$-connected and essentially $q_2$-transitive, then we can pass from $\BX_\bcdot$ to $X_\bcdot$ to obtain a measurably $q_1$-connected and $q_2$-transitive $G$-object. Thus, assuming actual transitivity (as opposed to essential transitivity) is not a restriction.
\end{rem}

We now need a measurable notion of compatibility for families of $G$-objects. Fix $R \in \bbN \cup \{\infty\}$ and assume that $(G_r)_{r \in [R]}$ is an \emph{ascending series} of lcsc groups, i.e. a sequence of lcsc groups with inclusions $G_{r} \hookrightarrow G_{r+1}$ for all $r<R$. Furthermore assume that for every $r \in [R]$ we are given a $G_r$-object $X_{r, \bcdot}$ that is $\tau(r)$-transitive, where $\tau: [R] \to \bbN$ is a function. 

\begin{defn} \label{def:compatibility} We say that $(G_r)_{r \in [R]}$ and $(X_{r, \bcdot})_{r \in [R]}$ are \emph{measurably $\tau$-compatible} if for every $r \in [R]$ and every $q \in [\tau(r)]$ there exist point stabilizers $H_{r,q}$ for the action of $G_r$ on $X_{r,q}$ and surjective homomorphisms $H_{r,p} \twoheadrightarrow G_{r-1-q}$ with amenable kernel such that for all $q<\tau(r)$ there is an inclusion $H_{r,q+1} \hookrightarrow H_{r,q}$ making the following diagram commute:
\begin{equation} \label{eq:inclusions}
		\begin{gathered}	
		\xymatrixcolsep{1.5pc}
		\xymatrix{H_{r,q+1} \ar@{^{(}->}[r] \ar@{->>}[d] & H_{r,q} \ar@{->>}[d] \\
		G_{r-q-2} \ar@{^{(}->}[r] & G_{r-q-1}} 
		\end{gathered}
	\end{equation}
\end{defn}

\begin{rem}
Here $G_r$ is understood as the trivial group whenever $r < 0$. One can define a similar notion if $G_r$ is merely \emph{essentially} $\tau(r)$-transitive on $\bar X_{r,p}$; in this case one has to require $H_{r,p}$ to be the stabilizer of a generic point. \vspace{2pt}
\end{rem}

\begin{defn} \label{def:Quillen_family} Let $R \in \bbN \, \cup \{\infty\}$, let $\gamma\!: [R] \to \bbN \, \cup \{-\infty, \infty\}$ and $\tau\!: [R] \to \bbN$ be functions, let $(G_r)_{r \in [R]}$ be an ascending series of lcsc groups, and, for every $r \in [R]$, let $X_{r, \bcdot}$ be an admissible $G_r$-object. We say that $(G_r, X_{r, \bcdot})_{r \in [R]}$ is a \emph{measured Quillen family} of length $R$ with parameters $(\gamma, \tau)$ if
\begin{enumerate}[(QM1)]
	\item $X_{r,\bcdot}$ is measurably $\gamma(r)$-connected for every $r \in [R]$;
	\item $X_{r,\bcdot}$ is $\tau(r)$-transitive for every $r \in [R]$;
	\item $(G_r)_{r \in [R]}$ and $(X_{r, \bcdot})_{r \in [R]}$ are measurably $\tau$-compatible. \vspace{5pt}
\end{enumerate} 
\end{defn}

\begin{klem}\label{thm:abstract_stability} Assume that $(G_r, X_{r, \bcdot})_{r \in \mathbb N}$ is an infinite measured Quillen family with parameters $(\gamma, \tau)$. If $\gamma(r) \to \infty$ and $\tau(r) \to \infty$ as $r \to \infty$, then $\Hcb^\bcdot$ stabilizes along $(G_r)_{r \in \mathbb N}$.
\end{klem}
\autoref{thm:abstract_stability} will be established in Section \ref{sec:proof_abstract_stability}. From the proof one obtains an explicit stability range in terms of the parameters $\gamma$ and $\tau$. We state a quantitative version of \autoref{thm:abstract_stability} in \autoref{thm:abstract_stability_explicit} below. From this quantitative version one also obtains a version of \autoref{thm:abstract_stability} for finite-length measured Quillen families, provided the parameters grow fast enough.

\subsection{Monod's stability result revisited}
To illustrate \autoref{thm:abstract_stability}, we explain how it implies Monod's stability results from \cite{Monod-Stab}. We use this opportunity to record the corrected bounds. Let $\kay$ be a local field or skew-field (e.g. $\bbR$, $\bbC$, $\bbQ_p$ or $\bbH$), and abbreviate $G_r := {\rm GL}_r(\kay)$. We consider $G_r$ as a subgroup of $G_{r+1}$ included into the bottom-right corner. Given $r \in \mathbb N$, we define
$\bar X_{r,k} := (\mathbb P^{r-1}(\kay))^{k+1}$. The homogeneous space $\mathbb P^{r-1}(\kay)$ admits a unique $G_r$-invariant measure class; we fix a probability measure $\mu$ in this class once and for all, and define a $G$-quasi-invariant measure on $\bar X_{r,k}$ by $\mu_{r,k} := \mu^{\otimes{(k+1)}}$. We denote by $X_{r,k} \subset \bar X_{r,k}$ the full-measure subset
\[
X_{r,k} := \{(p_0, \dots, p_{k}) \in \bar X_{r,k}\mid \dim{\rm span}(p_I) = \min\{|I|, r\} \},
\]
where for a subset of indices $I \subset [k]$, we set $p_I := \{p_i \mid i \in I\}$, and ${\rm span}(p_I)$ denotes the linear span in $\kay^r$ of the lines contained in $p_I$. For $k \in [r]$, $X_{r,k}$ is the unique $G_r$-orbit of full measure of $X_{r,k}$; for $k < r$, it is a non-compact Stiefel variety. Denote by $\delta_i: \bar X_{r,k+1} \to \bar X_{r,k}$ the face map given by deleting the $i$-th coordinate; these maps restrict to $X_{r,\bcdot}$.
\begin{lem}\label{GLQuillen} The pair $(G_r, X_{r, \bcdot})_{r\geq 1}$ is a measured Quillen family.
\end{lem}
We sketch the proof: It is well known that $X_{r, \bcdot}$ is $r$-transitive. As mentioned in the introduction, this high essential multiple transitivity on flag varieties is a unique feature of general and special linear groups. Compatibility follows from the fact that $G_r$ is the Levi factor of the point stabilizer of $G_{r+1}$ acting on projective space. It remains to show admissibility and measurable connectivity; both are easy because $\mu_{r,k} $ is a \emph{product measure}, another luxury that we will not encounter in other classical families. Concerning admissibility, the fact that $\mu_{r,k}$ is a product measure implies that the maps
\[\partial_i: L^1(\BX_{r,k+1}) \to L^1(\BX_{r,k}), \quad  \partial_i \phi(p_0,\ldots,p_k) = \textstyle \int_{\BX_{r,0}} \phi(p_0,\ldots,p_{i-1},t,p_i,\ldots,p_k) \, \dd\!\mu(t)\]
are well-defined by Fubini's theorem, and they are easily seen to be pre-dual to the face operators $\delta^i$ for every $i=0,\ldots,k$. The product structure also allows us to define a
morphism $h^\bcdot:\,\Linfty(X_{r,\bcdot+1}) \to \Linfty(X_{r,\bcdot})$ (by integration over the first variable) that satisfies $\dd^{k-1} h^{k-1} + h^k \dd^k = \id$ for all $k \geq 0$. This proves that $X_{r, \bcdot}$ is $\infty$-connected, thus establishing \autoref{GLQuillen}, and thereby stability of the family $(G_r)_{r \geq 1}$.

\begin{rem}[Stability range for $\GL_r$]\label{MonodRange}
Let us compute the stability range using the quantitative version of \autoref{thm:abstract_stability} in \autoref{thm:abstract_stability_explicit} below: We have $\gamma(r) = \infty$, $\tau(r) = r$ and $q_0 = 2$, hence $\widetilde{\gamma}(q,r) = \infty$ and $\widetilde{\tau}(q,r)  = r-2q+4$ with the notation of \autoref{thm:abstract_stability_explicit}. We thus get isomorphisms $\Hcb^q(\GL_{r+1}(\kay)) \to \Hcb^q(\GL_{r}(\kay))$ for all $q \geq 2$ and $r \geq 2q-3$, and an injection for $r = 2q-4$. \end{rem}

\begin{rem}[Stability for $\SL_r$ and its range]\label{MonodRange2}
Observe that the sequence $(\SL_r(\kay),X_{r,\bcdot})_r$ does not satisfy the measurable compatibility axiom (QM3) from \autoref{def:Quillen_family}, and hence, is not a measured Quillen family. To obtain stability for $({\rm SL}_r(\kay))_{r \in \mathbb N}$, we need to consider for any $r \geq 2$ the finite-length measured Quillen familiy 
\[\big(({\rm GL}_0(\kay), X_{0, \bcdot}), ({\rm GL}_1(\kay), X_{1, \bcdot}), \dots, ({\rm GL}_{r}(\kay), X_{r, \bcdot}), ({\rm SL}_{r+1}(\kay), X_{r+1, \bcdot})\big),\]
Note that in general $\SL_{r+1}(\kay)$ is only $r$-transitive, rather than $(r+1)$-transitive on $X_{r+1, \bcdot}$. One deduces from this and \autoref{thm:abstract_stability_explicit} that $\Hcb^q(\SL_{r+1}(\kay)) \to \Hcb^{q}(\GL_{r}(\kay))$ is an isomorphism whenever $r \geq \max\{2q-3, q+1\}$ and an injection if $r = \max\{2q-4, q\}$. Combining this with the result for $\GL_r(\kay)$, we deduce that there exists an isomorphism $\Hcb^q(\SL_{r+1}(\kay)) \cong \Hcb^q({\rm SL}_r(\kay))$ for $r \geq \max\{2q-2,q+2\}$. 

In the case $\kay = \bbC$, it is possible to obtain a better stability range. As pointed out in Corollary 8.5.5 of \cite{Monod-Book}, the inclusion $\SL_r(\bbC) \hookrightarrow \GL_r(\bbC)$ induces an isomorphism in continuous bounded cohomology. Hence, by \autoref{MonodRange}, the map $\Hcb^q(\SL_{r+1}(\bbC)) \to \Hcb^q(\SL_r(\bbC))$ is an isomorphism for all $q \geq 2$ and $r \geq 2q-3$, and an injection for $r = 2q-4$.
\end{rem}

\begin{rem}[Corollaries in degree three for $\SL_r$]
The two previous remarks fix the stability ranges given in \cite{Monod-Vanish}. While the correct bounds are in general worse than claimed in loc.\ cit., we still get injections
\[
\Hcb^3(\SL_r(\bbR)) \hookrightarrow \Hcb^3(\GL_2(\bbR)) \qand \Hcb^3(\SL_r(\bbC)) \hookrightarrow \Hcb^3(\SL_2(\bbC))
\]
for all $r \geq 3$. They have as corollaries that $\Hcb^3(\SL_r(\bbR)) = 0$ (see Theorem 1.2 in \cite{Monod-Stab}) and that $\dim \Hcb^3(\SL_r(\bbC)) \leq 1$ for all $r \in \bbN$, respectively. In particular, the latter statement is used in \cite{BBI} to conclude that the comparison map $\Hcb^3(\SL_r(\bbC)) \to \Hc^3(\SL_r(\bbC))$ is an isomorphism for every $r \in \bbN$.
\end{rem}

\subsection{Summary of results in the symplectic case}

The remainder of this article, with the exception of Section \ref{sec:proof_abstract_stability}, is devoted to the proof of the following result, which implies \autoref{ThmIntro} in view of \autoref{thm:abstract_stability}.

\begin{thm}\label{MainThmConvenient} For $\kay \in \{\bbR, \bbC\}$ there exists a measured Quillen family $(\Sp_{2r}(\kay), X_{r, \bcdot})$ with parameters 
\[
\gamma(r) = \sup \{q \in \mathbb N \mid 2^q + \lceil (q+1)/2\rceil \leq r\}
\]
and $\tau(r) = r-1$.
\end{thm}

\begin{rem} To determine the stability range implied by \autoref{MainThmConvenient} and \autoref{thm:abstract_stability_explicit}, let us abbreviate $\widehat{\gamma}(q) := 2^q + \lceil (q+1)/2\rceil $. We claim that the map $\Hcb^q({\rm Sp}_{2r+2}(\kay)) \to \Hcb^q({\rm Sp}_{2r}(\kay))$ is an isomorphism if $r \geq \widehat{\gamma}(q) -1$. 

As explained in \autoref{IC}, we may take $q_0 = 2$ in \autoref{thm:abstract_stability_explicit}. Fix $q \geq q_0 + 1 = 3$, and observe that $\widetilde{\tau}(q,r) = r-2q+3$. Therefore, we have
\begin{equation} \label{eq:tausp}
	\widetilde{\tau}(q,r) -1 \geq 0 \quad \mbox{ if and only if } \quad r \geq 2q-2.
\end{equation}
Note also that $\widehat{\gamma}$ is strictly increasing, and hence 
\[
	\gamma\big(r+1-2(q-j)\big)-j \geq 0 \quad \mbox{ if and only if } \quad r \geq 2q - 1 + (\widehat{\gamma}(j)-2j). 
\]
for all $j \in \{3,\ldots,q\}$. Thus, 
\begin{equation}  \label{eq:gammasp}
\widetilde{\gamma}(q,r) \geq 0 \quad \mbox{ if and only if } \quad r \geq 2q-1 + \max_{j=3}^q \, (\widehat{\gamma}(j)-2j).
\end{equation}
Since the function $\{3,\ldots,q\} \ni j \mapsto \widehat{\gamma}(j)-2j$ is, too, strictly increasing, we have $\max_{j=3}^q \, (\widehat{\gamma}(j)-2j) = \widehat{\gamma}(q)-2q$. As a result of the equivalences \eqref{eq:tausp} and \eqref{eq:gammasp}, we deduce from \autoref{thm:abstract_stability_explicit} that the desired isomorphism in continuous bounded cohomology holds if $r \geq \max\{2q-2,\widehat{\gamma}(q)-1\}=\widehat{\gamma}(q)-1$, completing the proof of the claim. 

We point out that because of the rapid growth of the function $\widehat{\gamma}$, we are not able to profit from the injectivity statement in \autoref{thm:abstract_stability_explicit}, as the corresponding range would be also given by the function $q \mapsto \max\{2q-3,\widehat{\gamma}(q)-1\} = \widehat{\gamma}(q)-1$.
\end{rem}

\section{The symplectic Stiefel complexes} \label{sec:stiefel}

\subsection{Symplectic Stiefel varieties}
Throughout this section let $(V, \omega)$ be a symplectic vector space over $\kay \in \{\bbR, \bbC\}$ of finite positive dimension $n=2r$, and abbreviate $G:=\Sp(V,\omega)$. We refer the reader to Appendix \ref{sec:symplectic} for background and notation concerning symplectic vector spaces. We recall from there that for $l \in [r]$ the \emph{symplectic Grassmannian} of type $l$ as defined by
\[\clG_l(V, \omega) := \{W \in {\rm Gr}_{l+1}(V) \mid W \text{ is isotropic}\}\]
is a homogeneous space under $G$. We also recall from \autoref{lem:symp_grass} that the Chabauty--Fell topology on these spaces coincides with the quotient topology, and that with respect to this topology the symplectic Grassmanians are compact. Since $(V, \omega)$ will be fixed throughout this section, we will abbreviate $\clG_l := \clG_l(V, \omega)$. We will also abbreviate $\clP := \clG_0 = \bbP(V)$ and $\clG := \bigsqcup_l \clG_l$. 

For $0 \leq k \leq r-1$, we define
\begin{align*}
\BX_k &:= \{(p_0, \dots, p_k) \in \clP^{k+1}  \mid \Span(p_0, \dots, p_k) \in \clG\}; \mbox{ and}\\
X_k &:= \{(p_0, \dots, p_k) \in \BX_k \mid {\rm span}(p_0, \dots, p_k) \in \mathcal G_k\}. 
\end{align*}
The space $\BX_k$ is a closed (hence compact) subspace of the product space $\clP^{k+1}$ for every $k$. On the other hand, $X_k$ is an open, dense subset of $\BX_k$, being Zariski open with respect to the natural variety structure on $\clP^{k+1}$. Note that both $\BX_k$ and $X_k$ admit $\frS_{k+1}$-actions by permuting the coordinates, and 
continuous $G$-actions by restricting accordingly the diagonal action on $\clP^{k+1}$. Also, the map 
\[
	\Span: X_k \to \clG_k, \quad (p_0,\ldots,p_k) \mapsto \Span(p_0,\ldots,p_k)
\]
is a continuous (\autoref{SpanCts}), $G$-equivariant, $\frS_{k+1}$-invariant surjection. In fact, $X_k$ is a fiber bundle over $\clG_k$ via this map. In analogy to the classical objects, we define:
\begin{defn} \label{def:Stiefel_variety}
The homogeneous space $X_k$ is called the \emph{$k$-Stiefel variety} of $(V,\omega)$, and $\BX_k$ is called the \emph{compactified} $k$-Stiefel variety of $(V,\omega)$.  
\end{defn}

For $i \in [k]$, the face maps $\delta_i: \clP^{k+1} \to \clP^k$ that delete the $i$-th component in a $(k+1)$-tuple restrict to continuous, $G$-equivariant maps $\BX_{k+1} \to \BX_k$ and $X_{k+1} \to X_k$ that we also denote by $\delta_i$. Thus, $\BX_\bcdot$ and $X_\bcdot$ are semi-simplicial objects in the category of topological $G$-spaces, i.e.\ of topological spaces with a continuous $G$-action by homeomorphisms. 

Note that, by \autoref{thm:transitivity}, the group $G$ acts transitively on $X_k$ for all $k \in [r-1]$, whence there exists a unique $G$-invariant measure class on $X_k$. If we choose for every $k \in [r-1]$ an arbitrary probability measure on $X_k$ in the $G$-invariant measure class, then $X_\bcdot$ together with these measures will always be an $(r-1)$-transitive measured $G$-object. However, its admissibility will depend on the choice of probability measures on $X_k$. Thus, our next task will be to construct explicit quasi-invariant probability measures on each $X_k$. The remainder of the section will then be devoted to the proof of admissibility of the resulting $G$-object.

\subsection{Perpendicular measures} 
We keep the notation of the previous subsection. We are going to introduce a class of probability measures on the projective space $\mathcal P = \mathbb P(V)$. These measures will be used in the next subsection to construct explicit quasi-invariant probability measures on the Stiefel varieties $X_k$ and will also feature prominently in our proof of measurable connectivity of the Stiefel $G$-object in the next section.

From now on we fix an inner product $\langle \cdot,\cdot \rangle$ on $V$ and denote by $K(V)$ the maximal compact subgroup of $\GL(V)$ that preserves $\langle \cdot, \cdot \rangle$. We denote the induced norm and metric by $\|\cdot\|$ and $d_V$ respectively.

Given a non-zero linear subspace $W \subset V$ we denote by $B_W$ the intersection of $W$ with the unit ball with respect to $d_V$. We then denote by $\clL_W$ the unique multiple of the Lebesgue measure on $W$ normalized to $\clL_W(B_W) = 1$. Finally we denote by $\bbP(W)$ the associated projective space and by $\pi_W: W\smallsetminus \{0\} \to \bbP(W)$ the associated quotient map. We then define a probability measure $\lambda_W$ on $\bbP(W)$ by
\[
	\int_{\bbP(W)} f \, \dd\!\lambda_W := \int_{B_W} \pi_W^\ast f \, \dd\!\clL_W, \quad f \in C(\bbP(W)).
\]
We will consider $\lambda_W$ as a measure on $\mathcal P = \bbP(V)$ supported on the compact subset $\bbP(W)$; this defines a map
\[
\lambda: {\rm Gr}(V)\setminus\{0\} \to {\rm Prob}(\mathcal P), \quad W \mapsto \lambda_W
\]
where ${\rm Gr}(V)$ is the union of the Grassmannians of $V$, equipped with the Chabauty topology, as in \eqref{Grassmann}.

\begin{rem} \label{rem:GptheoreticCharacterization} In group theoretic terms, we can characterize the measures $\lambda_W$ as follows. If $W\subset V$ is a linear subspace and $W^\perp$ denotes its orthogonal complement with respect to $\langle \cdot, \cdot \rangle$, then $V = W \oplus W^\perp$. Hence, we can extend every $g \in \GL(W)$ to an automorphism of $V$ by acting identically on $W^\perp$, and thereby define an embedding $\GL(W) \hookrightarrow \GL(V)$. Under this embedding, the group $K(W) := \GL(W) \cap K(V)$ is the unique maximal compact subgroup of $\GL(W)$ preserving the inner product $\langle \cdot,\cdot \rangle\!\mid_{W \times W}$. Moreover, if $P(V,W) := {\rm Stab}_{G(V)}(W)$ denotes the set-stabilizer of $W$ in $\GL(V)$, then 
\[
	{\rm Stab}_{K(V)}(W) = P(V,W) \cap K(V) = K(W) \times K(W^\perp).
\]
Now $\bbP(W)$ is a homogeneous space of both $\GL(W)$ and $K(W)$, and $\lambda_W$ is invariant under $K(W)$ by construction. It is thus the unique $K(W)$-invariant probability measure supported on $\bbP(W)$, and its measure class $[\lambda_W]$ is the unique $\GL(W)$-invariant measure class on $\bbP(W)$, and also invariant under $P(V,W)$. 
\end{rem}

Using \autoref{rem:GptheoreticCharacterization}, we can establish the following properties of the map $\lambda$.
\begin{lem}\label{KInvariantMeasuresQuasiEquivariant} 
The map $\lambda$ is continuous with respect to the weak-$*$-topology on ${\rm Prob}(\mathcal P)$. Moreover, it is $K(V)$-equivariant and ${\rm GL}(V)$-\emph{quasi-equivariant}, i.e.\ for every $g\in \GL(V)$ and every $W \in \Gr(V)$, the probability measures $g_*\lambda_W$ and $\lambda_{gW}$ are mutually absolutely continuous.  
\end{lem}
\begin{proof} We first establish the $K(V)$-equivariance and $\GL(V)$-quasi-equivariance of $\lambda$. Let $W \in \Gr(V)$. For $k \in K(V)$, the probability measure $k_*\lambda_W$ is supported on $kW$ and invariant under $kK(W)k^{-1} = K(kW)$, hence coincides with $\lambda_{kW}$ by uniqueness. This shows that $\lambda$ is $K(V)$-equivariant. Now, if $g \in \GL(V)$, there exists $k \in K(V)$ such that $gW = kW$ and hence $gk^{-1} \in P(V, kW)$, so $g_* \lambda_W = (gk^{-1})_* k_*\lambda_W = (gk^{-1})_*\lambda_{kW}$. Since $\lambda_{kW}$ is quasi-invariant under $P(V, kW)$ the latter is
equivalent to $\lambda_{kW} = \lambda_{gW}$, showing that the action is $\GL(V)$-quasi-equivariant. 

For the continuity of $\lambda$, in view of \eqref{Grassmann}, it suffices to establish continuity for each of the restrictions $\lambda_i := \lambda|_{\Gr_{i+1}(V)}$, where $i \in [n-1]$. For this purpose, fix $i$ and let $(W_n)_n$ be a sequence in $\Gr_{i+1}(V)$ converging to a subspace $W \in \Gr_{i+1}(V)$. Since $K(V)$ act transitively on $\Gr_{i+1}(V)$ we find $k_n \in K(V)$ with $W_n = k_n W$, and we may assume by passing to a subsequence that $k_n$ converges to some $k \in K(V)$. By continuity of the action, we then have $kW = \lim k_nW = W$, i.e. $k \in {\rm Stab}_{K(V)}(W) = K(W) \times K(W^\perp)$, and hence $k_*\lambda_W = \lambda_W$. We deduce that
\[
 \lim_{n \to \infty} \lambda_{W_n} = \lim_{n \to \infty} \lambda_{k_nW} = \lim_{n \to \infty}(k_n)_*\lambda_W = k_*\lambda_W = \lambda_W,
\]
where the second-to-last equality follows from the continuity of the $\GL(V)$-action on ${\rm Prob}(\clP)$. 
\end{proof} 
In probabilistic language, $\lambda_W$ is the distribution of a \emph{random point} $p \in \clP$ subject to the condition that $p \in \bbP(W)$. Similarly, the following definition describes the distribution of a ``random symplectic perpendicular'', i.e.\ a random point that is perpendicular with respect to $\omega$ to a given finite set of points in $\clP$. As in Subsection \ref{SympComp}, given a subset $S \subset V$, we write $S^\omega$ for the symplectic complement of $S$ with respect to $\omega$. 
\begin{defn}\label{DefNu}
Given an integer $ k \in [2r-2]$, we define  
\[
 \nu^k: \clP^{k+1} \to \Prob(\clP), \quad (p_0, \ldots, p_k) \mapsto \nu^k_{(p_0,\ldots,p_k)} := \lambda_{\Span(p_0, \dots, p_k)^\omega}.
\]
The measure $\nu^k_{(p_0,\ldots,p_k)}$ is called \emph{perpendicular measure to} $(p_0,\ldots,p_k)$. 
\end{defn}
The restriction of the definition to $k \leq 2r-2$ is important, since it may happen for  $k > 2r-2$ that $\Span(p_0,\ldots,p_k) = V$ and hence $\Span(p_0,\ldots,p_k)^\omega= \{0\} $. The following proposition summarizes basic properties of the map $\nu^k$. Here we recall our abbreviation $G := {\rm Sp}(V, \omega)$, and choose a maximal compact subgroup of $G$ by setting $K:= G \cap K(V)$. 
\begin{prop} \label{thm:nu_Borel}
For every $k \in [2r-2]$, the map $\nu^k$ is an $\frS_{k+1}$-invariant, $K$-equivariant and $G$-quasi-equivariant Borel map. For every $l \in [k+1]$, it is continuous on the subset $\clP^{k+1}_l := \{(p_0, \dots, p_k) \mid \dim \,{\rm span}\{p_0, \dots, p_k\} = l\} \subset \clP^{k+1}$.
\end{prop}
\begin{proof} Observe that $\nu^k$ can be written as the composition
\begin{equation}\label{nukDec}
\clP^{k+1} \xrightarrow{\Span} \Gr(V)\setminus\{V\} \xrightarrow{(-)^\omega} \Gr(V)\setminus\{0\} \xrightarrow{\lambda} \Prob(\clP).
\end{equation}
Now, the map $\Span\colon \clP^{k+1} \to \Gr(V)$ is Borel (and continuous on each $\clP^{k+1}_l$) by \autoref{SpanCts}, and clearly $\frS_{k+1}$-invariant and $\GL(V)$-equivariant; the symplectic polarity $(-)^\omega: {\rm Gr}(V) \to {\rm Gr}(V)$ is $G$-equivariant and continuous by \autoref{thm:omega_Borel}; and $\lambda$ is continuous, $K$-invariant and $G$-equivariant by \autoref{KInvariantMeasuresQuasiEquivariant}.
\end{proof}
\begin{rem} \label{rem:random_perp}
From now on, unless specificity is necessary, we will avoid making the index $k$ explicit and refer to all the maps $\nu^k$ simply as $\nu$ and write $\nu_{p_0, \dots, p_k} := \nu^k_{(p_0, \dots, p_k)}$. Note that the latter only depends on the set $\{p_0, \dots, p_k\}$. By construction,  $\nu_{p_0,\ldots,p_k}$ is supported on the subspace $\bbP(\Span(p_0,\ldots,p_k)^\omega)$ of $\clP$ and thus 
\begin{equation}\label{RandomPerpendicular}
\omega(p, p_j) = 0 \quad \text{for all }j\in \{0,\ldots,k\} \text{ and } \nu_{p_0, \dots, p_k}\text{-almost all }p\in \mathcal P.
\end{equation}
Because of this, we refer to a random variable distributed according to the measure $\nu_{p_0, \dots, p_n}$ as a \emph{random perpendicular} to $p_0, \dots, p_n$. The following lemma ensures that generically a random perpendicular to $(p_0, \dots, p_k)$ is linearly independent of $(p_0, \dots, p_k)$.
\end{rem}
\begin{lem} \label{thm:pre_Xk_full_meas} 
For all $k \in [r-2]$ and all $(p_0, \dots, p_k) \in X_k$, the Borel set $\{p \in \clP \mid (p_0, \dots, p_k, p) \in X_{k+1}\}$ has full $\nu_{p_0, \dots, p_k}$-measure. 
\end{lem}
\begin{proof} Let $W := \Span(p_0, \dots, p_k)$ and let $A := \{p \in \clP \mid (p_0, \dots, p_k, p) \in X_{k+1}\}$. Then we have that $W \subsetneq W^\omega$, since equality holds only in the Lagrangian case. Moreover, the equality $A = \bbP(W^\omega) \smallsetminus \bbP(W)$ holds, and as a positive-codimension, closed embedded submanifold of $\bbP(W^\omega)$, the projective subspace $\bbP(W)$ is Borel and a Lebesgue null set. Hence, 
\[
	\nu_{p_0, \dots, p_k}(A) = \lambda_{W^\omega}(\bbP(W^\omega) \smallsetminus \bbP(W)) = 1. \vspace{-8pt}
\]
\end{proof}

\subsection{Measures for symplectic Stiefel complexes}
We are going to define recursively probability measures $\mu_k$ on the compactified Stiefel varieties $\BX_k$ for all $ k \in [r-1]$, using perpendicular measures. For $k = 0$, we have $\BX_0 = X_0 = \bbP(V) = \clP$ and we define $\mu_0 := \lambda_V$. Given $k=1, \dots, r-1$, we set
\[
\int_{\BX_{k+1}} f(\mathbf{p},p_{k+1}) \, \dd\!\mu_{k+1}(\mathbf {p},p_{k+1}) := \int_{\bar X_k} \int_{\clP} f(\mathbf{p}, p_{k+1}) \; d\nu_{\mathbf{p}}(p_{k+1}) \, \dd\!\mu_{k}(\mathbf{p}) \quad (f \in C(\BX_{k+1})).
\]
To see that this is well-defined, we observe that the inner integral defines a Borel measurable function on $\BX_k$ by \autoref{GeneralMeasurability} and \autoref{thm:nu_Borel}.
\begin{prop} \label{thm:Xk_full_meas} For every $k \in [r-1]$, the measure $\mu_k$ restricts to a $G$-quasi-invariant probability measure on $X_k$.
\end{prop}
\begin{proof} Since $\mu_0$ is $G$-quasi-invariant and $\nu$ is $G$-quasi-equivariant by \autoref{thm:nu_Borel}, it follows by induction that the measures $\mu_1, \dots, \mu_{r-1}$ are all $G$-quasi-invariant. 

To show that $\mu_k(X_k) = 1$, we argue by induction on $k$. For $k = 0$, the lemma follows from $\BX_0  = X_0$. For $k + 1$, we have, using the induction hypothesis that $\mu_{k}(X_{k}) = 1$, that
\begin{align*}
	\mu_{k+1}(X_{k+1}) &= \int_{X_{k}} \nu_{p_0, \dots, p_k}\big(\{p_{k+1} \in \clP \mid  (p_0, \dots, p_{k},p_{k+1}) \in X_{k+1}\}\big) \dd\!\mu_{k}(p_0, \dots, p_{k}).
\end{align*}
By \autoref{thm:pre_Xk_full_meas}, the integrand equals to $1$ for every fixed $(p_0, \dots, p_{k}) \in X_{k}$, and $\mu_k(X_k) = \mu_{k-1}(X_{k-1}) = 1$ as claimed.
\end{proof}
In the sequel we will always consider $X_k$ as a probability space with respect to the measure $\mu_k$. By \autoref{thm:Xk_full_meas}, the probability measure $\mu_k$ represents the canonical measure class on $X_k$. We can summarize our results so far as follows.
\begin{cor} The pair $(X_\bcdot, \delta_\bullet)$ is a measured $G$-object in the sense of \autoref{DefGObject}.\qed
\end{cor}
\begin{defn} The measured $G$-object $X_\bcdot$ is called the \emph{symplectic Stiefel complex} associated to the symplectic vector space $(V, \omega)$.
\end{defn}
Since $\dim V = 2r$, then the symplectic Stiefel complex associated with $(V, \omega)$ is $(r-1)$-transitive by \autoref{thm:transitivity}.

\subsection{Admissibility of Stiefel complexes}
The purpose of this subsection is to establish admissibility of the symplectic Stiefel complexes:
\begin{prop} \label{thm:admissibility} For every symplectic vector space $(V, \omega)$, the associated Stiefel complex $(X_\bcdot, \delta_\bcdot)$ is admissible.
\end{prop}
We rely on the symmetry properties of the measures $\mu_k$ stated the next lemma. 
\begin{lem} \label{thm:symmetry_muk} The measure $\mu_k$ is symmetric, i.e. invariant under the action of the symmetric group $\frS_{k+1}$ on $X_k$ by permuting the coordinates.
\end{lem}
\autoref{thm:symmetry_muk} is a consequence of the \emph{co-area formula} from geometric measure theory, and we delay its proof to the next subsection.
\begin{proof}[Proof of \autoref{thm:admissibility} modulo \autoref{thm:symmetry_muk}]
In order to prove the admissibility of the Stiefel complex $X_\bcdot$, we show that for every $i \in [k]$, the maps $\delta^i: \Linfty(X_k,\mu_k) \to \Linfty(X_{k+1},\mu_k)$ induced by the face maps $\delta_i$ are weak-$\ast$ continuous.  Since $L^\infty(X_k, \mu_k) = L^\infty(\BX_k, \mu_k)$ we can as well work with the compactified Stiefel varieties $\BX_k$. Define bounded linear operators $\partial_i: L^1(\BX_{k+1},\mu_{k+1}) \to L^1(\BX_k,\mu_k)$ by
\[
	\partial_i \phi \, (p_0,\ldots,p_k) := \int_{\clP} \phi(p_0,\ldots,p_{i-1},p,p_i,\ldots,p_k) \, \dd\!\nu_{p_0,\ldots,p_k}(p) = \nu_{p_0,\ldots,p_k}\big(\phi \,\circ\, \tau_i(p_0,\ldots,p_k,\bcdot)\big) \\
\]
where $\tau_i$ is the cycle $(k,k-1,\ldots,i) \in \frS_{k+2}$. It is well defined: Indeed, $\partial_i\phi$ is measurable by \autoref{thm:nu_Borel} and \autoref{GeneralMeasurability}, and
\begin{align*}
	\|\partial_i \phi\|_1 &= \int_{\BX_k} | \partial_i\phi| \, \dd\!\mu_k \leq \int_{\BX_k} \int_{\clP} |\phi \,\circ\, \tau_i(p_0,\ldots,p_k,p)| \, \dd\!\nu_{p_0,\ldots,p_k}(p) \, \dd\!\mu_k(p_0,\ldots,p_k) \\
	&= \int_{\BX_{k+1}}  |\phi \,\circ\, \tau_i| \, \dd\!\mu_{k+1} = \int_{\BX_{k+1}}  |\phi| \, \dd\!\mu_{k+1} = \|\phi\|_1 < \infty,
\end{align*}
where the second-to-last equality is \autoref{thm:symmetry_muk}; the descent of the map $\partial_i:\scL^1(\BX_{k+1},\mu_{k+1}) \to \scL^1(\BX_k,\mu_k)$ to one at the level of $L^1$ follows from a similar computation. 

We show that $\partial_i$ is pre-dual to $\delta^i$. Indeed, if $f \in \Linfty(\BX_k,\mu_k)$, $\phi \in L^1(\BX_{k+1},\mu_{k+1})$, and $(- \! \mid \! -)$ denotes the dual pairing, then 
\begin{align*}
(f \mid \partial_i \phi) &= \int_{\BX_k} f(p_0,\ldots,p_k) \cdot \left(\int_{\clP} \phi(p_0,\ldots,p_{i-1},p,p_i,\ldots,p_k) \, \dd\!\nu_{p_0,\ldots,p_k}(p)\right) \, \dd\!\mu_k(p_0,\ldots,p_k) \\
&= \int_{\BX_k} \int_{\clP} \big((f\, \circ \, \delta_i) \cdot \phi\big)\circ \, \tau_i(p_0,\ldots,p_k,p) \, \dd\!\nu_{p_0,\ldots,p_k}(p) \, \dd\!\mu_k(p_0,\ldots,p_k) \\
&= \int_{\BX_{k+1}} \big(\delta^i f \cdot \phi\big)\circ \, \tau_i \, \dd\!\mu_{k+1} = \int_{\BX_{k+1}} \delta^i f \cdot \phi \, \dd\!\mu_{k+1} = (\delta^i f \mid \phi),
\end{align*}
where the second-to-last equality holds due to \autoref{thm:symmetry_muk} once again. In conclusion, we have that the coboundary $\dd: \Linfty(X_k,\mu_k) \to \Linfty(X_{k+1},\mu_k)$ is dual to the morphism $\partial = \sum_{i=0}^k (-1)^i \partial_i$, hence weak-$\ast$ continuous, which establishes admissibility of the Stiefel complex.
\end{proof}

\subsection{Proof of \autoref{thm:symmetry_muk}} We now derive the missing \autoref{thm:symmetry_muk} from the co-area formula. We fix a subspace $W$ of our symplectic vector space $(V, \omega)$ of dimension $d \in \{1, \dots, 2r\}$. Let ${\rm Rad}(W) := W \cap W^\omega$ denote the radical of $\omega\!\mid_{W \times W}$. It is a positive-codimension subspace of $W$ unless $W$ is an isotropic subspace of $V$, in which case ${\rm Rad}(W) = W$. For every $p \in \bbP(W)$ we have $p \subset W \cap p^\omega$, hence $W \cap p^\omega \neq \{0\}$. Thus, we may define maps
\begin{equation*} 
	s_W\colon\bbP(W) \to \Gr(V) \setminus\{0\}, \quad p \mapsto W \cap p^\omega \quad \text{and} \quad t_W \colon\bbP(W) \longrightarrow \Prob(\clP), \quad p \mapsto \lambda_{W \cap p^\omega}.
\end{equation*}
Note that the map $s_W$  is continuous by \autoref{IntersectionCts} and \autoref{thm:omega_Borel}, and hence $t_W = \lambda \circ s_W$ is continuous by \autoref{KInvariantMeasuresQuasiEquivariant}. Consider now the subspace $X_W \subset \bbP(W)^2$ given by
\[
X_W := \big\{(p,q) \in \bbP(W)^2 \mid \Span(p,q) \in \clG \big\}
\]
We define two probability measures $\mu_1, \mu_2$ on $X_W$ by
\[
\mu_1(f):= \int_\clP \int_\clP f(p,q) \, \dd\!\lambda_{W \cap p^\omega}(q) \, \dd\!\lambda_W(p) \, \mbox{ and } \, \mu_2(f):=  \int_\clP \int_\clP f(p,q) \, \dd\!\lambda_{W \cap q^\omega}(p) \, \dd\!\lambda_W(q).
\]
To see that these are well-defined, one has to check that the inner integrals are Borel measurable functions of their corresponding free variables. For this we observe that the map $p \mapsto W \cap p^\omega$ is continuous by \autoref{IntersectionCts} and \autoref{thm:omega_Borel}, and hence the inner integrals are measurable by \autoref{thm:nu_Borel} and \autoref{GeneralMeasurability}. We are going to show:
\begin{lem}\label{FubiniLemma}   The two probability measures $\mu_1$ and $\mu_2$ on $X_W$ coincide.
\end{lem} 
Note that if $W \subset V$ is isotropic, then $X_W = \bbP(W)^2$ and $\lambda_{W \cap p^\omega} = \lambda_W$, hence \autoref{FubiniLemma} reduces to Fubini's theorem. The non-trivial case of \autoref{FubiniLemma} is thus the case where $W$ is non-isotropic, and hence ${\rm Rad}(W) \subsetneq W$.  In this case we are going to use the following Fubini-type theorem, which follows from the co-area formula. Here we denote by $\pi_1,\pi_2: W \times W \to W$ the canonical projections, and given a subset $E \subset W \times W$ and $w \in W$ we denote by $E^{(j)}_w$ the $\pi_j$-fiber or $E$ over $w$.
\begin{lem}\label{AbstractFubini} Let $W^0 \subset W$ be an open subset, and let $E$ be a smooth codimension-one submanifold of $W^0 \times W^0$ such that $E$ projects surjectively onto both factors, and such that for all $(v,w) \in E$, the  tangent space $T_{(v, w)}E \subset W \times W$ projects surjectively onto both factors. Then for all $h \in C_c(E)$ one has
\[
\int_{W^0} \int_{E^{(1)}_{w}} h(v, w)\dd\mathcal H^{d-1}(v) \dd\mathcal H^{d}(w) = \int_{W^0} \int_{E^{(2)}_{v}} h(v, w)\dd\mathcal H^{d-1}(w) \dd\mathcal H^{d}(v)
\]
\end{lem}
\begin{proof} For all $(v,w) \in W \times W$, there exist bases of $T_{(v,w)}(W \times W)$, $T_{v}W$ and $T_wW$ such that $D\pi_j(v,w) = ({\bf 1}_d \, 0_{d})$. The assumption on $E$ implies that for all $(v, w) \in E$ one has $D(\pi_j|_E) = ({\bf 1}_d \, 0_{d-1})$, and thus, in the notation of \cite[Sec. 3.2]{Federer} we have the generalized Jacobian $J_d(\pi_j|_E) = \|\wedge^d D(\pi_j|_E)(v, w))\| = 1$. Then the co-area formula \cite[Thm. 3.2.22]{Federer} implies that both sides are equal to the integral of $h$ against $\mathcal H^{2d-1}$ on $E$.
\end{proof}
\begin{proof}[Proof of \autoref{FubiniLemma}] We may assume that $W$ is non-isotropic. This implies that $W^0 := W \smallsetminus {\rm Rad}(W)$ is a dense open subset of $W$, and 
in particular $(W^0, \mathcal H^d) \cong (W, \mathcal L_W)$ as measure spaces. We now consider
\[
E := \{(v,w) \in W^0 \times W^0 \mid \omega(v,w) = 0\} \subset W^0 \times W^0.
\]
If $v \in W^0$, then $v \not \in W^\omega$, hence $\omega(v,w) \neq 0$ for $w$ in a dense subset of $W$; in particular, we can choose $w \in W^0$. This shows that $\pi_1(E) =W^0$ and similary $\pi_2(E) = W^0$.

If $(v,w) \in E$, then $D\omega(v,w)(X, Y) = \omega(v,Y) - \omega(w,X)$, and since $v,w \not \in W^\omega$, neither this map nor either of its summands is zero. This implies, firstly, that $D\omega(v,w)$ has full rank; hence, $E \subset W\times W$ is a smooth codimension-one submanifold. Secondly, for all $(v,w) \in E$, the tangent space
\[
T_{(v,w)}E = \{(X,Y) \in W \times W \mid \omega(v,Y) = \omega(w,X)\}
\]
projects surjectively onto both factors. We deduce that Lemma \ref{AbstractFubini} applies. Given $(v,w) \in E$, the coresponding fibers are given by
\[
E^{(1)}_{w} = (W \cap w^\omega) \setminus W^\omega \qand E^{(2)}_{v} = (W \cap v^\omega) \setminus W^\omega
\]
Since $E$ has dimension $2d-1$, these fibers are $(d-1)$-dimensional. On the other hand, since $v,w \not \in W^\omega$, the vector spaces $W \cap w^\omega$ and $W \cap v^\omega$ are proper linear subspaces of $W$, hence of dimension $d-1$. It follows that $W^\omega$ intersects these vector spaces in positive codimension, and hence
\[
(E^{(1)}_{w}, \mathcal H^{d-1}) = (W \cap w^\omega, \mathcal L_{W \cap w^\omega}) \qand (E^{(1)}_{v}, \mathcal H^{d-1}) = (W \cap v^\omega, \mathcal L_{W \cap v^\omega}).
\]
We conclude that for all $h \in C_c(E)$ one has
\[\int_{W} \int_{W \cap w^\omega} h(v, w)\dd\mathcal L_{W \cap w^\omega}(v) \dd\mathcal L_W(w) = \int_{W} \int_{W \cap v^\omega} h(v, w)\dd\mathcal H_{W 
\cap v^\omega}(w) \dd\mathcal L_W(v).\]
If we denote by $\pi: W \setminus \{0\} \to\bbP(W)$ the canonical projection and choose $h(v,w) := \chi_{B_W}(v) \chi_{B _W}(w) f(\pi(v), \pi(w))$ for some $f \in C(X_W)$, then unravelling definitions, we see that the left-hand side equals $\mu_2(f)$ and the right-hand side equals $\mu_1(f)$.
\end{proof}
We can now finish the proof of \autoref{thm:symmetry_muk}, and thereby of \autoref{thm:admissibility}.
\begin{proof}[Proof of \autoref{thm:symmetry_muk}] We argue by induction on $k$. For $k=0$, there is nothing to show; for $k=1$, we can apply \autoref{FubiniLemma} with $W=V$ to obtain for every $f \in C(\BX_1)$ that
\begin{align*}
	\int_{\BX_1} f&(p_0,p_1) \, \dd\!\mu_1(p_0,p_1) = \int_{\BX_0} \int_\clP f(p_0,p_1) \, \dd\!\nu_{p_0}(p_1) \, \dd\!\mu_0(p_0)\\ &= \int_{\BX_0} \int_\clP f(p_0,p_1) \, \dd\!\nu_{p_1}(p_0) \, \dd\!\mu_0(p_1) \quad = \quad  \int_{\BX_1} f(p_0,p_1) \, \dd\!\mu_1(p_1, p_0).
\end{align*}
Now assume that $k \geq 2$ and let $W := \Span(p_0, \dots, p_{k-2})^\omega$. Then we have for $f \in C(\BX_k)$ that
\[
    \int_{\BX_k} f(p_0,\ldots,p_k) \, \dd\!\mu_k(p_0, \dots, p_k) = \int_{\BX_{k-1}} \int_\clP f(p_0,\ldots,p_k) \, \dd\!\nu_{p_0, \dots, p_{k-1}}(p_k) \, \dd\!\mu_{k-1}(p_0, \dots, p_{k-1});
\]
by the induction hypothesis, $\mu_k$ is invariant under all permutations of the variables $p_0, \dots, p_{k-1}$, and by \autoref{thm:nu_Borel}, so is $\nu$. Moreover, we have the chain of equalities
\begin{align*}
 \int_{\BX_k} f&(p_0,\ldots,p_k) \, \dd\!\mu_k(p_0, \dots, p_k) = \int_{\BX_{k-1}} \int_\clP f(p_0,\ldots,p_k) \, \dd\!\nu_{p_0, \dots, p_{k-1}}(p_k) \, \dd\!\mu_{k-1}(p_0, \dots, p_{k-1}) \\
&= \int_{\BX_{k-2}} \int_\clP \int_\clP f(p_0,\ldots,p_k) \, \dd\!\nu_{p_0, \dots, p_{k-1}}(p_k) \, \dd\!\nu_{p_0, \dots, p_{k-2}}(p_{k-1}) \, \dd\!\mu_{k-2}(p_0, \dots, p_{k-2}) \\
&= \int_{\BX_{k-2}} \int_\clP \int_\clP f(p_0,\ldots,p_k) \, \dd\!\lambda_{W \cap (p_{k-1})^\omega}(p_k) \, \dd\!\lambda_W(p_{k-1}) \, \dd\!\mu_{k-2}(p_0, \dots, p_{k-2})\\
&= \int_{\BX_{k-2}} \int_\clP \int_\clP f(p_0,\ldots,p_k) \, \dd\!\lambda_{W \cap (p_k)^\omega}(p_{k-1}) \, \dd\!\lambda_W(p_k) \, \dd\!\mu_{k-2}(p_0, \dots, p_{k-2})\\
&= \int_{\BX_k} f(p_0,\ldots,p_k) \, \dd\!\mu_k(p_0, \dots, p_{k-2}, p_{k}, p_{k-1}),
\end{align*}
where the second-to-last one follows from \autoref{FubiniLemma}. This shows that $\mu_k$ is invariant under the transposition $(k-1, k)$. Since this transposition and the copy of $\frS_k$ in $\frS_{k+1}$ that corresponds to permutations of the variables $p_0, \dots, p_{k-1}$ generate the symmetric group $\frS_{k+1}$, the conclusion follows.
\end{proof}

\subsection{Compatibility of symplectic Stiefel complexes}

So far we have considered Stiefel complexes for each symplectic vector space $(V, \omega)$ separately; we now organize them into a family. To this end we fix a field $\kay \in \{\bbR, \bbC\}$ and consider the embeddings
\[(\kay^0,0) \hookrightarrow (\kay^2, \omega) \hookrightarrow (\kay^4, \omega) \hookrightarrow (\kay^6, \omega) \hookrightarrow \dots \hookrightarrow  (\kay^{2r}, \omega) \hookrightarrow \dots\]
and 
\[1 \hookrightarrow \Sp_2(\kay) \hookrightarrow \Sp_4(\kay) \hookrightarrow \Sp_6(\kay) \hookrightarrow \dots\hookrightarrow \Sp_{2r}(\kay) \hookrightarrow \dots \]  
as in \eqref{SymplecticEmbeddings} and \eqref{SpEmbeddings}. Given $r \geq 0$ we abbreviate $G_r := \Sp_{2r}(\kay)$ and denote by $X_{r, \bcdot}$ the symplectic Stiefel complex associated to $(\kay^{2r}, \omega)$. By \autoref{thm:admissibility}, each $X_{r, \bcdot}$ is an admissible measured $G_r$-object, and by \autoref{thm:transitivity}, this $G_r$-object is $(r-1)$-transitive. We claim that $(G_r)_r$ and $(X_{r, \bcdot})_r$ are compatible. 

Indeed, consider the action of $G_r$ on $X_{r,k}$ for $r \geq 1$ and $k \in [r-1]$. Fix a symplectic basis $(e_{r-1}, \dots, e_0, f_0, \dots, f_{r-1})$ of $(\kay^{2r}, \omega)$ in antidiagonal form (see Subsection \ref{SympBase}). We can choose as a basepoint in $X_{r,k}$ the tuple $o = ([e_{r-1}], \dots, [e_{r-1-k}])$. The stabilizer of $o$ is given by
\[
H_{r, k} = \left\{\left.\begin{pmatrix}
			D & \ast & \ast \\   
			   & A      & \ast \\
			   &         & QD^{-1}Q
			\end{pmatrix} \ \  \right| \begin{array}{l}
								D = {\rm diag}(\lambda_{r-1},\ldots, \lambda_{r-1-k}) \in \GL_{k+1}(\kay) \\
								A \in G_{r-k-1}
			                                     \end{array} \right\}< G_{r},
\]
where $Q$ is the matrix with only 1's on its antidiagonal, and the asterisks correspond to entries conditioned so that the matrix is symplectic. It follows from the form of the matrices that $H_{r, k} < H_{r,k+1}$ for $k < r-1$. Moreover, we have surjective homomorphisms $H_{r,k} \twoheadrightarrow G_{r-k-1}$ with amenable kernel, making the diagram \eqref{eq:inclusions} commute. At this point we have established \autoref{MainThmConvenient}, except for the fact that the Stiefel complex associated with a $2r$-dimensional symplectic vector spaces is $\gamma(r)$-connected.

\section{Measurable connectivity of the symplectic Stiefel complexes} \label{sec:contractibility}
This is the first one of two sections in which we shall complete the proof of \autoref{MainThmConvenient} by proving the desired connectivity properties of symplectic Stiefel complexes.

For this and the next section, let us fix a $2r$-dimensional symplectic vector space $(V, \omega)$ over $\kay \in \{\bbR, \bbC\}$ with associated Stiefel complex $X_\bcdot$. For $q \in \bbN$ such that $\widehat{\gamma}(q) = 2^q + \lceil (q+1)/2\rceil  \leq r$, we consider the complex
\begin{equation}\label{BInftyComplex}
0 \to \bbR \xrightarrow{\dd^{-1}} \Binfty(X_0) \xrightarrow{\dd^0} \Binfty(X_1) \to \dots \to \Binfty(X_{q}) \xrightarrow{\dd^q} \Binfty(X_{q+1}),
\end{equation}
where $\dd^{-1}: \bbR \to \Binfty(X_1)$ denotes the inclusion of constants. We adopt the conventions that $\Binfty(X_{-1}) := \bbR$, and $\Binfty(X_{-k})=0$ and $\dd^{-k} := 0$ for $k \geq 2$. In \autoref{DefHomotopy} below we are going to construct maps
\[
h^k: \Binfty(X_{k+1}) \to \Binfty(X_{k})
\]
for all $k \in \{-1, \dots, q\}$ in \autoref{DefHomotopy} below, and for $k \geq 2$ we set $h^{-k} := 0$. We are then going to show:
\begin{thm}\label{ContractibilityRefined}
\emph{(i)} For $k \in \{-1, \dots, q\}$, the maps $h^k: \Binfty(X_{k+1}) \to \Binfty(X_{k})$ satisfy
\begin{equation} \label{eq:homotopy}
	h^{k} \, \dd^{k} + \dd^{k-1} \, h^{k-1} = \id.
\end{equation}
In particular, the complex \eqref{BInftyComplex} has trivial cohomology up to degree $q$. \vspace{4pt}

\noindent \emph{(ii)} For all $k \in \{-1, \dots, q\}$, $h^k$ descends to a map $h^k: \Linfty(X_{k+1}) \to \Linfty(X_{k})$. In particular, the complex 
\begin{equation}
0 \to \bbR \to \Linfty(X_0) \to \Linfty(X_1) \to \dots \to \Linfty(X_{q}) \to \Linfty(X_{q+1}) 
\end{equation}
has trivial cohomology up to degree $q$.
\end{thm}

\noindent \autoref{ContractibilityRefined} implies that $X_\bcdot$ is measurably $\gamma(r)$-connected, which concludes the proof of \autoref{MainThmConvenient}. This section will be devoted to the proof of part (i) of the theorem; part (ii) will be established in Section \ref{sec:magic}. 

As before, we denote by $\mathcal P := X_0 = \bbP(V)$ the underlying set of points of the Stiefel complex $X_\bcdot$. In the proof of \autoref{ContractibilityRefined} the assumption $\widehat{\gamma}(q) \leq r$ will be applied in the following form: 
\begin{lem}\label{GammaCondition} Let $I \subset \mathcal P$ be a subset of size at most $q+1$, and for every $J \subsetneq I$, let $x_J \in \clP$ be given. If $\widehat{\gamma}(q) \leq r$, then the symplectic complement $(I \cup \{x_J \mid J \subsetneq I\})^\omega$ is non-trivial.
\end{lem}
\begin{proof}
If $i := |I| \leq q+1$, then the set $I \cup \{x_J \mid J\subsetneq I\}$ has cardinality $i + 2^i - 1 \leq q+2^{q+1}$. In particular, we have 
\[
	\dim {\rm span}(I \cup(x_J \mid J \subsetneq I\})) \leq q+2^{q+1} \leq 2\widehat{\gamma}(q) -1 \leq 2r-1.
\]
We deduce that $\dim (I \cup \{x_J \mid J \subsetneq I\})^\omega \geq 2r-(2r-1) = 1$.
\end{proof}

\subsection{Finding a formula for $h^k$}
We have to guess a formula for the homotopies $h^k: \Binfty(X_{k}) \to \Binfty(X_{k-1})$. To get some inspiration, we first consider a toy case.
\begin{exmpl}\label{ex:product}
Let $X$ be a probability space and consider
\[
0 \to \bbR \xrightarrow{d^{-1}} \Binfty(X) \xrightarrow{d^0} \Binfty(X^2)  \xrightarrow{d^1}  \Binfty(X^3)  \to \dots
\]
We define $h^{-1}: \Binfty(X) \to \bbR$ by
\[
	h^{-1}f := \int_X f(t) \, \dd\!\mu(t) , \quad \mbox{for } f \in \Linfty(X).
\]
One then constructs $h_1, h_2, \dots$ satisfying \eqref{eq:homotopy} inductively. For example, if $h^0$ is assumed to satisfy \eqref{eq:homotopy}, then
\[
	h^0(\dd^0\!\varphi)(p_0) \overset{!}{=} \varphi(p_0) -d^{-1}h^{-1}\varphi(p_0) = \int_X (\varphi(p_0) - \varphi(t)) \, \dd\!\mu(t) =  \int_X \dd^0 \! \varphi(t,p_0) \, \dd\!\mu(t),\\
\]
which suggests to define $h^0:  \Binfty(X^2)\to \Binfty(X)$ by 
\[
h^0f(p_0) :=  \int_X f(t,p_0) \, \dd\!\mu(t).
\]
Inductively, one finds the formulas
\begin{equation} \label{eq:homotopy_product}
h^k:  \Binfty(X^{k+2})\to \Binfty(X^k), \quad h^k f(p_0,\ldots,p_k) = \int_{X} f(t,p_0,\ldots,p_k) \, \dd\!\mu(t),
\end{equation}
To interpret these formulas geometrically, think of elements of $X^{k}$ as $(k-1)$-simplices. If $\Delta \in X^{k}$ is such a $(k-1)$-simplex and  $f \in \Binfty(X^{k+1})$ is a function on $k$-simplices, then $h^k f(\Delta)$ is the expected value of the random variable $f(t(\Delta))$, where $t(\Delta)$ is a random $k$-simplex subject to the condition that $\Delta$ is the $0$-th face of $t(\Delta)$. The distribution of the $k$-simplex $t(\Delta)$ is given by the product measure $\mu \otimes \delta_{\Delta}$, where $\delta_{\Delta} \in {\rm Prob}(X^k)$ is the Dirac measure at $\Delta$.
\end{exmpl}
We now try to argue similarly in our case of interest: the case of the Stiefel complex $X_\bcdot$.
\begin{rem}[Degree 0]\label{RemDeg0} We may choose again $h^{-1}: \Binfty(X_0) \to \bbR$ to be the integral with respect to $\mu_0$. The condition we derive for $h^0: \Binfty(X_1) \to \Binfty(X_0)$ is
\[
h^0(\dd^0\!\varphi)(p_0) \overset{!}{=} \varphi(p_0) -d^{-1}h^{-1}\varphi(p_0) = \int_{X_0} \varphi(p_0)-\varphi(t)  \, \dd\!\mu_0(t).
\]
The first attempt would be to rewrite the integrand $\varphi(p_0) - \varphi(t)$ as $\dd^0\!\varphi(t,p_0)$, but this rewriting is illegal, since $(t,p_0)$ is generically not an element of $X_1$. The correct way to rewrite the integrand is to observe that for every $t_0$ perpendicular to both $t$ and $p_0$ we have
\[
\varphi(p_0) - \varphi(t)  = \varphi(p_0)-\varphi(t_0)+\varphi(t_0)-\varphi(t) = \dd^0\!\varphi(t_0,p_0)-\dd^0\!\varphi(t_0,t).
\]
In particular, we can choose $t_0$ to be an auxiliary \emph{random} perpendicular to $t$ and $p_0$. Passing to the expectation then yields the condition
\[
h^0(\dd^0 \varphi)(p_0)  \overset{!}{=} \int_{X_0}\left( \int_{\clP} \big(\dd^0\!\varphi(t_0,p_0)-\dd^0\!\varphi(t_0,t)\big) \, \dd\!\nu_{t,p_0}(t_0)\right)\, \dd\!\mu_0(t),
\]
where $\nu_{t,p_0}$ denotes the perpendicular measure from \autoref{DefNu}. Note that the function of $t$ in brackets above is Borel measurable by \autoref{GeneralMeasurability}, hence integrable. We may thus choose
\[
	h^0 f(p_0) := \int_{X_0} \int_{\clP} \big(f(t_0,p_0)-f(t_0,t)\big) \, \dd\!\nu_{t,p_0}(t_0) \, \dd\!\mu_0(t).
\]
If we consider $t$ and $t_0$ as (dependent) random variables, then we can write this formula as $h^0f(p_0) = \bbE(f(t_0,p_0)-f(t_0,t))$.
If we continue to higher degrees, we have to choose more and more (mutually dependent) auxiliary random variables, and we need to introduce some form of bookkeeping device to keep track of the dependencies among these auxiliary random variable. This will lead us to the notion of a random chaining in the next subsection.
\end{rem}

\subsection{Random chainings}
Given $q \in \mathbb N$, we denote by $\mathcal P^{[\leq q+1]}$ the collection of all finite subsets of $\mathcal P$ of cardinality at most $q+1$. 
\begin{defn} \label{DefChaining} A \emph{random $q$-chaining} of $\clP$ is a collection $t=\{t_I \mid I \in \clP^{[\leq q+1]}\}$ of (mutually dependent) $\clP$-valued random variables with the following properties:
\begin{enumerate}[(i)]
\item The variable $t_\emptyset$ is distributed according to $\mu_0$.
\item If $I \in \mathcal P^{[\leq q+1]}$ and $p \in I$, then $\omega(t_I, p)=0$ almost surely.
\item If $I \in \mathcal P^{[\leq q+1]}$ and $J \subseteq I$, then $\omega(t_I , t_J) = 0$ almost surely.
\item If $g \in G$ and $I \in \mathcal P^{[\leq q+1]}$, then the distributions of $t_{g.I}$ and $g.t_I$ are mutually absolutely continuous.
\end{enumerate}
\end{defn}

The existence of chainings on $\clP$ will be treated in \autoref{ExistsChaining}. Using a random $q$-chaining we can define $\bar X_n$-valued random values for every $n \in [q]$, as follows: Assume we are given a tuple $(p_0, \dots, p_n) \in \BX_n$, and set $t_0 := t_{\{p_0\}}$, $t_{01}:= t_{\{p_0, p_1\}}$, \dots, $t_{01\dots n} := t_{\{p_0, \dots, p_n\}}$. Then expressions like
\[
(t_{01\dots n}, \ldots, t_{01}, t_0, t_\emptyset), \;(t_{01\dots n}, \ldots, t_{01}, t_0, p_0), \; (t_{01\dots n}, \dots, t_{01}, p_0, p_1), \text{ etc.}
\]
define $\BX_{n+1}$-valued random variables. We will formalize this idea in \autoref{ChainingGivesWhatWeWant} after setting up some notation.

\begin{defn} Let $k \in \bbN$. Given $0\leq n \leq k+1$, a \emph{descending $k$-chain} of \emph{length} $n$ is a sequence of the form
\begin{equation}\label{ChainC}
C = (C_0 \supset C_1 \supset \dots \supset C_n)
\end{equation}
with $C_i \subset [k]$ and $|C_i| = k+1-i$. The sets $C_j$ is called the $j$-th \emph{component} of $C$ and $C_n$ is called its \emph{final component}. We also define its \emph{ordered final component} to be ${\rm ord}(C_n) := (i_0, \dots, i_{k-n})$ if  $C_n = \{i_0, \dots, i_{k-n}\}$ and $i_0 < \dots < i_{k-n}$.
\end{defn}
In the sequel, we will denote by $\frC_k$, $\mathfrak C_k^\emptyset$ and $\mathfrak C_k^+$ the collection of all descending $k$-chains, all descending $k$-chains of length $k+1$ (i.e. of maximal length) and all descending $k$-chains of length at most $k$ respectively.
\begin{exmpl} \label{ex:chains}
\begin{enumerate}[(i)]
\item Let $k=0$. Then $\mathfrak C_0^+ = \{(\{0\})\}$ and $\mathfrak C_0^\emptyset = \{(\{0\} \supset \emptyset)\}$. 
\item Let $k=1$. Then $\mathfrak C_1^+$ contains the three chains $(\{0,1\})$,  $(\{0,1\} \supset \{1\})$ and $(\{0,1\} \supset \{0\})$, and the last two can be extended uniquely into chains in $\frC_1^\emptyset$.
\item Let $k=2$. Then $\mathfrak C_k^+$ consists of 10 chains given by
\begin{align*}
&(\{0,1,2]), \;(\{0,1,2\}\supset \{1,2\}),\; (\{0,1,2\}\supset \{0,2\}), \;(\{0,1,2\}\supset \{0,1\}),\\
&(\{0,1,2\}\supset \{1,2\}\supset 2\}),\; (\{0,1,2\}\supset \{1,2\}\supset \{1\}),\; (\{0,1,2\}\supset \{0,2\}\supset \{2\}),\\
&(\{0,1,2\}\supset \{0,2\}\supset \{0\}),\; (\{0,1,2\}\supset \{0,1\}\supset \{1\}),\; (\{0,1,2\}\supset \{0,1\}\supset \{0\})
\end{align*}
Each of the last six chains can be uniquely prolonged (by $\emptyset$) into a chain in $\mathfrak C_k^\emptyset$.
\end{enumerate}
\end{exmpl}
\begin{defn}\label{def:ofc} Assume we are given a $q$-chaining $t$ of $\mathcal P$ and an element  $p = (p_0, \dots, p_k) \in \clP^{k+1}$ for some $k \leq q$. Given $C \in \frC_k$ of length $n$ with ordered final component $(i_0, \dots, i_{k-n})$, we define a $\mathcal P^{k+2}$-valued random variable $t(p, C)$ by
\begin{equation}\label{tpC}
t(p, C) := (t_{p_{C_0}},\, t_{p_{C_1}} \dots,\,  t_{p_{C_n}} ,\, p_{i_0},\, \dots, p_{i_{k-n}}),
\end{equation}
where for a subset $S \subset [k]$ we write ${p}_S := \{p_s \mid s \in S\}$.
\end{defn}
Note that if $n = k+1$, then $C_n = C_{k+1} = \emptyset$, and \eqref{tpC} has to be understood as $t(p, C) := (t_{p_{C_0}}, \dots, t_{p_{C_n}})$.
\begin{lem}\label{ChainingGivesWhatWeWant} For every $p = (p_0, \dots, p_k) \in \BX_k$ and every $k$-chain $C$, we have $t(p, C) \in \BX_{k+1}$ almost surely.
\end{lem}
\begin{proof} We observe that $i_0, \dots, i_{k-n}$ are contained in each of the sets $C_j$, hence $p_{i_0},\ldots,p_{i_{k-n}}$ are perpendicular with respect to $\omega$ to every $t_{p_{C_j}}$ almost surely by property (iii) in \autoref{DefChaining}. Moreover, if $m<n$, then $C_n \subsetneq C_m$ and hence $\omega(t_{p_{C_n}}, t_{p_{C_m}}) = 0$  almost surely by property (ii).
\end{proof}

\begin{defn} A random $q$-chaining is called \emph{generic} if for all $k \leq q$ and $(p_0, \dots, p_k) \in X_k$ and all $k$-chains $C$, we have $t(p, C) \in X_{k+1} \subset \bar X_{k+1}$ almost surely.
\end{defn}

\begin{rem}[Geometric interpretation] Assume that $t$ is a generic random $q$-chaining and let $k\leq q$. If we think of $p =(p_0,\ldots,p_k) \in X_k$ as a $k$-simplex, then every $C \in \mathfrak C_k$ defines a random $(k+1)$-simplex $t(p, C) \in X_{k+1}$, and the following hold: 
\begin{itemize}
\item If $C = ([k])$ has length $0$, then $t(p,C)$ is a random $(k+1)$-simplex in $X_\bullet$ with base given by the $k$-simplex $p$ and tip $t_{\{p_0, \dots, p_k\}}$. 
\item If we prolong a given chain $C = (C_0 \supset \dots \supset C_n)$ to a chain $C' = (C_0 \supset \dots \supset C_n \supset C_{n+1})$, then the random $(k+1)$-simplices $t(p, C)$ and $t(p, C')$ have a common face.
\end{itemize}
\end{rem}

\begin{exmpl} \label{ex:picture}
Based on the previous remark, we expand on the geometric interpretation for a low value of $k$. Assume that $t$ is a generic random $q$-chaining of $\clP$ for $q \geq 1$. Fix $k = 1$ and let $p=(p_0,p_1) \in X_1$. There are five chains in $\frC_1$, which were listed in \autoref{ex:chains} (ii). Evaluating them in $t(p,\bcdot)$ gives rise to the following five random variables in $X_2$: 
\[ \begin{array}{ll}
	t(p,(\{0,1\})) = (t_{\{p_0,p_1\}},p_0,p_1)\, ; & t(p,(\{0,1\} \supset \{0\})) = (t_{\{p_0,p_1\}},t_{\{p_0\}},p_0)\,; \\
	t(p,(\{0,1\} \supset \{1\})) = (t_{\{p_0,p_1\}},t_{\{p_1\}},p_1) \, ; & t(p,(\{0,1\} \supset \{0\} \supset \emptyset)) = (t_{\{p_0,p_1\}},t_{\{p_0\}},t_\emptyset) \, ; \\
	t(p,(\{0,1\} \supset \{1\} \supset \emptyset)) = (t_{\{p_0,p_1\}},t_{\{p_1\}},t_\emptyset)
\end{array} \]
Set $t_i := t_{\{p_i\}}$ for $i \in [1]$, and $t_{01}= t_{\{p_0,p_1\}}$. Geometrically, each one of them corresponds to a random 2-simplex of $X_\bcdot$. We can visualize their arrangement as the simplicial complex in \autoref{fig:thepicture}. \par
\begin{figure}[!h]
\begin{center}
\includegraphics[scale=0.6]{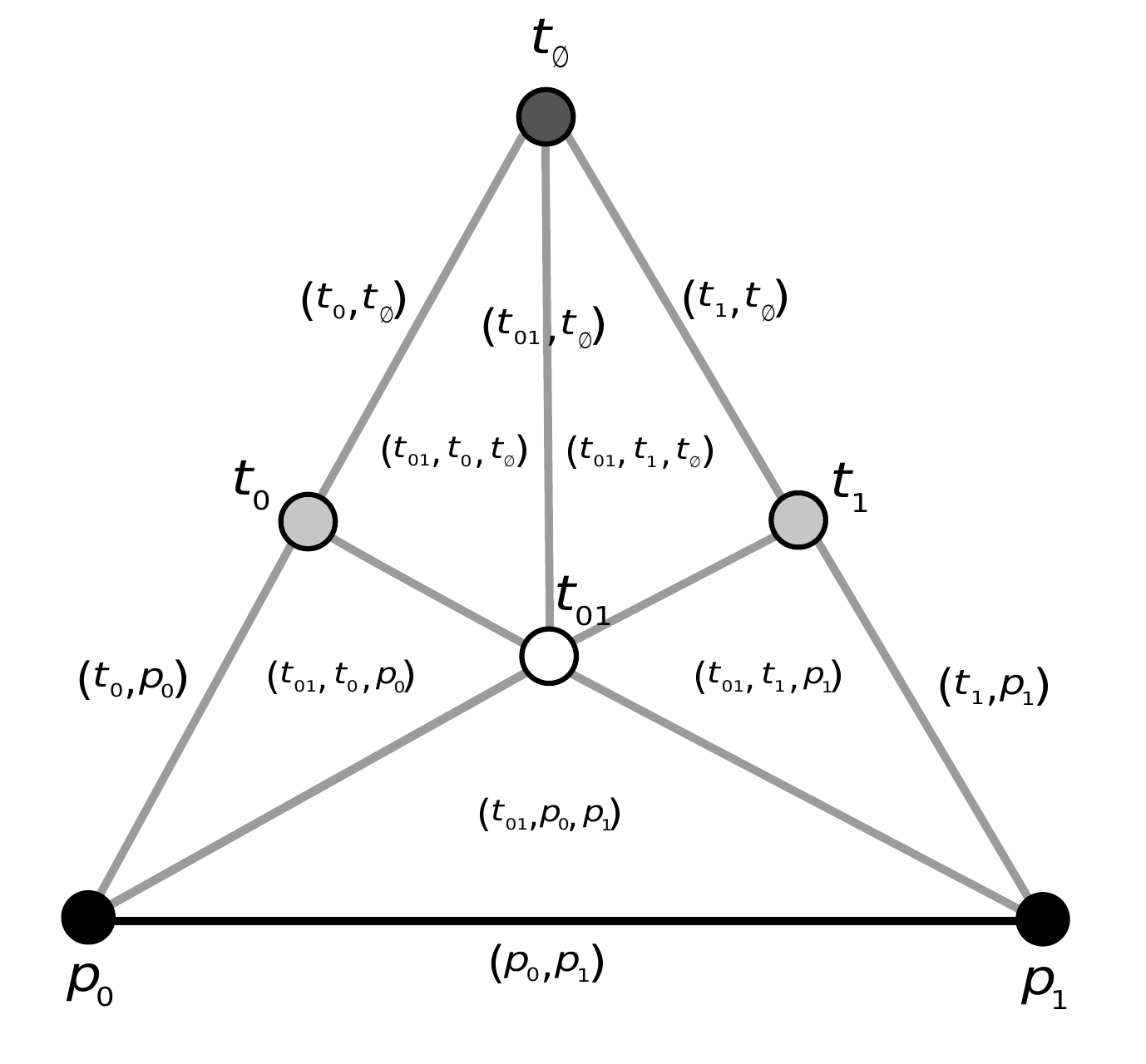} \vspace{-15pt}
\end{center}
\caption{Geometric interpretation of a generic random chaining, and the images of $1$-chains.}
\label{fig:thepicture}
\end{figure}
\noindent The vertices in this complex are the points $p_0, p_1 \in \clP$ and all the $\clP$-valued random variables $t_I$ for every subset $I \subset \{p_0,p_1\}$. We place an edge between two of these vertices, say $u_0$ and $u_1$, if and only if $(u_0,u_1)$ lies in $X_1$ almost surely. Similarly, we place a 2-simplex joining vertices $u_1$, $u_2$ and $u_3$ if and only if $(u_0,u_1,u_2) \in X_2$ almost surely. The genericity assumption on the chaining guarantees that the complex above is ``non-degenerate'', for instance, in the sense that all of its cells are distinct and that apart from the ones in \autoref{fig:thepicture}, there exist no further edges or 2-cells in the arrangement.
\end{exmpl}

In the next subsection, we shall prove that if $\widehat{\gamma}(q) \leq r$, then there exists a generic random $q$-chaining of $\clP$. 

\begin{rem} So far, the necessity of property (iv) in \autoref{DefChaining} and of the genericity of a random chaining have not yet become evident. We anticipate that these features will be highly exploited only in Section \ref{sec:magic}. \vspace{10pt}
\end{rem}

\subsection{Constructing a generic random chaining on $\clP$} Let $q$ be a natural number such that $\widehat{\gamma}(q) \leq r$. We are going to construct a generic random $q$-chaining on $\clP$. For this we have to define a random variable $t_I$ for every $I \in \clP^{[\leq q+1]}$.

We consider first the case of two distinct point $p_0, p_1 \in \clP$. In this case the joint distribution of the four random variables $t_\emptyset$, $t_0 := t_{\{p_0\}}
$,  $t_1 := t_{\{p_1\}}$ and $t_{01} := t_{\{p_0, p_1\}}$ will be given by
\[
\dd\!t_{01}\,\dd\!t_0 \,\dd\!t_1 \,  \dd\!t_\emptyset :=  \dd\!\nu_{p_0,p_1,t_\emptyset,t_0,t_1}(t_{01}) \, \dd\!\nu_{p_0,t_\emptyset}(t_0) \, \dd\!\nu_{p_1,t_\emptyset}(t_1)\, \dd\!\mu_0(t_\emptyset),
\]
which means that for every $f \in C(\mathcal P^4)$ the expectation $\mathbb E(f(t_\emptyset,t_0,t_1,t_{01})) $ of the random variable $f(t_\emptyset,t_0,t_1,t_{01})$ is given by
\[
 \int_{\clP} \left( \int_{\clP} \left(\int_{\clP} \left( \int_\clP f(t_\emptyset,t_0,t_1,t_{01}) \, \dd\!\nu_{p_0,p_1,t_\emptyset,t_0,t_1}(t_{01}) \right) \,  \dd\!\nu_{p_0,t_\emptyset}(t_0) \right)\, \dd\!\nu_{p_1,t_\emptyset}(t_1) \right)\,  \dd\!\mu_0(t_\emptyset)
\]
Here, $\nu$ denotes the assignment of perpendicular measures as in \autoref{DefNu}, and the integrability of the functions given by the inner integrals follows from  \autoref{thm:nu_Borel} and \autoref{GeneralMeasurability}.

With the same notation, the formula in the general case is as follows: If $I = \{p_0, \dots, p_n\}$, then the joint distribution of the random variables $\{t_J\mid J \subset I\}$ is given by
\[
\prod_{J \subset I} \dd\!t_J = \left(\prod_{\emptyset \neq J \subset I} \dd\!\nu_{
J \cup \{t_{J'} \mid J' \subset J\}}(t_J)\right) \dd\!\mu_0(t_\emptyset).
\]
Here the terms in both products are arranged according to the cardinality of $J$ from largest to smallest as in the example above.

To see that this is well-defined we make the following two observations: Firstly, the assumption that $\widehat{\gamma}(q) \leq r$ guarantees in view of \autoref{GammaCondition} that $(J \cup \{t_{J'} \mid J' \subset J\})^\omega \neq \{0\}$, hence the measures $\nu_{\{p \mid p \in J\} \cup \{t_{J'} \mid J' \subset J\}}$ are well-defined. Secondly, as in the previous case the integrability of the functions given by the inner integrals in the iterated integral $\mathbb E[f(t_J)]$ follows from  \autoref{thm:nu_Borel} and \autoref{GeneralMeasurability}.

A crucial property of the construction above is that if $J, J' \subset I$ satisfy $|J| = |J'|$, then $t_{J}$ and $t_{J'}$ are independent relative to the variables $\{t_{J''} \mid J'' \subset J \cup J', |J''| < |J| \}$.

\begin{prop} \label{ExistsChaining}
The collection $\{t_I \mid I \in \clP^{[\leq q+1]}\}$ of random variables defines a generic random $q$-chaining $t$ on $\clP$. 
\end{prop}

\begin{proof} Property (i) of \autoref{DefChaining} holds by definition, Properties (ii) and (iii) follow from \eqref{RandomPerpendicular} and Property (iv) follows from \autoref{thm:nu_Borel}. It remains to show that the chaining is generic. 

For fixed $k \in [q]$, let $p=(p_0,\ldots,p_k) \in X_k$ be a $k$-simplex and $C \in \frC_k$ be a $k$-chain of length $n$, say $C = (C_0 \supset C_1 \supset \cdots \supset C_n)$. If $C \in \frC_k^+$, let ${\rm ord}(C_n) = (i_0,\ldots,i_{k-n})$ be the ordered final component of $C$; otherwise, $C \in \frC_k^\emptyset$ and $C_n = C_{k+1} = \emptyset$. By \autoref{ChainingGivesWhatWeWant}, the random variable $t(p,C)$ as in \eqref{tpC} lies in $\BX_k$ almost surely. Thus, in order to complete the proof of the genericity of $t$, we need to show that $t(p,C) \in X_k$ almost surely, or in other words, that the set
\[ 
	\{t_{p_{C_0}},t_{p_{C_1}},\ldots,t_{p_{C_n}}\} \cup p_{C_n}
\]
is \emph{linearly independent} in $V$ almost surely, i.e. consists of linearly independent 1-dimensional subspaces of $V$ almost surely. Here $p_{C_i}$ is as in \autoref{def:ofc}. In turn, that statement is the case $m=n$ of the next claim, which we prove by induction.
\begin{claim*}
For all $m \in \{-1,0,\ldots,n\}$, the set 
\[
	T_m:=\{t_{p_{C_{n-m}}},\ldots,t_{p_{C_n}}\} \cup p_{C_n}
\]
is almost surely linearly independent in $V$.
\end{claim*}

The case $m=-1$ corresponds to $T_{-1}= p_{C_n}$, which is almost surely a linearly independent set since $p \in X_k$. 
For the induction step, assume that $T_{m-1}$ is almost surely linearly independent in $V$ for an integer $m \in [n]$. Then, the statement follows immediately after showing that $t_{p_{C_{n-m}}} \notin \bbP(\Span(T_{m-1}))$ almost surely. To see this, set first
\begin{longtable}{lll} 
$a_0 := p_{i_0}$, && $b_0:= t_{p_{C_n}\smallsetminus \{p_{i_0}\}}$, \\
$a_1 := p_{i_1}$, && $b_1:= t_{p_{C_n}\smallsetminus \{p_{i_1}\}}$, \\
$\qquad \vdots$ && $\qquad \vdots$ \\
$a_{k-n} := p_{i_{k-n}}$, && $b_{k-n} := t_{p_{C_n} \smallsetminus \{p_{i_{k-n}}\}}$, \\
$a_{k-(n-1)} := t_{p_{C_n}}$, && $b_{k-(n-1)} := p_{C_{n-1} \smallsetminus C_n}$,\\
$a_{k-(n-2)} := t_{p_{C_{n-1}}}$, && $b_{k-(n-2)} := t_{p_{C_{n-2} \smallsetminus (C_{n-1} \smallsetminus C_n)}}$, \\
$\qquad \vdots$ && $\qquad \vdots$ \\
$a_{k-(n-m)} := t_{p_{C_{n-m+1}}}$, && $b_{k-(n-m)}:= t_{p_{C_{n-m} \smallsetminus (C_{n-m+1} \smallsetminus C_{n-m+2})}}$,
\end{longtable}
\noindent and observe that by definition of the random variables, the following identities hold almost surely:
\[ \begin{array}{l}
	\omega(a_i,b_i) \neq 0 \, \mbox{ for all } i \in [k], \quad \omega(a_i,a_j) = 0 \, \mbox{ for all } i,j \in [k], \qand \vspace{2pt} \\
	\omega(a_i,b_j) = 0 \mbox{ for all } i,j \in [k], i \neq j.
\end{array} \]
Thus, by \autoref{thm:lemma_genericity}, the subspace $\Span(T_{m-1})$ of $V$ is, with probability one, symplectic. Now, by definition the random variable $t_{p_{C_{n-m}}}$ is distributed uniformly in the projective subspace $\bbP(\Span(p_{C_{n-m}} \cup \{t_J \mid J \subsetneq C_{n-m}\})^\omega)$ of $\bbP(V)$. But 
\[
	\Span(p_{C_{n-m}} \, \cup\, \{t_J \,\mid \,J \subsetneq p_{C_{n-m}}\})^\omega \cap \Span(T_{m-1}) \subset \Span(T_{m-1})^\omega \cap \Span(T_{m-1}) = (0)
\]
with probability one. Therefore, $t_{C_{n-m}} \notin \bbP(\Span(T_{m-1}))$ almost surely, as asserted above. This completes the proof of the induction step. 
\end{proof}

\subsection{The contracting homotopy}
Let $q \in \mathbb N$ be such that $\gamma(q) \leq r$, and $t$ be a generic random $q$-chaining on $\clP$.

\begin{defn}\label{DefhC} Given $0\leq k \leq q$ and a chain $C \in \frC_k$ we define the associated \emph{partial homotopy} by
\[
h_C: \Binfty(X_{k+1}) \to \Binfty(X_{k}), \quad  h_Cf(p_0, \dots, p_k) = \bbE[f(t(p,C))].
\]
\end{defn}
We now construct the desired homotopy $h^\bcdot$ as follows. 
\begin{defn}\label{DefHomotopy}
We define $h^{-1}:  \Binfty(X_0) \to \bbR$ by $h^{-1}f := \bbE[f(t_\emptyset)]$, and for every $k \in \{0, \dots, q\} $ we set
\begin{equation}\label{Defhk}
h^k:  \Binfty(X_{k+1}) \to \Binfty(X_{k}), \quad h^k = \sum_{C \in \frC_k}{\rm sgn}(C) \cdot h_C,
\end{equation}
where ${\rm sgn}: \mathfrak C_k \to \{\pm 1\}$ is defined by the following convention.
\end{defn}
\begin{rem}[Sign convention for chains] If $I=\{i_0,\ldots,i_l\} \subset [k]$ is such that $i_0 < \cdots < i_l$, we define the operators
\[
	\partial_{a}(I) := I \smallsetminus \{i_a\} \quad \mbox{for any } a \in [l]. 
\]
Now, assume first that $C \in \frC_k^+$ has length $n\leq k$ so that $C_n \neq \emptyset$. There exist unique integers $j_0, \dots, j_{n-1} \in \{0, \dots, k\}$ such that the components of $C$ are given as
\begin{equation}\label{SignOfChains}
C_m = \partial_{j_m} \circ \dots \circ \partial_{j_0}([k]),
\end{equation}
and we define the \emph{sign} of $C$ by
\[
{\rm sgn}(C) := (-1)^{n + j_0 + \dots +j_{n-1}}.
\]
If $C' \in \mathfrak C_k^\emptyset$ has length $k+1$, then we have to modify this definition as follows. In this case, $C_k = \{j\}$ is a singleton, and we define
\[
{\rm sgn}(C') = (-1)^{k+1+j}.
\]
This ensures that if $C$ is a chain of length $k$ and $C'$ is its unique extension to a chain of length $k+1$, then $C$ and $C'$ have opposite signs.
\end{rem}
\begin{exmpl} Let $C = (\{0,1,2\}\supset \{0,2\}\supset \{2\})$ and $C'=(\{0,1,2\}\supset \{0,2\}\supset \{2\}\supset \emptyset)$ be two chains in $\frC_2$. Then $C = \partial_1 \, \circ \, \partial_1([2])$, and thus 
\[ 
	{\rm sgn}(C) = (-1)^{2+1+1} = 1 \qand {\rm sgn}(C') = (-1)^{2+1+2} = -1.
\]
\end{exmpl}

\begin{exmpl}
The following are the expressions of $h^k$ for small values of $k$.
\begin{enumerate}[(i)]
\item For $k=0$, we recover the formula from \autoref{RemDeg0}:
\[
h^0f(p_0) = \bbE[f(t_{\{p_0\}}, p_0) - f(t_{\{p_0\}}, t_\emptyset)].
\]
\item For $k=1$: 
\begin{eqnarray*}
h^1f (p_0, p_1) &=& \bbE\left[f(t_{\{p_0, p_1\}}, p_0, p_1) - f(t_{\{p_0, p_1\}}, t_{\{p_1\}}, p_1) + f(t_{\{p_0, p_1\}}, t_{\{p_0\}}, p_0)\right.\\
&& \quad \left.+ f(t_{\{p_0, p_1\}}, t_{\{p_1\}}, t_\emptyset) - f(t_{\{p_0, p_1\}}, t_{\{p_0\}}, t_\emptyset)\right].
\end{eqnarray*}
\item For $k=2$:
\begin{align*}
& h^2f (p_0, p_1, p_2) = \\
= \bbE & \left[f(t_{012}, p_0, p_1, p_2) - f(t_{012}, t_{12}, p_1, p_2)+f(t_{012}, t_{02}, p_0, p_2) - f(t_{012}, t_{01}, p_0, p_1)\right. \\
 & + f(t_{012}, t_{12}, t_2, p_2) - f(t_{012}, t_{12}, t_1, p_1) - f(t_{012}, t_{02}, t_2, p_2) + f(t_{012}, t_{02}, t_0, p_0)\\
 & + f(t_{012}, t_{01}, t_1, p_1) - f(t_{012}, t_{01}, t_0, p_0) - f(t_{012}, t_{12}, t_2, t_\emptyset) + f(t_{012}, t_{12}, t_1, t_\emptyset) \\
 & + \left. f(t_{012}, t_{02}, t_2, t_\emptyset) - f(t_{012}, t_{02}, t_0, t_\emptyset)- f(t_{012}, t_{01}, t_1, t_\emptyset) + f(t_{012}, t_{01}, t_0, t_\emptyset)\right],
\end{align*}
where we have used the shorthand notations $t_{i} := t_{\{p_i\}}, t_{ij} := t_{\{p_i, p_j\}}$ and $t_{012} := t_{\{p_0, p_1, p_2\}}$.
\end{enumerate}
For larger values of $k$, writing out $h^k$ explicitly gets quite tedious.
\end{exmpl}
\begin{proof}[Proof of \autoref{ContractibilityRefined}.(i)] We fix $k \leq q$, $p=(p_0, \dots, p_k) \in X_k$ and $f \in \Binfty(X_k)$. We have to show that 
\begin{equation}\label{ToyToShow}
h^kd^k f(p) +d^{k-1}h^{k-1}f(p)  = f(p).
\end{equation}
The cases $k=-1$ and $k=0$ are immediate from the formulas above, hence we will assume $k \geq 1$. Since $p$ will be fixed throughout our discussion, we will use the shorthand notations $t_A := t_{p_A}$ for $A \subset [k]$ and $t(C) := t(p,C)$ for $C \in \mathfrak C_k$.
Now, by definition,
\[
h^k d^k f(p) \quad =\quad \bbE\left[ \sum_{C \in \frC_k} {\rm sgn}(C) d^k f(t(C))\right] \quad = \quad  \bbE\left[   \sum_{j=0}^{k+1}\sum_{C \in \frC_k}  (-1)^j{\rm sgn}(C) \delta^j f (t(C)) \right].
\]
Let us first deal with the summand $j=0$. We distinguish two cases: First we consider the length $0$ chain $C = ([k])$. In this case we have 
\[
 (-1)^j{\rm sgn}(C) \delta^j f (t(C)) = \delta^0 f(t_{\{0, \dots, k\}},p_0,\ldots,p_k) =  f(p_0, \dots, p_k).
\]
Secondly, let $C$ be a chain of length $n \geq 1$. Assume that  $C_1 = [k] \smallsetminus \{i\}$ and ${\rm ord}(C_n) = (i_0, \dots, i_{k-n})$. Then
\begin{eqnarray*}
 (-1)^j{\rm sgn}(C) \delta^j f (t(C)) &=& {\rm sgn}(C) \delta^0 f(t_{[k]}, t_{[k]\smallsetminus \{i\}}, \dots, t_{\{i_0, \dots, i_{k-n}\}}, p_{i_0}, \dots, p_{i_{k-n}})\\
 &=& {\rm sgn}(C) f(t_{[k]\smallsetminus \{i\}}, \dots, t_{\{i_0, \dots, i_{k-n}\}}, p_{i_0}, \dots, p_{i_{k-n}})
\end{eqnarray*}
Using this identity, it is not hard to see that for a fixed $i \in [k]$,
\[
	\sum_{\hspace{10pt}\underset{\scriptstyle C_1 = [k] \smallsetminus \{i\}}{C \in \frC_k \smallsetminus \{[k]\}}} {\rm sgn}(C) \, \delta^0 f(t(C)) = (-1)^{i+1} \delta^i h^{k-1}f(p),
\]
and the sum over all $i \in [k]$ of the right-hand side equals $-d^{k-1}h^{k-1}f(p)$. Hence we obtain
\[
h^kd^kf(p) +d^{k-1}h^{k-1}f(p)  \; = \; f(p) + \bbE\left[ \sum_{C \in \frC_k}  \sum_{j=1}^{k+1}  (-1)^j{\rm sgn}(C) \delta^j f (t(C)) \right].
\]
Now let $C$ be a chain of length $n\geq 1$ and let $1 \leq j \leq n-1$. Let $C_{j-1}\smallsetminus C_j = \{a\}$ and $C_j \smallsetminus C_{j+1} = \{b\}$ and set
\[
C' := (C_0 \supset \dots \supset C_{j-1} \supset C_{j-1} \smallsetminus \{b\} \supset C_{j+1} \supset \dots C_n)
\]
Then
\[
\delta^j f(t(C)) = \delta^j(t({C'})),
\]
but $C$ and $C'$ have opposite signs, hence these two terms cancel in the sum above. We thus obtain
\begin{equation} \label{eq:midcomphomotopy}
h^kd^kf(p) +d^{k-1}h^{k-1}f(p) - f(p) \,  = \, \bbE\left[ \sum_{C \in \frC_k}  \sum_{j={\rm length}(C)}^{k+1}  (-1)^j{\rm sgn}(C) \delta^j f (t(C)) \right].
\end{equation}
Now assume that $C$ has length $k$ and $C' := (C_0 \supset \dots \supset C_k \supset \emptyset) \in \frC_k^\emptyset$ is the unique extension of $C$ to length $k+1$. On the one hand we observe that
\[
\delta^{k+1}f(t(C)) = \delta^{k+1}f(t({C'})),
\]
and since $C$ and $C'$ have opposite signs, these terms cancel each other.\footnote{This is the reason why we had to define the sign differently for chains of maximal length.} Let us abbreviate by $\frC_k^k \subset \frC_k^+$ the subset of all chains of length $k$. Observe that for any $C \in \frC_k^\emptyset$, the inner sum in the right-hand side of \eqref{eq:midcomphomotopy} consists of a single term, corresponding to $j = k+1 = {\rm length}(C)$; for $C \in \frC_k^k$, the sum has two terms, namely $j=k={\rm length}(C)$ and $j=k+1$.  We thus obtain
\begin{eqnarray*}
&& \hspace{-10pt} h^kd^kf(p) +d^{k-1}h^{k-1}f(p) - f(p)\\
&\hspace{-10pt} =& \hspace{-10pt} \bbE\left[ \sum_{C \in \frC_k^+}  \sum_{j={\rm length}(C)}^{k+1}  (-1)^j{\rm sgn}(C) \delta^j f (t(C)) -
\sum_{C \in \frC_k^k}(-1)^{k+1}{\rm sgn}(C) \delta^{k+1} f (t(C)) \right]\\
&\hspace{-10pt} =& \hspace{-10pt} \bbE\left[\sum_{C \in \frC_k^+}{\rm sgn}(C)(-1)^{{\rm length}(C)} \delta^{{\rm length}(C)}f(t(C)) + \sum_{C' \in \frC_k^+\smallsetminus \frC_k^k}\sum_{j={\rm length}(C') + 1}^{k+1}  {\rm sgn}(C') (-1)^j \delta^j f(t(C')) \right].
\end{eqnarray*}
Each term which appears in the first sum also appears in the second sum with the opposite sign and vice versa. Indeed, if $C \in \frC_k^+$ has length $n$, then we can define 
\[
C' := (C_0 \supset \dots \supset C_{n-1}).
\]
Then $C' \in \frC^+_k \smallsetminus \frC_k^k$ (since $C'$ is shorter than $C$) and 
\[
 \delta^{{\rm length}(C)}f(t(C)) = \delta^j f((t(C'))
\]
for a unique $j\geq {\rm length}(C')+1 = n$. Upon checking that the signs in front of these terms are opposite, this finishes the proof.
\end{proof}
\begin{rem} What we have actually proved is that if $\clP$ admits a generic random $q$-chaining, then the statement of \autoref{ContractibilityRefined}.(i) holds. It is a consequence of the rather crude estimate from \autoref{GammaCondition} that such a chaining exists if $\widehat{\gamma}(q) \leq r$. If one were able to obtain a more efficient random chaining, then one would obtain a better bound in \autoref{ContractibilityRefined}, which would result in a better stability range.
\end{rem}

\section{From $\Binfty$ to $L^\infty$} \label{sec:magic}
The purpose of this section is to establish Part (ii) of \autoref{ContractibilityRefined} which says that the maps $h^k:  \Binfty(X_{k+1}) \to \Binfty(X_{k})$ descend to maps $h^k:  \Linfty(X_{k+1}) \to \Linfty(X_{k})$ for all $k = -1, \dots, q$. This means: if $f_0, f_1 \in  \Binfty(X_{k+1})$ agree $\mu_{k+1}$-almost everywhere, then $h^k f_0$ and $h^k f_1$ agree $\mu_{k}$-almost everywhere. This is obvious for $h^{-1}: L^\infty(X_0) \to \bbR$, since if $f_0$ and $f_1$ agree $\mu_0$-almost everywhere then 
\[
h^{-1}(f_1) = \int_{X_0} f_1(t_\emptyset) \, \dd\!\mu_0(t_\emptyset) = \int_{X_0} f_2(t_\emptyset) \, \dd\!\mu_0(t_\emptyset)= h^{-1}(f_2).
\]
Now if $k \in 0, \dots, q$, then by definition
\[
 h^k = \sum_{C \in \frC_k}{\rm sgn}(C) \cdot h_C,
\]
hence the proof of \autoref{ContractibilityRefined}.(ii) reduces immediately to the following lemma.
\begin{lem}\label{DescentConvenient} Assume that $\mathcal P$ admits a generic $q$-chaining (which holds e.g. if $\hat\gamma(q) \leq r$) and let $k \leq q$. Then for every $C\in \mathfrak C_k$ the map $h_C: \Binfty(X_{k+1}) \to \Binfty(X_{k})$ descends to a map $h_C: \Linfty(X_{k+1}) \to \Linfty(X_{k})$.
\end{lem}

From now on we fix $q$ and $k$ as in \autoref{DescentConvenient} and a chain $C \in \mathfrak C_k$. We recall that the map $h_C:  \Binfty(X_{k+1}) \to \Binfty(X_{k})$ is given by $h_C(f) = \bbE[f(t(p,C))]$. For every $p \in X_k$, recall from \eqref{eq:jointdistr} that $\sigma_{p}$ denotes the distribution of the random variable $t(p, C) \in X_{k+1}$. This defines a map
\begin{equation}\label{DefsigmaC}
\sigma: X_k \to {\rm Prob}(X_{k+1}), \; {p} \mapsto \sigma_{p}, \quad \text{such that} \quad h_C(f)({p}) = \int_{X_k} f \, \dd\!\sigma_{p}.
\end{equation}
Note that the measures $\sigma_p$ are probability measure on the \emph{non-compact} space $X_{k+1}$.

\begin{rem}[Probability measures on non-compact spaces]\label{NoncompactMeasures}
Let $X$ be a locally compact second-countable (lcsc) space. If $X$ is compact, then we will always topologize ${\rm Prob}(X)$ as a subspace of $C(X)^*$, where the latter is equipped with the weak-$*$-topology with respect to $C(X)$. Since $X$ is second-countable, it is metrizable and hence ${\rm Prob}(X)$ is again a compact metrizable space. If $X$ is non-compact, then we will always topologize ${\rm Prob}(X)$ as follows: We denote by  $X^+ = X \cup \{\infty\}$ the one-point compactification of $X$, and 
consider ${\rm Prob}(X)$ as a subspace of ${\rm Prob}(X^+)$ given by
\[
{\rm Prob}(X) = \{\mu \in {\rm Prob}(X^+) \mid \mu(\{\infty\}) = 0\} \subset {\rm Prob}(X^+).
\]
We then equip ${\rm Prob}(X) \subset {\rm Prob}(X^+)$ with the subspace topology. In either case, our choice of topology on ${\rm Prob}(X)$ defines a canonical Borel $\sigma$-algebra on ${\rm Prob}(X)$.
\end{rem}
We are going to establish measurability of the map $\sigma$ with respect to the Borel structure just defined.
\begin{prop}\label{PropertiesSigma} The map $\sigma$ from \eqref{DefsigmaC} is continuous and $G$-quasi-equivariant, i.e. for every $\mathbf{p} \in X_k$ and every $g \in G$ the measures $g_*\sigma_{\mathbf{p}}$ and $\sigma_{g\mathbf{p}}$ are mutually absolutely continuous.
\end{prop}
\begin{proof} The assignment $\sigma$ is $G$-quasi-equivariant, since $\mu_0$ is $G$-quasi-invariant and $\nu$ is  $G$-quasi-equivariant by \autoref{thm:nu_Borel}. To see continuity (with respect to the weak-$*$-topology on $X_{k+1}^+$) we have to show that for every $f \in C_0(X_{k+1})$ the map $\mathbf{p} \mapsto \int_{X_{k+1}} f d\sigma_{\mathbf{p}}$ is continuous. For this we first observe that the map $X_{m} \to {\rm Prob}(\mathcal P)$, $(q_0, \dots, q_m) \mapsto \nu_{q_0, \dots, q_m}$ is continuous for every $m$ by \autoref{thm:nu_Borel}. In view of this observation, continuity of $\int_{X_{k+1}} f d\sigma_{\mathbf{p}}$ in $\mathbf{p}$ follows from the explicit formula by iterated application of \autoref{ContinuityMeasureFamilies}; here we use that the chaining is generic, so that at each integration step the random points in the index of $\nu$ are linearly independent almost surely.
\end{proof}
From \autoref{PropertiesSigma}, we deduce that \autoref{DescentConvenient}, and hence \autoref{ContractibilityRefined} (ii), are direct consequences of the following general measure-theoretical result. See \autoref{defn:regspace} for the definition of a regular $G$-space.
\begin{lem}\label{MagicLemma} Let $(X, \mu_X)$ and $(Y, \mu_Y)$ be regular $G$-spaces and let $\sigma: Y \to {\rm Prob}(X)$, $p \mapsto \sigma_p$ be a $G$-quasi-equivariant Borel map. If $G$ acts transitively on $X$, then the map
\[
\sigma^*: \Binfty(X) \to \Binfty(Y), \quad \sigma^*f(y) = \int_X f \, \dd\!\sigma_y
\]
descends to a map $\sigma^*: L^\infty(X, \mu_X) \to L^\infty(Y, \mu_Y)$.
\end{lem}
The proof of  \autoref{MagicLemma} is based on the following general lemma. Let $X$ and $Y$ be lcsc spaces and let $\sigma: Y \to {\rm Prob}(X)$, $y \mapsto \sigma_y$ be a Borel map. We recall from Subsection \ref{subsec:IntegralMeasure} that if $\mu$ is a probability measure on $Y$, then the integral $ \int_Y \sigma_y\, d\mu(y)$ is defined as the unique probability measure on $X$ such that
\[
\int_X f \, \dd\!\left(\int_Y\sigma_y\, \dd\!\mu(y)\right) :=   \int_Y \left(\int_X f \, \dd\!\sigma_y\right) \, \dd\!\mu(y) \quad (f \in C_c(X)).
\]
We are interested in equivariance properties of this construction. Thus assume that a lcsc group $G$ acts continuously on $X$ and $Y$. If $\mu \in {\rm Prob}(Y)$ is $G$-quasi-invariant, i.e.\ $g_*\mu \ll \mu$ for all $g \in G$, then all the measures $\{g_*\mu \mid g \in G\}$ are mutually equivalent, and we denote by 
\[
r_\mu(g,y) := \frac{d(g^{-1}_*\mu)}{d\mu}(y)
\]
the associated Radon-Nikodym cocycle. Note that, with our convention,
\[
\int_Y f(g.y)d\mu(y) = \int_Y f(y) r_\mu(g^{-1},y) d\mu(y).
\]
Also note that $r_\mu(g, \cdot)$ is a $\mu$-almost everywhere positive function. 
\begin{lem}\label{QuasiInvarianceUnderIntegration} Assume that $G$ acts continuously by homeomorphisms on the lcsc spaces $X$ and $Y$ and that
$\sigma: Y \to {\rm Prob}(X)$ is a $G$-quasi-equivariant Borel map.  If $\mu \in {\rm Prob}(Y)$ is $G$-quasi-invariant, then so is $\int_Y \sigma_y d\mu(y) \in {\rm Prob}(X)$.
\end{lem}
\begin{proof} Let us abbreviate
\[
h(g, x, y) :=  \frac{d(g_*\sigma_y)}{d\sigma_{g.y}}(x).
\]
Let $A \subset X$ be a Borel subset and let $g \in G$. Then we have
\begin{eqnarray*}
\left(g_*\left(\int_Y \sigma_y d\mu(y)\right)\right)(A) &=& \left(\int_Y \sigma_y d\mu(y)\right)(g^{-1}.A) =  \int_Y \left(\sigma_y(g^{-1}.A)\right) d\mu(y)\\
&=& \int_Y \left(g_*\sigma_y(A)\right) \, d\mu(y) = \int_Y \int_A d(g.\nu_y)(x) \, d\mu(y)\\
&=&  \int_Y \int_A  \frac{d(g_*\sigma_y)}{d\sigma_{g.y}}(x) \, d\sigma_{g.y}(x) \, d\mu(y)\\
&=&  \int_Y \left(\int_A h(g,x, y) \, d\sigma_{g.y}(x)\right) \, d\mu(y)\\
&=& \int_Y \left(\int_A  h(g,x, g^{-1}.y) \, d\sigma_{y}(x)\right)  r_\mu(g^{-1}, y) \, d\mu(y)\\
&=& \int_Y \int_A  h(g,x, g^{-1}.y) \, r_\mu(g^{-1}, y)  \,d\sigma_{y}(x)  \,d\mu(y).
\end{eqnarray*}
Thus $A$ is a null set for $g_*\left(\int_Y \sigma_y d\mu(y)\right)$ if and only if 
\[
 \int_Y \int_A  h(g,x, g^{-1}.y) r_\mu(g^{-1}, y)  d\nu_{y}(x)  d\mu(y) = 0.
\]
Since the integrand is positive, this is equivalent to
\[
 0 = \int_Y \, d\nu_{y}(x)  d\mu(y) = \int_Y \nu_y(A)\, d\mu(y) = \left(\int_Y \sigma_y d\mu(y)\right)(A).
 \]
 This shows that $g_*\left(\int_Y \sigma_y d\mu(y)\right)$ and $\left(\int_Y \sigma_y d\mu(y)\right)$ have the same null sets and finishes the proof.
\end{proof}

\begin{proof}[Proof of Lemma \ref{MagicLemma}] In view of Lemma \ref{DescendCriterion} it suffices to check that
\[
\mu'_X := \int_Y \sigma_y d\mu_Y(y)\ll \mu_X.
\]
Since $\mu_Y$ is $G$-quasi-invariant and $\sigma$ is $G$-quasi-equivariant, the measure $\mu'_X$ is $G$-quasi-invariant  by Lemma \ref{QuasiInvarianceUnderIntegration}. Since $X$ is homogeneous, there is a unique $G$-quasi-invariant measure class on $X$, hence in particular $\mu'_X \ll \mu_X$.
\end{proof}
We have now established \autoref{ContractibilityRefined}, and hence also \autoref{MainThmConvenient}.

\section{Computations in spectral sequences} \label{sec:proof_abstract_stability}

In this section we are going to establish the following quantitative version of \autoref{thm:abstract_stability}:
\begin{thm} \label{thm:abstract_stability_explicit}
Let $R \in \mathbb N\cup \{\infty\}$, let $(G_r, X_{r, \bcdot})_{r \in [R]}$ be a measured Quillen family of length $R$ with parameters $(\gamma,\tau)$, and let $q_0 \geq 1$ be a natural number such that for every $r \in [R-1]$, the inclusion $G_r \hookrightarrow G_{r+1}$ induces an isomorphism 
\[
	\Hcb^q(G_{r+1}) \xrightarrow{\sim} \Hcb^q(G_r) \quad \mbox{ whenever } q \leq q_0.
\]
Furthermore, let the functions $\widetilde{\gamma}$ and $\widetilde{\tau}$ be defined as 
\begin{equation} \label{eq:parameters_tilde}
	\widetilde{\gamma}(q, r) := \min_{j=q_0+1}^q \gamma\big(r+1-2(q-j)\big) - j \qand \widetilde{\tau}(q,r) :=  \min_{j=q_0+1}^q \tau\big(r+1-2(q-j)\big) - j.
\end{equation}
Then for all $r \in [R-1]$ and $q \geq 0$, the inclusion $G_r \hookrightarrow G_{r+1}$ induces an isomorphism (resp. an injection) $\Hcb^q(G_{r+1}) \to \Hcb^q(G_r)$ whenever 
\[
\min\{\widetilde{\gamma}(q,r), \widetilde{\tau}(q,r)-1\} \geq 0 \qquad (\text{resp. }\min\{\widetilde{\gamma}(q,r), \widetilde{\tau}(q,r)\} \geq 0).
\]
\end{thm}

\begin{rem}[Initial conditions]\label{IC} We refer to the condition that the inclusions induce isomorphisms $\Hcb^q(G_{r+1}) \to \Hcb^q(G_{r})$ for all $q \leq q_0$ as the \emph{initial condition}. We can always choose $q_0 := 1$, since $\Hcb^1 = 0$ for trivial coefficients. If the $G_r$ are all connected simple Lie groups and either all of Hermitian type (as in the symplectic case) or all of non-Hermitian type, then we can choose $q_0:=2$ by \cite{Burger-Monod2,Burger-Monod1}. For example this is the case for the families $({\rm Sp}_{2r}(\bbR))$ and $({\rm Sp}_{2r}(\bbC))$.
\end{rem}

Let us reassure ourselves that \autoref{thm:abstract_stability_explicit} implies \autoref{thm:abstract_stability}: Let $(G_r, X_{r, \bcdot})_{r \in \mathbb N}$ be an infinite measured Quillen family with parameters $(\gamma, \tau)$ and assume that $\gamma(r) \to \infty$ and $\tau(r) \to \infty$ as $r \to \infty$. By \autoref{IC} we may choose the initial condition $q_0 := 1$. Since $\gamma(r) \to \infty$ and $\tau(r) \to \infty$ we then find for every $q \geq q_0$ some $r(q) \in \mathbb N$ such that for all $j \in \{q_0+1, \ldots, q\}$ and all $r \geq r(q)$ we have
\[
 \gamma\big(r+1-2(q-j)\big) \geq  j \qand  \tau\big(r+1-2(q-j)\big) \geq j+1.
\]
Then \autoref{thm:abstract_stability_explicit} implies that
\[
\Hcb^q(G_{r(q)})\cong \Hcb^q(G_{r(q)+1}) \cong \Hcb^q(G_{r(q)+2}) \cong \dots
\]
This shows that $(G_r)_{r \in \mathbb N}$ is stable.

The proof of \autoref{thm:abstract_stability_explicit} will be given in Subsection \ref{subsec:proof} below based on a spectral sequence argument. The relevant spetral sequences will be constructed in Subsection \ref{subsec:DoubleComplex}.

\subsection{The double complex and its associated spectral sequences}\label{subsec:DoubleComplex}
Let $G$ be an lcsc group, and $X_\bcdot$ be an admissible $G$-object (see Subsection \ref{subsec:GObjects} for basic definitions about admissible $G$-objects). By \autoref{DefAdmissibility}, the $\Linfty$-complex 
\begin{equation} \label{eq:Sp_linfty}
	0 \to \Linfty(X_0) \xrightarrow{\dd^0} \Linfty(X_1) \xrightarrow{\dd^1} \Linfty(X_2) \xrightarrow{\dd^2} \cdots
\end{equation}
associated to $X_{\bcdot}$ is a complex of coefficient $G$-modules. We recall that for every $q \geq 0$, the operator $\dd^q$ is defined as the alternating sum $\dd^q = \sum_{i=0}^{q+1} \, (-1)^i \delta^i$, where $\delta^i: \Linfty(X_q) \to \Linfty(X_{q+1})$ is the map induced by the $i$-th face operator $\delta_i: X_{q+1} \to X_q$. 

For all $p,q \geq 0$, we define the Banach spaces and morphisms
\[
	\LL^{p,q}:=\Linfty\big(G^{p+1} \times X_{q}\big)^{G} \cong \Linfty\big(G^{p+1}; \Linfty(X_{q})\big)^{G} \vspace{-6pt}
\]
\begin{align*}
	\dd^{p,q}_{\rm H}:\,\LL^{p,q} \to \LL^{p+1,q}, \qquad &  \dd^{p,q}_{\rm H}f(g_0,\ldots,g_p):= \sum_{i=0}^p (-1)^i \,f(g_0,\ldots,\hat{g_i},\ldots,g_p), \\
	\dd_{\rm V}^{p,q}:\,\LL^{p,q} \rightarrow\LL^{p,q+1}, \qquad & \dd^{p,q}_{\rm V}f(g_0,\ldots,g_{p-1}):= \dd^q\big(f(g_0,\ldots,g_{p-1})\big),
\end{align*}
where $G$ is equipped with the $G$-action by left-multiplication, and $\dd^q$ are the coboundary operators in \eqref{eq:Sp_linfty} (see \autoref{LInftyGeneral} and \autoref{thm:explaw}). A computation shows that $(\LL^{\bcdot,\bcdot},\dd_{\rm H},\dd_{\rm V})$ is a first-quadrant double complex. Let $\IE_\bcdot^{\bcdot,\bcdot}$ and $\IIE_\bcdot^{\bcdot,\bcdot}$ be the spectral sequences associated with the horizontal and vertical filtrations of $\LL^{\bcdot,\bcdot}$, respectively, both of which converge to the cohomology of the total complex of $\LL^{\bcdot,\bcdot}$ and whose first-page terms and differentials are given by 
\begin{align*}
\IE_{1}^{p,q} & =\HH^q (\LL^{p,\bcdot},\,\dd_{\rm V}^{p,\bcdot}),\quad\Id_{1}^{p,q}=\HH^q(\dd_{\rm H}^{p,\bcdot}): \,\IE_{1}^{p,q}\rightarrow\IE_{1}^{p+1,q},\\
\IIE_{1}^{p,q} & =\HH^q(\LL^{\bcdot,p},\,\dd^{\bcdot,p}_{\rm H}),\quad\IId_{1}^{p,q}= \HH^q(\dd_{\rm V}^{\bcdot,p}):\,\IIE_{1}^{p,q}\rightarrow\IIE_{1}^{p+1,q}. 
\end{align*}
We say that $\IE^{\bcdot,\bcdot}_\bcdot$ and $\IIE^{\bcdot,\bcdot}_\bcdot$ are the \emph{spectral sequences associated to the pair $(G,X_\bcdot)$}. 

If we make additional assumptions on the measurable connectivity and transitivity of $X_\bcdot$, we are able to gather more information about some of their terms and differentials. Concerning the spectral sequence $\IE^{\bcdot,\bcdot}_\bcdot$, we have the following lemma. 

\begin{lem} \label{thm:IE}
Assume that $X_\bcdot$ is measurably $\gamma_0$-connected. Then for all $p \in [\gamma_0]$, the limit term $\IE_\infty^p$ is isomorphic to $\Hcb^p(G)$.
\end{lem}
\begin{proof}
By the measurable $\gamma_0$-connectivity of $X_\bcdot$, the complex \eqref{eq:Sp_linfty} is exact up to degree $\gamma_0$. Applying the functor $\Linfty(G^{p+1};-)^G$ to it, we obtain
\[
0 \to \Linfty\big(G^{p+1}\big)^G \to \LL^{p,0} \to \LL^{p,1} \to \cdots \to \LL^{p,\gamma_0} \to \LL^{p,\gamma_0+1} \to \cdots,
\]
which is, too, exact up to degree $\gamma_0$ because of \autoref{thm:exactftr}. This implies vanishing of the rows $\IE_1^{\bcdot,q}$ for $1\leq q \leq \gamma_0$, and shows that the zeroth row of the first page $\IE^{\bcdot,\bcdot}_1$ is given by
\[
	\IE_1^{p,0} = \HH^0(\LL^{p,\bcdot}, \dd^{p,\bcdot}_\mathrm{V}) \cong \Linfty\big(G^{p+1}\big)^G,\qquad \mbox{for } p \geq 0.
\]
\begin{figure}[b!]
\hspace{-25pt} \begin{minipage}{\textwidth}
\[
\xymatrix@R=3pt@C=3pt{
  & q \\
\vdots & & \vdots && \vdots && \cdots && \vdots && \vdots && \vdots && \iddots \\
{\scriptstyle \gamma_0+1} & & \ast && \ast && \cdots && \ast && \ast && \ast && \cdots \\
{\scriptstyle \gamma_0} & & 0 && 0 && \cdots && 0 && 0 && 0 && \cdots \\
{\scriptstyle \gamma_0-1} & & 0 && 0 && \cdots && 0 && 0 && 0 && \cdots \\
\vdots & & \vdots && \vdots && \iddots && \vdots && \vdots && \vdots && \cdots \\ 
{\scriptstyle 1} & & 0 && 0 && \cdots && 0 && 0 && 0 && \cdots \\ 
{\scriptstyle 0} & & \Linfty(G)^G \ar[rr] && \Linfty(G^2)^G \ar[rr] && \cdots \ar[rr] && \Linfty(G^{\gamma_0})^G \ar[rr] && \Linfty(G^{\gamma_0+1})^G \ar[rr] && \Linfty(G^{\gamma_0+2})^G \ar[rr] && \cdots \\ 
\ar[rrrrrrrrrrrrrrr] & & & & && && && && && && \hspace{-8pt} p \\
& \ar[uuuuuuuuu] & {\scriptstyle 0} && {\scriptstyle 1} && \cdots && {\scriptstyle \gamma_0-1} && {\scriptstyle \gamma_0} && {\scriptstyle \gamma_0+1} && \cdots }
\]
\vspace{-16pt}
\end{minipage}
\caption{First page $\IE_1^{\bcdot,\bcdot}$}
\vspace{-9pt}
\label{fig:IE_1}
\end{figure}

\autoref{fig:IE_1} is a visualization of the first page $\IE_1^{\bcdot,\bcdot}$. The asterisks denote potentially non-vanishing terms, and the arrows in the bottom row correspond to the maps $\Id_{1}^{p,0}=\HH^0(\dd_{\rm H}^{p,0})=\dd_{\rm H}^{p,0}$ for $p \geq 0$. The cohomology of this row is canonically isomorphic to the continuous bounded cohomology of $G$, so $\IE_2^{p,0}\cong\Hcb^p(G)$ for all $p \geq 0$. The vanishing of the $\gamma_0$ rows above the bottom row implies the lemma. Indeed, for $s \geq 2$, the differentials $\Id_s^{p,q}$ that meet any of the terms $\IE_s^{p,0}$ for $0 \leq p \leq \gamma_0$ emanate from terms at the positions $(p,q)$ for $0 \leq p+q \leq \gamma_0-1$. However, since all these terms already vanished in the second page, the image of the differentials $\Id_s^{p,q}$ will be trivial, leaving the zeroth row unchanged until the limit. In other words, $\IE_\infty^p \cong \IE_\infty^{p,0} \cong \IE_2^{p,0}$ for all $0 \leq p \leq \gamma_0$.
\end{proof}

While the convergence of the spectral sequence $\IE_\bcdot^{\bcdot,\bcdot}$ relies on the measurable connectivity of $X_{\bcdot}$, it will be mostly the transitivity of its $G$-action what will provide information on $\IIE_\bcdot^{\bcdot,\bcdot}$. From now on we assume that $G$ acts $\tau$-transitively on $X_\bcdot$ for some $\tau \in \mathbb N$. We may then assume that for $p \in [\tau]$ we have $X_p = G/H_p$, where $H_0 \supset \dots \supset H_\tau$ are subgroups of $G$.  We denote by $j_p: H_{p+1} \hookrightarrow H_p$ the canonical inclusions, and by \[j_{p+1}^*:\Linfty(G^{\bcdot+1})^{H_p} \hookrightarrow \Linfty(G^{\bcdot+1})^{H_{p+1}} \qand \Hcb^\bcdot(j_{p+1};\id):\Hcb^q(H_{p}) \to \Hcb^q(H_{p+1}),\] 
the induced maps.

\begin{rem}
By \cite[Proposition 10.1.3]{Monod-Book} the Eckmann--Shapiro map 
\[
	\Linfty\big(G^{\bcdot+1}\big)^{H_{p}} \xrightarrow{{\rm ind}} \Linfty\big(G^{\bcdot+1};\Linfty(G/H_{p})\big)^{G}, \quad {\rm ind}(f)(\mathbf{g})(g_0 H_{p}) = f(g_0^{-1}\mathbf{g})
\]
induces for every $p \in [\tau]$ an isomorphism
\[
{\rm ind}: \Hcb^q(H_p) \overset{\sim}{\longrightarrow} \Hcb^q\big(G;\Linfty(G/H_{p})\big) = \Hcb^q\big(G;\Linfty(X_{p})\big) = \IIE^{p,q}_1.
\]
\end{rem}
The remainder of this subsection is devoted to the proof of the following fact:
\begin{prop} \label{thm:IIE}
For every $p< \tau_0$ and $q \geq 0$ define $\Delta^{p,q}: \Hcb^q(H_{p}) \to \Hcb^q(H_{p+1})$ by the following commuting diagram
\begin{equation} \label{eq:diag_IIE}
\begin{gathered}
\xymatrix{\IIE_1^{p,q}\ar[r]^{\IId_1^{p,q}}& \IIE_1^{p+1,q} \\
\Hcb^q(H_{p}) \ar[u]^{\rm ind} \ar[r]^{\Delta^{p,q}} & \Hcb^q(H_{p+1}) \ar[u]_{\rm ind}}
\end{gathered}
\end{equation}
Then $\Delta^{p,q} = \Hcb^q(j_{p+1};\id) $ if $p$ is odd and $\Delta^{p,q} = 0$ if $p$ is even.
\end{prop}
For the moment fix $p \in [\tau_0-1]$ and $i \in [p]$. We may then choose elements $w_i \in G$ such that the face map $\delta_i: G/H_{p+1} \to G/H_p$ is given by $\delta_i(eH_{p+1}) = w_iH_p$. Since $\delta_i$ is $G$-equivariant we then have
\begin{equation} \label{eq:Ai}
	\delta_i(gH_{p+1})=gw_iH_p \quad \mbox{for all } g \in G. 
\end{equation}
Left-multiplication by $w_i$ then induces a map
\[
\lambda_{w_i}: \Linfty(G^{q+1})^{H_{p}} \to \Linfty(G^{q+1})^{H_{p+1}}, \quad \lambda_{w_i}f(\mathbf{g}) := f(w_i^{-1}\mathbf{g}),
\]
since for all $f \in \Linfty(G^{q+1})$, $h \in H_{p+1}$ and $\mathbf{g} \in G^{q+1}$ we have
\[
	(h w_i)(f)(\mathbf{g}) = f(w_i^{-1}h^{-1}\mathbf{g}) = f(w_i^{-1}h^{-1}w_iw_i^{-1}\mathbf{g}) = f(w_i^{-1}\mathbf{g})= w_i(f)(\mathbf{g}).
\]
Unravelling definitions shows that if we denote by $\delta_i^\ast$ the map obtained from applying the functor $\Linfty(G^{\bcdot+1} \times -)^G$ to $\delta_i$, then
the diagrams
\begin{equation}\label{wiDiagram}
\begin{gathered}
\xymatrix{\Linfty\big(G^{\bcdot+1};\Linfty(G/H_{p})\big)^{G}\ar[r]^{\delta_i^\ast}&\Linfty\big(G^{\bcdot+1};\Linfty(G/H_{p+1})\big)^{G}\\ 
\Linfty(G^{q+1})^{H_{p}} \ar[u]^{{\rm ind}} \ar[r]^{\lambda_{w_i}} & \Linfty(G^{q+1})^{H_{p+1}} \ar[u]_{{\rm ind}}
}\end{gathered}\end{equation}
commute for every $q\geq 0$. We deduce:
\begin{lem} \label{thm:homotopy_wi} Let $p \in [\tau_0-1]$ and $i \in [p]$. Then for every $q \geq 0$ we have a commuting diagram
\[
\xymatrix{\IIE_1^{p,q} \ar[rr]^{\Hcb^q(\id;\delta_{i}^\ast)} && \IIE_1^{p,q}\\
\Hcb^q(H_{p}) \ar[u]^{\rm ind} \ar[rr]^{\HH^q(j_{p+1};\id)} && \Hcb^q(H_{p+1}) \ar[u]_{\rm ind}.
}\]
In particular, $\Hcb^q(\id;\delta_{i}^\ast)$ is independent of $i$.
\end{lem}
\begin{proof} In view of the diagram \eqref{wiDiagram} it suffices to show that the maps $j_{p+1}^*$ and $\lambda_{w_i}$ induce the same map in cohomology. Define  \emph{prism operators} $h^\bcdot: \Linfty(G^{\bcdot+2})^{H_p} \to \Linfty(G^{\bcdot+1})^{H_{p+1}}$ by the formula
\[
h^q f(g_0,\ldots,g_q) = \sum_{l=0}^q (-1)^l f(g_0,\ldots,g_l,w_i^{-1}g_l,\ldots, w_i^{-1}g_q).
\]
It is then a routine verification to check that $\lambda_{w_i} - j_{p+1}^\ast   =  h^q \dd^q + \dd^{q-1}h^{q-1}$; see e.g. the proof of Theorem 2.10 in \cite{Hatcher-AT}.
\end{proof}
\begin{proof}[Proof of \autoref{thm:IIE}]
Under the identification $\IIE^{p,q}_1 \cong \Hcb^q(G;\Linfty(X_{p}))$ the differential is given by
\begin{equation*} 
	\IId_1^{p,q} = \textstyle \sum_{i=0}^p (-1)^i \, \Hcb^q(\id;\delta_i^\ast)
\end{equation*}
for all $p,q \geq 0$. Then the proposition follows from \autoref{thm:homotopy_wi}.
\end{proof}

\subsection{Proof of \autoref{thm:abstract_stability_explicit}} \label{subsec:proof}
From now on let $(G_r,X_{r,\bcdot})_{r\in [R]}$ be a measured Quillen family with parameters $(\gamma,\tau)$, as in \autoref{def:Quillen_family}. Furthermore, let $\iota_r$ denote the inclusion $G_r \hookrightarrow G_{r+1}$. Fix an arbitrary $r \in [R-1]$, and choose compatible point stabilizers $H_{r+1,p}$ of $X_{r+1, p}$ for $p \in [\tau]$. By the construction in the last subsection, each pair $(G_{r},X_{r,\bcdot})$ gives rise to two spectral sequences ${}_{r+1}^{{\rm I}}{\rm E}^{\bcdot, \bcdot}$ and ${}_{r+1}^{\rm II}{\rm E}^{\bcdot, \bcdot}$. From \autoref{thm:IIE} we then obtain commutative diagrams
\begin{equation} \label{eq:diag_IIE}
\begin{gathered}
\xymatrix{{}^{\rm II}_{r+1}{\rm E}_1^{p,q}\ar[r]^{\IId_1^{p,q}}&{}^{\rm II}_{r+1}{\rm E}_1^{p+1,q} \\
\Hcb^q(H_{r+1, p}) \ar[u]^{\rm ind} \ar[r]^{\Delta^{p,q}} & \Hcb^q(H_{r+1, p+1}) \ar[u]_{\rm ind}}
\end{gathered}
\end{equation}
By \autoref{def:compatibility} we have  surjective homomorphisms $\pi_p: H_{r+1,p} \twoheadrightarrow G_{r-p}$ with amenable kernel which by \cite[Corollary 8.5.2]{Monod-Book} induce isomorphisms ${\Hcb^q(\pi_p;\id)}$ in continuous bounded cohomology. Combining these with the induction isomorphisms we obtain isomorphisms
\[
	 \Hcb^q(G_{r-p}) \overset{\sim}{\longrightarrow}{}^{\rm II}_{r+1}E_1^{p,q}.
\]
Since by \eqref{eq:inclusions} the diagram
\begin{equation} \label{eq:diag_IIE2}
\begin{gathered}
\xymatrix{\Hcb^q(H_{r+1,p}) \ar[rr]^{\Hcb^q(j_{p+1},\id)} && \Hcb^q(H_{r+1,p+1}) \\
\Hcb^q(G_{r-p}) \ar[u]^{\Hcb^q(\pi_p;\id)} \ar[rr]^{\Hcb^q(\iota_{r-p-1},\id)} && \Hcb^q(G_{r-p-1}) \ar[u]_{\Hcb^q(\pi_{p+1};\id)}}
\end{gathered}
\end{equation}
commutes for every $p < \tau(r+1)$ and $q \geq 0$, we deduce:
\begin{cor} \label{thm:cor_IIE}
For any $p \in [\tau(r+1)-1]$ and $q \geq 0$ there is a commuting diagram of the form
\[ \xymatrixcolsep{4pc}
\xymatrix{{}^{\rm II}_{r+1}E_1^{p,q} \ar[r]^{\IId_1^{p,q}} &{}^{\rm II}_{r+1}E_1^{p+1,q}  \\
\Hcb^q(G_{r-p}) \ar[u]_{\sim} \ar[r]^{D^{p,q}} & \Hcb^q(G_{r-p-1}) \ar[u]^{\sim}},
\]
where $D^{p,q} = \Hcb^q(\iota_{r-p-1};\id)$ if $p$ is odd and $D^{p,q} = 0$ if $p$ is even.\qed
\end{cor}
For the rest of this subsection we fix $r \in [R-1]$ and introduce the following shorthand notations:
\begin{itemize}
\item We write ${}^{{\rm I}}{\rm E}^{\bcdot, \bcdot}$ and ${}^{\rm II}{\rm E}^{\bcdot, \bcdot}$ for ${}_{r+1}^{{\rm I}}{\rm E}^{\bcdot, \bcdot}$ and ${}_{r+1}^{\rm II}{\rm E}^{\bcdot, \bcdot}$ respectively.
\item We write $\HH^{q}_{r-p}$ instead of  $\Hcb^{q}(G_{r-p})$ for all $p \in [r]$ and $q \in \bbN$.
\item We write $\HH^q(\iota_r): \HH^q_{r+1} \to \HH^q_r$ for the map induced by the inclusion $\iota_{r}: G_r \hookrightarrow G_{r+1}$.
\end{itemize}
With this notation, \autoref{thm:abstract_stability_explicit} can then be stated as follows:
\begin{mclaim*}
For any $q \geq 0$, the map $\HH^q(\iota_r): \HH^q_{r+1} \to \HH^q_r$ is 
\begin{itemize}
\item[$(A_q)$] an injection if $q$ is such that $\min\{\widetilde{\gamma}(q,r), \widetilde{\tau}(q,r)\} \geq 0$, and  
\item[$(B_q)$] an isomorphism if $q$ is such that $\min\{\widetilde{\gamma}(q,r), \widetilde{\tau}(q,r)-1\} \geq 0$.
\end{itemize}
\end{mclaim*}
We are going to establish the main claim by induction on $q$. For $q \in [q_0]$ both $(A_q)$ and $(B_q)$ hold by our initial condition. Thus let $q > q_0$, and assume as induction hypothesis that the statements $(A_{q'})$ and $(B_{q'})$ in the claim above hold for any $q' < q$. We then have to show that $(A_q)$ and $(B_q)$ hold. We are going to give a detailed proof of the statement $(B_q)$, and then indicate the necessary replacements to convert this proof into an argument for $(A_q)$. These proofs will occupy the remainder of this section.

\subsubsection{Proof of $(B_q)$}
Assume that $\min\{\widetilde{\gamma}(q,r),\widetilde{\tau}(q,r)-1\} \geq 0$. By the definition of $\widetilde{\gamma}$ and $\widetilde{\tau}$ in \eqref{eq:parameters_tilde}, this is equivalent to saying that
\begin{equation} \label{eq:cond_isom_bis}
	\gamma(r+1-2(q-j)) \geq j \qand \tau(r+1-2(q-j)) \geq j + 1 \vspace{-2pt}
\end{equation}
for all $j=q_0+1,\ldots,q$. For the value $j=q$, the inequalities in \eqref{eq:parameters_tilde} read $\gamma(r+1) \geq q$ and $\tau(r+1) \geq q+1$. Thus, \autoref{thm:IE} and \autoref{thm:cor_IIE} imply respectively that \vspace{-1pt}
\[\begin{array}{rcll}
	\IE_\infty^{q'} &\cong & \HH_{r+1}^{q'} &  \mbox{for all } q' \in [q], \qand \\
	\IIE_1^{p',q'} &\cong & \HH^{q'}_{r-p} & \mbox{for all } p' \in [q+1] \mbox{ and all } q' \geq 0.
\end{array} \vspace{-1pt}\]
We now have to show:

\noindent \textbf{Claim:} The limit term $\IIE_\infty^q$ is isomorphic to $\HH^q_r$. In particular, there is an isomorphism
\[
	\HH^q_{r+1} \cong \IE_\infty^q \cong \IIE_\infty^q \cong \HH^q_r,
\]
which is induced by the inclusion $G_r \hookrightarrow G_{r+1}$.

\begin{figure}[b!]

\[\resizebox{\hsize}{!}{
\xymatrix@R=1pt@C=0.5pt{
 & q' \\
{\scriptstyle q} & & \HH^{q}_{r} && \HH^{q}_{r-1} &&\ast && \ast&& \ast&& && && && && \\
{\scriptstyle q-1} & & \HH^{q-1}_{r} && \HH^{q-1}_{r-1} && \HH^{q-1}_{r-2} && \HH^{q-1}_{r-3} && \ast&& && && && && \\
{\scriptstyle q-2} & &\ast && \HH^{q-2}_{r-1} && \HH^{q-2}_{r-2} && \HH^{q-2}_{r-3} && \HH^{q-2}_{r-3} && && && && && \\
\vdots & & && && && && & \ddots & \\
{\scriptstyle 1} & & && && && && && \HH^{1}_{r-q+2} && \HH^{1}_{r-q+1} && \HH^{1}_{r-q} && \HH^{1}_{r-q-1} &&\ast \\  
{\scriptstyle 0} & & && && && && && \ast&& \HH^0_{r-q+1} && \HH^0_{r-q} && \HH^0_{r-q-1} && \IIE_1^{q+2,0} \\ 
\ar[rrrrrrrrrrrrrrrrrrrrrrr] & & && && && && && && && && && && & \hspace{-8pt} & p' \\
& \ar[uuuuuuuu] & {\scriptstyle 0} && {\scriptstyle 1} && {\scriptstyle 2} &&  {\scriptstyle 3} && {\scriptstyle 4} & \cdots & {\scriptstyle q-2} && {\scriptstyle q-1} && {\scriptstyle q} && {\scriptstyle q+1} && {\scriptstyle q+2}
}}
\]	
\vspace{-10pt}

\caption{Some entries of the first page $\IIE_1^{\bcdot,\bcdot}$}
\label{fig:IIE_1}
\end{figure}

\autoref{fig:IIE_1} below displays the arrangement of those first-page terms in $\IIE_1^{\bcdot,\bcdot}$ that are relevant to verify the claim. According to \autoref{thm:cor_IIE}, the maps between these terms are given as follows. The map at the very left of the $q$th row is always given by the zero map
\begin{equation} \label{eq:toparrow}
\HH^q_{r} \xrightarrow{0} \HH^q_{r-1}
\end{equation}
If $1 \leq q' \leq q-1$ then we consider the following three maps in the $q'$th row:
\[
\IIE_1^{q-q'-1,q'} \to  \IIE_1^{q-q',q'}  \to  \IIE_1^{q-q'+1,q'}  \to  \IIE_1^{q-q'+2,q'} \vspace{20pt}
\]
If we set $p := q-q'$, then $1 \leq p \leq q-1$ and our three relevant maps are given by
\begin{equation} \label{eq:midarrows}
\left\{\begin{array}{ll}
			 \HH^{q-p}_{r-p+1} \xrightarrow{0} \HH^{q-p}_{r-p} \xrightarrow{\HH^{q-p}(\iota_{r-p-1})} \HH^{q-p}_{r-p-1} \xrightarrow{0} \HH^{q-p}_{r-p-2}, & \text{if }p \text{ is odd;} \\
			 \HH^{q-p}_{r-p+1} \xrightarrow{\HH^{q-p}(\iota_{r-p})} \HH^{q-p}_{r-p} \xrightarrow{0} \HH^{q-p}_{r-p-1} \xrightarrow{\HH^{q-p}(\iota_{r-p-2})} \HH^{q-p}_{r-p-2}, & \text{if }p \text{ is even.}  
			\end{array} \right.
\end{equation}
Finally consider the bottom row. Here we have the maps
\begin{equation} \label{eq:bottomarrow}
\left\{\begin{array}{ll}
			\HH^{0}_{r-q+1} \xrightarrow{0} \HH^{0}_{r-q} \xrightarrow{\HH^{0}(\iota_{r-q-1})} \HH^{0}_{r-q-1} \to \IIE_1^{q+2,0}, & \text{if }q \text{ is odd;} \\
			\HH^{0}_{r-q+1} \xrightarrow{\HH^{0}(\iota_{r-q})} \HH^{0}_{r-q} \xrightarrow{0} \HH^{0}_{r-q-1} \to \IIE_1^{q+2,0}, & \text{if }q \text{ is even.}  
			\end{array} \right.
\end{equation}
From the arrangement of terms in \autoref{fig:IIE_1}, one can deduce that the following claims are sufficient to conclude that $\IIE_\infty^q \cong \HH^q_r$:
\begin{enumerate}[(a)]
	\item The map $\HH^{q-p}(\iota_{r-p-1})$ is an isomorphism for $p \in \{1,\ldots,q\}$ odd,
	\item The map $\HH^{q-p}(\iota_{r-p})$ is an isomorphism for $p \in \{1,\ldots,q\}$ even,
	\item The map $\HH^{q-p}(\iota_{r-p-2})$ is an injection for $p \in \{1,\ldots,q-1\}$ even, and
	\item The map $\HH_{r-q-1}^0 \to \IIE_{1}^{\scriptscriptstyle q+2,0}$ is an injection if $q$ is even. 
\end{enumerate} 
Indeed, if all these assertions hold, then by \eqref{eq:midarrows} and \eqref{eq:bottomarrow} then the second page has zeros as indicated in \autoref{fig:IIE_2} below, and by \eqref{eq:toparrow} we have $\IIE_2^{0,q} \cong \HH_r^q$. All the zeros on the second page remain zeros on all later pages, and $\IIE_s^{0,q}$ will always be mapped to one of these zeros. We deduce that
\[
	\IIE_\infty^{q} = \bigoplus_{p+q' = q} \IIE_\infty^{p,q'} \cong \HH_r^{q} \oplus 0 \oplus 0 \oplus \dots \oplus 0 \cong \HH_r^{q}.
\]

\begin{figure}[b!]
\hspace{-25pt} 
\[
\xymatrix@R=1pt@C=0.5pt{
 & q' \\
{\scriptstyle q} & & \HH^{q}_{r} && \ast && \ast && \ast && && && && && \\
{\scriptstyle q-1} & & \ast && 0 && 0 && \ast && && && && && \\
{\scriptstyle q-2} & & \ast && \ast && 0 && 0 && && && && && \\
\vdots & & && && && \ddots & \ddots & \\
{\scriptstyle 1} & & && && && && 0 && 0 && \ast && \ast \\  
{\scriptstyle 0} & & && && && && \ast && 0 && 0 && \ast \\ 
\ar[rrrrrrrrrrrrrrrrrrr] & & && && && && && && && && & \hspace{-8pt} & p' \\
& \ar[uuuuuuuu] & {\scriptstyle 0} && {\scriptstyle 1} && {\scriptstyle 2} &&  {\scriptstyle 3} & \cdots & {\scriptstyle q-1} && {\scriptstyle q} && {\scriptstyle q+1} && {\scriptstyle q+2} 
}
\]
\caption{Some entries of the second page $\IIE_2^{\bcdot,\bcdot}$}
\label{fig:IIE_2}
\end{figure}
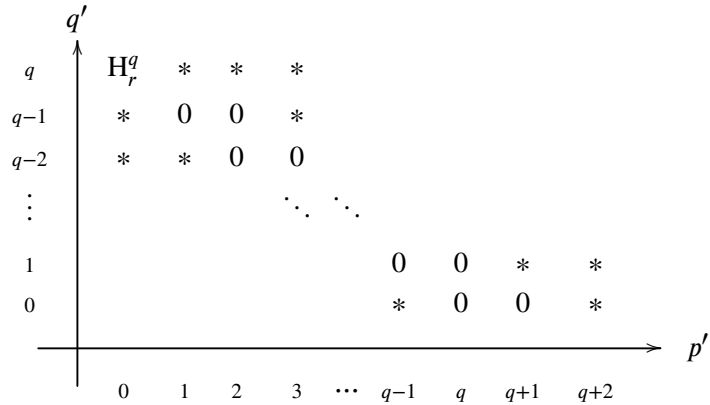
We have thus proved the claim up to establishing (a)--(d). The latter will follow from the induction hypothesis and re-writing the equations \eqref{eq:cond_isom_bis} adequately:

\noindent \emph{Proof of }(a). By the initial condition in \autoref{thm:abstract_stability_explicit}, the map $\HH^{q-p}(\iota_{r-p-1})$ is an isomorphism for every $p \geq q - q_0$, so it remains to show that this map is an isomorphism for all odd $p \in \{1,\ldots,q-q_0-1\}$. Note that $p$ lies in that set if and only if $q-p \in \{q_0 +1, \ldots, q-1\}$. Hence, the claim will follow from the induction hypothesis $(B_{q-p})$ upon verifying that for every odd $p \in \{1,\ldots,q-q_0-1\}$, we have $\min\{\widetilde{\gamma}(q-p, r-p-1),\widetilde{\tau}(q-p, r-p-1)-1\} \geq 0$, or in other words, that 
\[
	\gamma(r-p-2(q-p-j)) \geq j \qand \tau(r-p-2(q-p-j)) \geq j + 1 \vspace{-2pt}
\]
for all $j \in \{q_0+1,\ldots,q-p\}$. Note that 
\[
	j + (p-1)/2 \in \{q_0+1+(p-1)/2,\ldots,q-(p+1)/2\} \subset \{q_0+1,\ldots,q\},
\]
where the number $(p-1)/2$ is an integer because $p$ is odd, and therefore, by \eqref{eq:cond_isom_bis},
\begin{align*}
	\hspace{20pt} \gamma(r-p-2(q-p-j)) &= \gamma\left(r+1-2\big(q-(j+(p-1)/2)\big)\right) \geq j+(p-1)/2 \geq j \qand \\
	\tau(r-p-2(q-p-j)) &= \tau\left(r+1-2\big(q-(j+(p-1)/2)\big)\right) \geq j+1+(p-1)/2 \geq j+1,
\end{align*}
completing the verification. \vspace{5pt}

\noindent \emph{Proof of }(b). As with (a), we need to check that $\min\{\widetilde{\gamma}(q-p,r-p),\widetilde{\tau}(q-p, r-p)-1\} \geq 0$ for every even $p \in \{1,\ldots,q-q_0-1\}$, or in other words, that
\[
	\gamma(r-p+1-2(q-p-j)) \geq j \qand \tau(r-p+1-2(q-p-j)) \geq j + 1 \vspace{-2pt}
\]
for all such $p$ and all $j \in \{q_0+1,\ldots,q-p\}$.  Note that $p/2$ is an integer, $j+p/2 \in \{q_0+1,\ldots,q\}$, and by \eqref{eq:cond_isom_bis},
\begin{align*}
	\hspace{20pt} \gamma(r-p+1-2(q-p-j)) &= \gamma\left(r+1-2\big(q-(j+p/2)\big)\right) \geq j+p/2 \geq j \qand \\
	\tau(r-p+1-2(q-p-j)) &= \tau\left(r+1-2\big(q-(j+p/2)\big)\right) \geq j+1+p/2 \geq j+1. 
\end{align*}
We then conclude by the induction hypothesis $(B_{q-p})$. \vspace{5pt}

\noindent \emph{Proof of }(c). For the injectivity of $\HH^{q-p}(\iota_{r-p-2})$ with $p \in \{1,\ldots,q-1\}$ even, we check must verify that the inequality $\min\{\widetilde{\gamma}(q-p,r-p-1),\widetilde{\tau}(q-p, r-p-1)\} \geq 0$ holds for every even $p \in \{1,\ldots,q-q_0-1\}$. This translates into the collection of inequalities  
\[
	\gamma(r-p-2(q-p-j)) \geq j \qand \tau(r-p-2(q-p-j)) \geq j \vspace{-2pt}
\]
with $j \in \{q_0+1,\ldots,q-p\}$. By the same verification from (b), this is the case, and we conclude by $(A_{q-p})$.
\vspace{5pt}

\noindent \emph{Proof of }(d). Note that $\IIE_1^{q,0} \cong \HH^0(G_{r+1};\Linfty(X_{r+1,q})) \cong \Linfty(X_{r+1,q})^{G_{r+1}}$, that the differential $\IId_1^{q,0}: \IIE_1^{q,0} \to \IIE_1^{q+1,0}$ corresponds under this isomorphism to the coboundary operator $\dd^q$, and that the diagram
\[ \xymatrixcolsep{2pc}
\xymatrix{\Linfty(X_{r+1,q})^{G_{r+1}} \ar[r]^{\dd^{q}} & \Linfty(X_{r+1,q+1})^{G_{r+1}} \\
\bbR \ar[u]_{\sim} \ar@{^{(}->}[ur]} 
\]
commutes, where both maps $\bbR \to \Linfty(X_{r+1,q'})^{G_{r+1}}$ and $\bbR \to \Linfty(X_{r+1,q'+1})^{G_{r+1}}$ are the coefficient inclusions. The former is an isomorphism because of the transitivity of the $G_{r+1}$-action on $X_{r+1,q}$.

\subsubsection{Proof of $(A_q)$}
The proof of $(A_q)$ is completely analogous to the one of $(B_q)$. We go through its main points and omit all estimates that are similar to the ones carried out in the previous subsection. Assume that $\min\{\widetilde{\gamma}(q,r),\widetilde{\tau}(q,r)\} \geq 0$, or equivalently, that 
\begin{equation} \label{eq:cond_inj_bis}
	\gamma\big(r+1-2(q-j)\big) \geq j \qand \tau\big(r+1-2(q-j)\big) \geq j
\end{equation}
hold for all $j=q_0+1,\ldots,q$. As in the previous proof, the value $j=q$ yields the inequalities $\gamma(r+1)\geq q$ and $\tau(r+1) \geq q$. This implies by \autoref{thm:IE} and \autoref{thm:cor_IIE} that \vspace{-1pt}
\[\begin{array}{rcll}
	\IE_\infty^{q'} &\cong & \HH_{r+1}^{q'} &  \mbox{for all } q' \in [q], \qand \\
	\IIE_1^{p',q'} &\cong & \HH^{q'}_{r-p} & \mbox{for all } p' \in [q] \mbox{ and all } q' \geq 0.
\end{array} \vspace{-1pt}\]
The claim to prove now is the following:

\noindent \textbf{Claim:} The limit term $\IIE_\infty^q$ injects into $\HH^q_r$. In particular, we have the inclusion
\[
	\HH^q_{r+1} \cong \IE_\infty^q \cong \IIE_\infty^q \hookrightarrow \HH^q_r,
\]
which is induced by the map $G_r \hookrightarrow G_{r+1}$.
\begin{figure}[b!]
\[
 \xymatrix@R=1pt@C=0.5pt{
   & q' \\
 {\scriptstyle q} & & \HH^{q}_{r} && \HH^{q}_{r-1} && \ast && \ast && && && && && && && \\
 {\scriptstyle q-1} & & \HH^{q-1}_{r} && \HH^{q-1}_{r-1} && \HH^{q-1}_{r-2} && \ast && && && && && && && \\
 {\scriptstyle q-2} & & \ast && \HH^{q-2}_{r-1} && \HH^{q-2}_{r-2} && \HH^{q-2}_{r-3} && && &&  && && && \\
 \vdots & & && && && & \ddots & && \\
 {\scriptstyle 1} & & && && && && \HH^{1}_{r-q+2} && \HH^{1}_{r-q+1} && \HH^{1}_{r-q} && \ast \\  
 {\scriptstyle 0} & & && && && && \ast && \HH^0_{r-q+1} && \HH^0_{r-q} && \IIE_1^{q+1,0} && \\ 
 \ar[rrrrrrrrrrrrrrrrrrr] & & & & && && && && && && && & \hspace{-8pt} & p' \\
 & \ar[uuuuuuuu] & {\scriptstyle 0} && {\scriptstyle 1} && {\scriptstyle 2} &&  {\scriptstyle 3} & \cdots & {\scriptstyle q-2} && {\scriptstyle q-1} && {\scriptstyle q} && {\scriptstyle q+1}}
\]
\caption{Some entries of the first page $\IIE_1^{\bcdot,\bcdot}$}
\label{fig:IIE_1_inj}
\end{figure}

\autoref{fig:IIE_1_inj} below displays the arrangement in the first page $\IIE_1^{\bcdot,\bcdot}$ that are relevant to the proof of this claim. As a consequence of \autoref{thm:cor_IIE}, the maps between these terms are given as follows. As in \eqref{eq:toparrow}, the map at the left of the $q$-th row is always the zero map
\begin{equation} \label{eq:toparrow_inj}
\HH^q_{r} \xrightarrow{0} \HH^q_{r-1}
\end{equation}
If $1 \leq q' \leq q-1$, then we consider the two maps in the $q'$-th row:
\[
\IIE_1^{q-q'-1,q'} \to  \IIE_1^{q-q',q'}  \to  \IIE_1^{q-q'+1,q'}
\]
Setting $p := q-q'$, we have $1 \leq p \leq q-1$ and our relevant maps are given by
\begin{equation} \label{eq:midarrows_inj}
\left\{\begin{array}{ll}
			 \HH^{q-p}_{r-p+1} \xrightarrow{0} \HH^{q-p}_{r-p} \xrightarrow{\HH^{q-p}(\iota_{r-p-1})} \HH^{q-p}_{r-p-1}, & \text{if }p \text{ is odd;} \\
			 \HH^{q-p}_{r-p+1} \xrightarrow{\HH^{q-p}(\iota_{r-p})} \HH^{q-p}_{r-p} \xrightarrow{0} \HH^{q-p}_{r-p-1}, & \text{if }p \text{ is even.}  
			\end{array} \right.
\end{equation}
In the bottom row, we have
\begin{equation} \label{eq:bottomarrow_inj}
\left\{\begin{array}{ll}
			\HH^{0}_{r-q+1} \xrightarrow{0} \HH^{0}_{r-q} \to \IIE_1^{q+1,0}, & \text{if }q \text{ is odd;} \\
			\HH^{0}_{r-q+1} \xrightarrow{\HH^{0}(\iota_{r-q})} \HH^{0}_{r-q} \to \IIE_1^{q+1,0}, & \text{if }q \text{ is even.}  
			\end{array} \right.
\end{equation}

Our claim will follow if the following conditions are fulfilled:
\begin{enumerate}[(a')]
	\item The map $\HH^{q-p}(\iota_{r-p-1},\id)$ is an injection for $p \in \{1,\ldots,q-1\}$ odd,
	\item The map $\HH^{q-p}(\iota_{r-p},\id)$ is an isomorphism for $p \in \{1,\ldots,q\}$ even, and
	\item The map $\HH_{r-q}^0 \to \IIE_{1}^{\scriptscriptstyle q+1,0}$ is an injection if $q$ is odd. 
\end{enumerate} 
\begin{figure}[!]
\hspace{-25pt} 
\[
 \xymatrix@R=1pt@C=1pt{
   & q' \\
 {\scriptstyle q} & & \HH^{q}_{r} && \ast && \ast && && && && && && \\
 {\scriptstyle q-1} & & \ast && 0 && \ast && && && && && && && \\
 {\scriptstyle q-2} & & \ast && \ast && 0 && && && && && \\
 \vdots & & && && & \ddots & && \\
 {\scriptstyle 1} & & && && && 0 && \ast && \ast \\  
 {\scriptstyle 0} & & && && && \ast && 0 && \ast && \\ 
 \ar[rrrrrrrrrrrrrrr] & & & & && && && && && & \hspace{-8pt} & p' \\
 & \ar[uuuuuuuu] & {\scriptstyle 0} && {\scriptstyle 1} && {\scriptstyle 2} & \cdots & {\scriptstyle q-1} && {\scriptstyle q} && {\scriptstyle q+1}}
\]
\caption{Some entries of the second page $\IIE_2^{\bcdot,\bcdot}$}
\label{fig:IIE_2_inj}
\end{figure}
Indeed, by \eqref{eq:midarrows_inj} and \eqref{eq:bottomarrow_inj} then the second page has zeros as indicated in \autoref{fig:IIE_2_inj} above, and by \eqref{eq:toparrow_inj} we have $\IIE_2^{0,q} \cong \HH_r^q$. All the zeros on the second page remain zeros on all later pages. For every $s \geq 2$, the target of the differential $\IIE_s^{0,q} \to \IIE_s^{s,q-s+1}$ will be a non-zero map, since we have no control of the term $\IIE_s^{s,q-s+1}$. This means that as a kernel of that differential, $\IIE_{s+1}^{0,q}$ injects into $\IIE_s^{0,q}$. In the limit page, we have therefore
\[
 	\IIE_{\infty}^{0,q} \hookrightarrow \IIE_2^{0,q} \cong \HH^q_r,
\]
and thus
\[
	\IIE_\infty^{q} = \bigoplus_{p'+q' = q} \IIE_\infty^{p',q'} \hookrightarrow \HH_r^{q} \oplus 0 \oplus 0 \oplus \dots \oplus 0 \cong \HH_r^{q}.
\]
The proof of the assertion (c') follows exactly as statement (d) in the proof of $(B_q)$. Assertions (a') and (b') are consequences of the induction hypothesis upon re-writing adequately the inequalities \eqref{eq:cond_inj_bis}, in an analogous way to the proof of the conditions (a)-(c) in the proof of $(B_q)$. The verification is left to the reader.

\newpage
\appendix

\section{Monod's category of coefficient $G$-modules} \label{sec:monodcoeff}
The entire content of this appendix has been extracted from \cite{Monod-Book}; we refer to it for either proofs of the statements or further references to their proofs. 

Given a locally compact, second-countable (lcsc) topological group $G$ we define the category ${\rm Ban}_G$ as follows: Objects in ${\rm Ban}_G$ are separable Banach spaces over the field of real numbers equipped with a jointly continuous $G$-action by bounded linear isometries; they will be referred to as \emph{separable, continuous Banach $G$-modules}. Morphisms in ${\rm Ban}_G$ are given by $G$-equivariant bounded linear maps, here referred to as \emph{$G$-morphisms}.

\begin{defn} \emph{Monod's category of coefficient $G$-modules} is the opposite category ${\rm Ban}_G^{\rm op}$ of ${\rm Ban}_G$.
\end{defn}

\begin{rem}\label{DefCoefficientModule}
We will use the following concrete model for ${\rm Ban}_G^{\rm op}$: Given a Banach $G$-module $B$, denote by $B^\sharp$ the dual Banach space of $B$ equipped with the contragredient $G$-action. Note that $B^\sharp$ is in general not an object of ${\rm Ban}_G$, since neither needs it be separable, nor needs the contragredient action be continuous. We will refer to the pair $(B, B^\sharp)$ as a \emph{coefficient $G$-module}. Now let $(B^\flat, B)$ and $(C^\flat, C)$ be a pair of coefficient $G$-modules. Then a \emph{dual morphism} $(B^\flat, B) \to (C^\flat, C)$ is a weak-$\ast$ continuous $G$-morphism $f: B \to C$, or equivalently, a $G$-morphism which is dual to a $G$-morphism $f^\flat: C^\flat \to B^\flat$. We may then define ${\rm Ban}_G^{\rm op}$ as the category whose objects are coefficient $G$-modules and whose morphisms are dual morphisms. 
\end{rem}
We recall at this point \autoref{defn:regspace:Intro}, which allows us to produce some examples of coefficient modules.
\begin{defn} \label{defn:regspace}
A \emph{regular $G$-space} is a standard Borel space $X$, endowed with a Borel $G$-action and a Borel probability measure $\mu$ on $X$ that is $G$-quasi-invariant in the sense that $g_\ast \mu \sim \mu$ for all $g\in G$. We denote by ${\rm Reg}_G$ the category whose objects are regular $G$-spaces and morphisms are Borel $G$-maps. 
\end{defn}

\begin{exmpl}\label{LInfty}
If $(X, \mu)$ is a regular $G$-space, then the pair $(L^1(X), L^\infty(X))$ is a coefficient $G$-module, where the action of $G$ on $L^\infty(X)$ is defined by $g.f(x) := f(g^{-1}x)$, and the one on $L^1(X)$ is given by the formula 
\[
g.f(x) := \rho(g, x) f(g^{-1}x),
\]
where $\rho(g, x) := d(g^{-1}_*\mu)/d\mu(x)$ denotes the Radon-Nikodym cocycle of $\mu$. See also Appendix D in \cite{Buehler}.
\end{exmpl}

We will need the following generalization of this example.
\begin{exmpl}\label{LInftyGeneral}
Let $(X, \mu)$ is a regular-$G$-space, let $(B^\flat, B)$ be a coefficient $G$-module and define
\[
 L^\infty(X;B):=\{\phi: X \to B \: : \: \phi \mbox{ is weak-* Borel and essentially bounded}\}/\sim,
\]
where $\sim$ denotes $\mu$-almost everywhere equality. Equipped with the essential supremum norm, this is a Banach space. We endow it with the $G$-action defined by the formula $g.f(x) = g.f(g^{-1}x)$ as in the previous example. Then
\[
L^\infty(X; B) \cong (L^1(X) \hat{\otimes} B^\flat)^\sharp,
\]
as objects of ${\rm Ban}_G$, where $\hat{\otimes}$ denotes the projective tensor product of Banach spaces and $L^1(X) \hat{\otimes} B^\flat$ is endowed with its natural $G$-action. Consequently, the pair $(L^1(X) \hat{\otimes} B^\flat, L^\infty(X, B))$ is a coefficient $G$-module. 
\end{exmpl}
The previous construction satisfies the following so-called exponential law:
\begin{lem} \label{thm:explaw}
	For regular $G$-spaces $(X_1,\mu_1)$ and $(X_2,\mu_2)$ and a coefficient $G$-module $(B^\flat, B)$, there is an isomorphism
\[
	\Linfty(X_1 \times X_2; B) \cong \Linfty(X_1; \Linfty(X_2; B))
\]	
of coefficient $G$-modules.\qed
\end{lem}
We have the following notion of exactness in ${\rm Ban}_G^{\rm op}$:
\begin{defn}\label{Exactness}
Given a coefficient $G$-module $(B^\flat, B)$ we refer to $B$ as the \emph{underlying vector space}. We say that a sequence 
\begin{equation} \label{eq:SeqCoeffMod1}
	0 \to ((B^0)^\flat, B^0) \to ((B^1)^\flat, B^1) \to ((B^2)^\flat, B^2) \to \cdots
\end{equation}
of coefficient modules and dual morphisms is a \emph{complex}, resp. \emph{exact}, if the underlying sequence 
\begin{equation} \label{eq:SeqCoeffMod2}
	0 \to  B^0 \to B^1 \to B^2 \to \dots
\end{equation}
of vector spaces has the corresponding property. In this case, in order to avoid an overloaded notation, we will omit mention of the preduals and refer simply to \eqref{eq:SeqCoeffMod2} and not to \eqref{eq:SeqCoeffMod1} as the complex, resp. exact sequence of coefficient modules whenever it is affordable. 
\end{defn}
The following lemma is our main technical tool; here ${\rm Vect}$ denotes the category of vector spaces.
\begin{lem} \label{thm:exactftr}
 The functor ${\rm Ban}_G^{\rm op} \to {\rm Vect}$ given by $(A^\flat, A) \mapsto L^\infty(G^n, A)^G$ is exact, i.e. if $0\to A \to B \to C \to 0$ is a short exact sequence in ${\rm Ban}_G^{\rm op}$, then 
\[
	0 \to \Linfty(G^n; A)^G \to \Linfty(G^n; B)^G \to \Linfty(G^n; C)^G \to 0
\]
is an exact sequence of vector spaces.\qed
\end{lem}
\section{Topology of symplectic Grassmannians} \label{sec:symplectic}
\subsection{Symplectic bases}\label{SympBase}

Let $\kay \in \{\bbR, \bbC\}$ and let $V$ be a $\kay$-vector space. A \emph{symplectic form} $\omega:V\times V \to \kay$ is a non-degenerate alternating bilinear form, in which case $(V,\omega)$ is called a \emph{symplectic vector space}. A linear map between symplectic vector spaces is called \emph{symplectic} if it intertwines the given symplectic forms; we 
denote by $\Sp(V, \omega)$ the group of linear symplectic automorphisms of $(V,\omega)$. It is a simple closed algebraic subgroup of $\GL(V)$. 

If $(V, \omega)$ is a symplectic vector space, then $V$ is even-dimensional and there exists a basis $(e_{r-1}, \dots, e_0, f_0, \dots, f_{r-1})$ of $\kay^{2r}$ such that  
\[\omega(e_i, f_j) = \delta_{ij} \qand \omega(e_i, e_j) = \omega(f_i, f_j) = 0 \text{ for all }i, j \in \{0, \dots, r-1\}.\]
In any such basis $\omega$ is represented by the matrix
\[
J = \left(\begin{matrix}0    & Q \\
	-Q & 0\end{matrix}\right), \text{ where } Q = \left(\begin{matrix} &&1\\&\iddots&\\1&& \end{matrix}\right),
\]
hence we refer to $(e_{r-1}, \dots, e_0, f_0, \dots, f_{r-1})$ as a \emph{symplectic basis in antidiagonal form}. Note that if $\mathcal B := (e_{r-1}, \dots, e_0, f_0, \dots, f_{r-1})$ 
and $\mathcal B' := (e'_{r-1}, \dots, e'_0, f'_0, \dots, f'_{r-1})$ are symplectic bases in antidiagonal form for symplectic vector spaces $(V, \omega)$ and $(V', \omega')$ respectively, then there exists a unique linear symplectic isomorphism $A: V \to V'$ which maps $\mathcal B$ to $\mathcal B'$. In particular, there is a precisely one symplectic vector space of dimension $2r$ over $\kay$ up to symplectic linear isomorphism. Let us denote by $(\kay^{2r}, \omega)$ the representative of this isomorphism class, in which the standard basis is in antidiagonal form; we then also write $\Sp_{2r}(\kay)$ for the corresponding automorphism groups.

A subspace $W$ of a symplectic vector space $(V, \omega)$ is called \emph{symplectic} if $\omega$ is non-degenerate on $W$. If $(e_{r-1}, \dots, e_0, f_0, \dots, f_{r-1})$ is a symplectic basis in $(\kay^{2r}, \omega)$ in antidiagonal form, then $V_k :={\rm span}(e_{k}, \dots, e_0, f_0, \dots, f_{k})$ is symplectic for all $0 \leq k \leq r-1$. In particular we obtain embeddings
\begin{equation}\label{SymplecticEmbeddings}
(\kay^0,0) \hookrightarrow (\kay^2, \omega) \hookrightarrow (\kay^4, \omega) \hookrightarrow (\kay^6, \omega) \hookrightarrow \dots \hookrightarrow  (\kay^{2r}, \omega) \hookrightarrow \cdots.
\end{equation}
by sending $(v_{r-1}, \ldots, v_0, w_0, \ldots, w_{r-1})^\top$ to $(0, v_{r-1}, \ldots, v_0, w_0, \ldots, w_{r-1}, 0)^\top$. Moreover, every $A \in {\rm Sp}_{2r}(\kay)$ extends to an automorphism $\widetilde{A} \in {\rm Sp}_{2(r+1)}(\kay)$ acting trivially on ${\rm span}(e_{r}, f_{r})$. This defines embeddings
\begin{equation}\label{SpEmbeddings}
1 \hookrightarrow {\rm Sp}_2(\kay) \hookrightarrow{\rm Sp}_4(\kay) \hookrightarrow{\rm Sp}_6(\kay) \hookrightarrow \dots\hookrightarrow {\rm Sp}_{2r}(\kay) \hookrightarrow \cdots, 
\end{equation}
which on the level of matrices (with respect to the standard bases) are given by.
\[
A \mapsto \widetilde{A} = \left(\begin{matrix}1&&\\&A&\\&&1\end{matrix}\right).
\]
Throughout this article we will always use these embeddings when considering $ {\rm Sp}_{2r}(\kay)$ as a subgroup of ${\rm Sp}_{2(r+1)}(\kay)$.
 
\subsection{Symplectic complements and isotropic subspaces} \label{SympComp}
Given a symplectic vector space $(V,\omega)$ and $v, w\in V$, we write $v \perp_\omega w$ provided $\omega(v,w) = 0$, and given subsets $S, T \subset V$ we write $S \perp_\omega T$ and say that $S$ and $T$ are \emph{(symplectically) perpendicular} if $v\perp_\omega w$ for all $v \in S$ and $w \in T$. We also define the \emph{symplectic complement} $S^\omega$ of a subset $S\subset V$ as
\[
S^\omega := \{v \in V \mid \{v\} \perp_\omega S\}.
\]
This complement is always a linear subspace of $V$ since $S^\omega = ({\rm span}(S))^\omega$, and for any linear subspace $W < V$ we have equalities 
\begin{equation} \label{eq:sympl_complement}
(W^\omega)^\omega = W	\qand \dim W + \dim W^\omega = \dim V.
\end{equation}

A linear subspace $W \subset V$ is then \emph{symplectic} if $W \cap W^\omega = \{0\}$; it is called \emph{isotropic} if $W \subset W^\omega$. Every one-dimensional subspace of $V$ is isotropic, and an isotropic subspace $W$ is maximal in $V$ with respect to inclusion if and only if $W=W^\omega$; in this case $\dim W = 1/2 \dim V$ and $W$ is called \emph{Lagrangian}. If $(e_{r-1}, \dots, e_0, f_0, \dots, f_{r-1})$ is a symplectic basis of $V$ in antidiagonal form, then $(e_{r-1}, \dots, e_0)$ and $(f_0, \dots, f_{r-1})$ are Lagrangian, and $(e_{k},\dots, e_0, f_0, \dots, f_k)$ is symplectic for every $0\leq k \leq r-1$.

\begin{lem}[{\cite[Lemma 2.1.5]{McDuffSalamon}}]\label{McDS} Every isotropic subspace is contained in a Lagrangian, and every basis of a Lagrangian can be extended to a symplectic basis in antidiagonal form.\qed
\end{lem}

\begin{prop}\label{thm:transitivity}  Let $(V, \omega)$ be a symplectic vector space of dimension $2r$ and let $0 \leq k \leq r-1$. Then the group ${\rm Sp}(V,\omega)$ acts transitively on bases of $(k+1)$-dimensional isotropic subspaces of $V$.
\end{prop}
\begin{proof} Let $W\subset V$ be a $(k+1)$-dimensional isotropic subspace, and let $(w_0, \dots, w_k)$ be a basis of $W$. By Lemma \ref{McDS}, $W$ is contained in a Lagrangian subspace $L$ and we can extend $(w_0, \dots, w_k)$ to a basis $(w_0, \dots, w_{r-1})$ of $L$, and then further to a symplectic basis $(v_{r-1}, \dots, v_{0}, w_{0}, \dots, w_{r-1})$ in antidiagonal form by the same lemma. Since $\Sp(V, \omega)$ acts transitively on such bases, the corollary follows.
\end{proof}

The following lemma will be useful in Section \ref{sec:contractibility}. If $a,b \in \bbP(V)$ are 1-dimensional subspaces of a symplectic vector space $(V,\omega)$, we say that $\omega(a,b) = 0$ (resp. $\omega(a,b) \neq 0$) whenever $\omega(v,w)$ equals zero (resp. does not equal zero) for two non-zero vectors $v \in a$ and $w \in b$. Note that this property is independent of the choices of $v$ and $w$ within $a$ and $b$, respectively. 
\begin{lem} \label{thm:lemma_genericity}
Let $(V,\omega)$ be a symplectic vector space, and let $a_0,\ldots,a_k,b_0,\ldots,b_k \in \bbP(V)$ be 1-dimensional subspaces of $V$ such that 
\[ \begin{array}{l}
	\omega(a_i,b_i) \neq 0 \, \mbox{ for all } i \in [k], \quad \omega(a_i,a_j) = 0 \, \mbox{ for all } i,j \in [k], \qand \vspace{2pt} \\
	\omega(a_i,b_j) = 0 \mbox{ for all } i,j \in [k], i \neq j.
\end{array} \]
Then $W:= \Span\{a_0,\ldots,a_k,b_0,\ldots,b_k\}$ is a symplectic subspace of $V$ of dimension $2(k+1)$. In particular, $\dim V \geq 2(k+1)$. 
\end{lem}
\begin{proof}
We prove the lemma by induction on $k$, where the base case $k=0$ is immediate. We take the statement of the lemma for $k-1$ as our induction hypothesis, and prove it for $k$. Thus, let $a_0,\ldots,a_k,b_0,\ldots,b_k$ be 1-dimensional subspaces as in the assumptions. By the induction hypothesis, the subspace $W_0 := \Span\{a_0,\ldots,a_{k-1},b_0,\ldots,b_{k-1}\}$ is symplectic and of dimension $2k$. In particular, $V$ splits as the direct sum $V = W_0 \oplus W_0^\omega$. 

By assumption, $a_k \subset W_0^\omega$. Now let $v_k$ and $w_k$ be non-zero vectors in $a_k$ and $b_k$, respectively, and let $w_{k,0} \in W_0$ and $w_{k,1} \in W_0^\omega$ be vectors such that $w_k = w_{k,0} + w_{k,1}$. Then 
\[
	0 \neq \omega(v_k,w_k) = \omega(v_k,w_{k,0})+\omega(v_k, w_{k,1})=\omega(v_k, w_{k,1}).
\]
Hence, $\Span\{v_k,w_{k,1}\} \subset W_0^\omega$ is symplectic and of dimension two, and in consequence, $W=\Span(W_0 \cup  \{v_k,w_{k,1}\})$ is, too, symplectic and of dimension $2(k+1)$. 
\end{proof}

\subsection{Topologies on Grassmannians}
If $X$ is an arbitrary lcsc space, then the set $\mathcal{C}(X)$ of closed subsets admits a compact metrizable topology, called the \emph{Chabauty topology}, in which convergence of sequences can be characterized as follows \cite[Proposition E.1.2]{Bened-Petro}: Given subsets $W_m, W \in \mathcal{C}(X)$, we have convergence $W_m \to W$ if and only if the following two conditions hold:
\begin{enumerate}[(C1)]
	\item If $(w_{m_k})_k$ is a sequence in $X$ such that $w_{m_k} \in W_{m_k}$ for every $k$ and such that $w_{m_k} \to w \in X$ as $k \to \infty$, then $w \in W$.
	\item If $w \in W$, then there exists a sequence $(w_m) \subset G$ with $w_m \in W_m$ and such that $w_m \to w$. 
\end{enumerate}
It is immediate from this characterization of convergence that if $V$ happens to be a finite-dimensional topological vector space over $\kay \in \{\bbR, \bbC\}$, then the subset ${\rm Gr}(V) \subset \mathcal C(V)$ of linear subspaces is closed, and that ${\rm GL}(V)$ acts continuously on ${\rm Gr}(V)$. Given $0 \leq k \leq \dim V$, denote by ${\rm Gr}_k(V)$ the $k$-Grassmannian of $V$, i.e. the subset of ${\rm Gr}(V)$ consisting of $k$-dimensional linear subspaces, equipped with the restriction of the Chabauty topology. We observe, firstly, that 
${\rm GL}(V)$ acts transitively on ${\rm Gr}_k(V)$ and, secondly, that ${\rm Gr}_k(V)$ is closed in ${\rm Gr}(V)$, hence compact. Since
\begin{equation}\label{Grassmann}
{\rm Gr}(V) = \bigsqcup_{k=0}^{\dim V}{\rm Gr}_k(V)
\end{equation}
is a finite disjoint union, we deduce that all of the subsets ${\rm Gr}_k(V) \subset {\rm Gr}(V)$ are clopen; in fact they are precisely the connected components of ${\rm Gr}(V)$, since they are connected by the following lemma.
\begin{lem}\label{ChabautyQuotientTop} The Chabauty topology on ${\rm Gr}_k(V)$ coincides with the quotient topology with respect to the ${\rm GL}(V)$-action.
\end{lem}
\begin{proof} Since ${\rm GL}(V)$ acts continuously and transitively on ${\rm Gr}_k(V)$, this follows from the open mapping theorem for homogeneous spaces; see e.g. \cite{Koshi-Takesaki}.
\end{proof}
The following provides yet another description of the same topology.
\begin{prop} \label{thm:chabauty_conv}
For a fixed $l$, let $(W_m)$ be a sequence in $\Gr_{l}(V)$, and $W \in \Gr_{l}(V)$. Then $W_m \to W$ as $m \to \infty$ if and only if Property (C1) above holds. 
\end{prop}
\begin{proof}
We only need to show the sufficiency of property (C1). Thus assume $(W_m)$ and $W$ satisfies (C1) and let $W_{m_k}$ be a convergent subsequence, say $\lim_{k \to \infty} W_{m_k} = W'$. If $w' \in W'$, then by (C2) we find $w_{m_k} \in W_{m_k}$ with $w_{m_k} \to w'$, but then $w' \in W$ since $(W_m)$ and $(W)$ satisfy (C1) and hence $W' \subset W$. However, since $\dim W = \dim W' = l$ we deduce that $W = W'$, and since $(W_{m_k})$ was arbitrary and ${\rm Gr}_l(V)$ is compact we deduce that $W_m \to W$.
\end{proof}
Here are two applications that we will use in the main part of this article.
\begin{lem}\label{IntersectionCts} For every subspace $W \subset V$ the map ${\rm Gr}(V) \to {\rm Gr}(V)$, $U \mapsto U \cap W$ is continuous.
\end{lem}
\begin{proof} Assume $U_n \to U$ in ${\rm Gr}(V)$ and assume that $u_{m_k} \in U_{m_k} \cap W$ and $u_{m_k} \to u$. Then $u_{m_k} \in U$ (by (C1), since $U_n \to U$) and $u_{m_k} \in W$, since $W$ is closed. Thus $u \in U \cap W$ and hence $U_n \cap W \to U \cap W$ by \autoref{thm:chabauty_conv}.
\end{proof}
\begin{lem}\label{SpanCts} Let $0 \leq l\leq k \leq \dim V-1$ and set
\[
X_{k,l} := \{ (p_0, \dots, p_{k}) \in \mathbb P(V)^{k+1} \mid \dim{\rm span}(p_0, \dots, p_k) = l+1\}.
\]
Then the map
\[
{\rm span}: X_{k,l} \to {\rm Gr}_{l+1}(V), \quad (p_0, \dots, p_k) \mapsto {\rm span}(p_0, \dots, p_k) 
\]
is continuous. In particular, the map ${\rm span}: \mathbb P(V)^{k+1} \to {\rm Gr}(V)$ is Borel measurable.
\end{lem}
\begin{proof} Assume first that $l = k$. Since the group ${\rm GL}(V)$ acts continuously and transitively on $X_{k,k}$, the topology on $X_{k,k}$ coincides with the quotient topology with respect to an orbit map ${\rm GL}(V) \to X_{k,k}$. Since the topology on ${\rm Gr}_{k+1}(V)$ is also a quotient topology by Lemma \ref{ChabautyQuotientTop}, we deduce that the map ${\rm span}: X_{k,k} \to {\rm Gr}_{k+1}(V)$ is a quotient map, in particular continuous.

Now let $0 \leq l\leq k \leq \dim V-1$ be arbitrary and assume that $(p_0^{(n)}, \dots, p_{k}^{(n)}) \to (p_0, \dots, p_k)$ in $X_{k,l}$. Then there exist indices $i_0, \dots, i_l$ such that 
 $(p_{i_0}, \dots, p_{i_l})$ are linearly independent. Since linear independence is an open condition, we deduce that  $(p_{i_0}^{(n)}, \dots, p_{i_l}^{(n)})$ are linearly independent for all sufficiently large $l$, hence $(p_{i_0}^{(n)}, \dots, p_{i_l}^{(n)}) \to (p_{i_0}, \dots, p_{i_l})$ in $X_{l,l}$. From the previous case we thus deduce that
\[
 {\rm span}(p_{i_0}^{(n)}, \dots, p^{(n)}_{i_k}) = {\rm span}(p_{i_0}^{(n)}, \dots, p_{i_l}^{(n)}) \to {\rm span}(p_{i_0}, \dots, p_{i_l}) = {\rm span}(p_{i_0}, \dots, p_{i_k}), 
\]
 which establishes the desired continuity.
\end{proof}
\subsection{Continuity of symplectic polarities} We return to the case where $(V, \omega)$ is a symplectic vector space. By \eqref{eq:sympl_complement} we then have an  ${\rm Sp}(V, \omega)$-equivariant involution
\[
(-)^\omega:  \Gr(V) \to \Gr(V), \quad W \mapsto W^\omega,
\]
called the \emph{symplectic polarity}. The purpose of this subsection is to show that taking symplectic complements in a symplectic vector space defines a continuous involution on the corresponding Grassmannian with respect to the various equivalent topologies discussed above.
\begin{lem} \label{thm:omega_Borel}
The symplectic polarity $(-)^\omega:  \Gr(V) \to \Gr(V)$ is continuous. 
\end{lem}
\begin{proof} It suffices to show that the restriction $(-)^\omega: \Gr_l(V) \to \Gr_{\dim(V)-l}(V)$ is continuous. Let $(W_m)$ be a sequence in $\Gr_l(V)$ converging to $W \in \Gr_l(V)$. We show that $W^\omega_m \to W^\omega$ as $m \to \infty$ by using \autoref{thm:chabauty_conv}. Let $x_{m_j} \in W^\omega_{m_j}$ and assume that $x_{m_j} \to x \in V$ as $j \to \infty$, and let $w \in W$. By the condition (C2) of Chabauty convergence, our assumption $W_m \to W$ implies that there exists a sequence $w_{m_j} \in W_{m_j}$ such that $w_{m_j} \to w$. But then 
\[
	\omega(x,w) = \lim_{j \to \infty} \omega\big(x_{m_j},w_{m_j}\big) = 0.
\]
Since the choice of $w \in W$ was arbitrary, it follows that $x \in W^\omega$, which finishes the proof. 
\end{proof}

\subsection{The symplectic Grassmannian}
Let $(V, \omega)$ be a symplectic vector space of dimension $2r$. For every $l \in \{0, \dots, r-1\}$ we define a subset $\clG_l $ of the Grassmannian ${\rm Gr}_{l+1}(V)$ by
\begin{equation}\label{SymplecticGrassmannian}
\clG_l(V, \omega) := \{W \in {\rm Gr}_{l+1}(V) \mid W \text{ is isotropic}\}.
\end{equation}
We refer to $\clG_l(V, \omega)$ as the \emph{symplectic Grassmannian} of \emph{type $l$}. In terms of incidence geometry, these are the points, lines, etc. of the polar geometry associated with $(V, \omega)$ (hence the shift in enumeration). We set $\clG(V, \omega) := \sqcup \clG_l(V, \omega)$ and we equip $\clG(V, \omega) \subset {\rm Gr}(V)$ with the subspace topology. By \autoref{thm:transitivity} the group ${\rm Sp}(V, \omega)$ acts transitively on $\clG_l(V, \omega)$ for all $l\in \{0, \dots, r-1\}$.
\begin{lem}\label{lem:symp_grass}
\begin{enumerate}[(i)]
\item $\clG(V, \omega)$ is compact, and all $\clG_l(V, \omega)$ are compact.
\item The subsets $ \clG_l(V, \omega) \subset \clG(V, \omega)$ are the connected components of $ \clG(V, \omega)$.
\item On each $\clG_l(V, \omega)$ the subspace topology coincides with the quotient topology from ${\rm Sp}(V, \omega)$.
\end{enumerate}
\end{lem}
\begin{proof} (iii) follows again from the open mapping theorem for homogeneous spaces \cite{Koshi-Takesaki}, since the action of ${\rm Sp}(V, \omega) < {\rm GL}(V)$ is continuous and transitive. This in turn implies that the $\clG_l(V, \omega)$ are connected, and being clopen we deduce that they are the connected components of $ \clG(V, \omega)$, and it remains to show only that $\clG(V, \omega)$ is closed in ${\rm Gr}(V)$. However, by \autoref{thm:omega_Borel} the symplectic polarity is continuous, and it follows that the condition $W \subset W^\omega$ is a closed condition.
\end{proof}

\section{Tools from measure theory}\label{sec:measuretheory}
In this appendix, we record some basic facts and notions from measure theory that we will use in the proofs of \autoref{thm:nu_Borel}, \autoref{thm:admissibility} and \autoref{MagicLemma}. The material of this appendix is completely standard and only compiled here for ease of reference. Throughout this appendix we fix locally compact second-countable (lcsc) spaces $X$ and $Y$ and consider a map
\[\sigma: Y \to {\rm Prob}(X), y \mapsto \sigma_y,\]
which is measurable with respect to the Borel structures on $Y$ and ${\rm Prob}(X)$ (see \autoref{NoncompactMeasures}).

\subsection{Integrating measures}\label{subsec:IntegralMeasure}
We recall the construction of an integral of a family of measures.
\begin{lem}\label{IntegralOfMeasures} For every $\mu \in {\rm Prob}(Y)$ the assignment
\[
\sigma_*\mu: C_c(X) \to \bbR, \quad f \mapsto  \int_Y \left(\int_X f \,\dd\!\sigma_y\right) \,\dd\!\mu(y)
\]
defines a probability measure on $X$.
\end{lem}
\begin{proof} Assume first that $X$ is compact. Let $f \in C(X)$ and let $g: Y \to \bbR$ be given by $g(y) := \sigma_y(f)$. Since $\nu$ is Borel measurable, the function $g$ is Borel measurable, and since each $\nu_y$ is a probability measure, it is moreover bounded with $|g(y)| = |\mu_y(f)| \leq \|f\|_\infty$. This implies that
$\sigma_*\mu(f) = \int g \,\dd\!\mu$ exists, and since $\mu$ is a probability measure we have
\[
\left|\sigma_*\mu(f)\right|  =  |\mu(g)| \leq \|g\|_\infty \leq \|f\|_\infty.
\]
Since this holds for all $f \in C(X)$, the linear functional $\sigma_*\mu$ is bounded. Positivity of $\sigma_*\mu$ follows from positivity of the measures $\sigma_y$ and of $\mu$, and finally
\[
\sigma_*\mu(X) = \int \sigma_y(X) \,\dd\!\mu(y) = \int \,\dd\!\mu(y) = 1.
\]
If $X$ is non-compact and $X^+:= X \cup \{\infty\}$ denotes its one-point compactification, then apply this to the map $\sigma: Y \to {\rm Prob}(X) \hookrightarrow {\rm Prob}(X^+)$ to obtain a probability measure $\sigma_*\mu$ on $X^+$ and observe that $\sigma_*\mu(\{\infty\}) = 0$ since $\sigma_y(\{\infty\}) = 0$ for all $y \in Y$.
\end{proof}
In view of the defining formula we refer to $\sigma_*\mu$ as the \emph{integral} of the family of measures $\{\sigma_y\mid y\in Y\}$ and write
\[
 \int_Y \sigma_y \,\dd\!\mu(y) := \sigma_*\mu.
\]

\subsection{A general descent criterion}
 We now consider the map 
 \[
\sigma^*: \Binfty(X) \to \Binfty(Y), \quad \sigma^*f(y) = \int_X f d\sigma_y.
\]
 \begin{lem}\label{DescendCriterion} Let $X$ and $Y$ be lcsc spaces, let $\mu_X \in \mathcal P(X)$, $\mu_Y \in \mathcal P(Y)$ and assume that 
\[
\int_Y \sigma_y d\mu_Y(y)\ll \mu_X.
\]
Then $\sigma^*$ descends to a map
\[
\sigma^*: L^\infty(X, \mu_X) \to L^\infty(Y, \mu_Y), \quad [f] \mapsto [\sigma^*f].
\]
\end{lem}
\begin{proof} Let $f_1, f_2 \in \Binfty(X)$ be $\mu_X$-equivalent in the sense that $f:=f_1 - f_2$ satisfies $\mu_X(f) = 0$.
Since $\int_Y \sigma_y d\mu_Y(y)\ll \mu_X$ this implies $\left(\int_Y \sigma_y d\mu_Y(y)\right) (f) = 0$. We deduce that
\begin{eqnarray*}
\mu_Y(\sigma^* f_1 - \sigma^* f_2) &=& \int_Y\left(\int_X f_1(x) d\sigma_y(x) -  \int_X f_2(x) d\sigma_y(x)\right)d\mu_Y(y)\\
&=&  \int_Y f(x) d\sigma_y(x)d\mu_Y(y)\\
&=&\left(\int_Y \sigma_y d\mu_Y(y)\right) (f) = 0,
\end{eqnarray*}
i.e. $\sigma^*f_1$ and $\sigma^*f_2$ are $\mu_Y$-equivalent.
\end{proof}

\subsection{Measurability of parameter-dependent integrals}
In the sequel we denote by $C_b(X)$ the space of continuous bounded functions on $X$.
\begin{lem}\label{ContinuityMeasureFamilies} If $X$ is compact and $\sigma: Y \to {\rm Prob}(X)$ is continuous, then for every continuous bounded function $f \in C_b(X \times Y)$ the map
\begin{equation}\label{widetildefintegral}
\widetilde{f}: Y \to \bbR, \quad \widetilde{f}(y) := \int_X f(x,y) d\sigma_y(x)
\end{equation}
is continuous.
\end{lem}
\begin{proof} If $y_n \to y$ in $Y$, then by assumption the functions $f_{y_n}(x) := f(x,y_n)$ converge uniformly to $f_y(x) := f(x,y)$ and $\sigma_{y_n} \to \sigma_y$ in the weak-$*$-topology. We deduce that 
\[
|\widetilde{f}(y)-\widetilde{f}(y_n)| = |\langle f_{y_n} - f_y, \sigma_y \rangle + \langle f_{y_n}, \sigma_{y_n} - \sigma_y\rangle | \leq \|f_{y_n} - f_{y}\|_\infty\|\sigma_y\| + \|f_{y_n}\|_\infty \|\sigma_{y_n} - \sigma_{y}\|.
\]
Since $\|f_{y_n}\|_\infty$ is bounded uniformly by $\|f\|_\infty$, we deduce that $|\widetilde{f}(y)-\widetilde{f}(y_n)| \to 0$.
\end{proof}
We will apply this in the following form:
\begin{cor}\label{GeneralMeasurability}
Let $Y = \bigcup_\alpha Y_\alpha$ be a partition of $Y$ into at most countably many Borel sets. 
Assume that $X$ is compact and let $\sigma: Y \to {\rm Prob}(X)$ be a Borel function and $f: X \times Y \to \bbR$ be a bounded Borel function. If for every $\alpha$ the maps $\sigma|_{Y_\alpha}$ and $f|_{X \times Y_\alpha}$ are continuous, then the function $\widetilde{f}: Y \to \bbR$ from \eqref{widetildefintegral} is Borel measurable.\qed
\end{cor}


\vspace{10pt}

\begin{thebibliography}{123}
\bibitem{Bened-Petro}
	Benedetti, R.; Petronio, C.
	\emph{Lectures on hyperbolic geometry}. 
	Universitext. Springer-Verlag, Berlin, 1992. xiv+330 pp. ISBN: 3-540-55534-X 

\bibitem{Bestvina}
	Bestvina, M.
	\emph{Homological stability of Aut($F_n$) revisited}. 
	Hyperbolic geometry and geometric group theory, 1-11,
	Adv. Stud. Pure Math., 73, 
	Math. Soc. Japan, 
	Tokyo, 2017. 

\bibitem{BBI}
	Bucher, M.; Burger, M.; Iozzi, A.
	\emph{The bounded Borel class and 3-manifold groups}.
	Duke Math. J.
	Volume 167, Number 17 (2018), 3129-3169.

\bibitem{Buehler}
	B\"uhler, T.
	\emph{On the algebraic foundations of bounded cohomology}.  
	Mem. Amer. Math. Soc. 214 (2011), no. 1006, xxii+97 pp. ISBN: 978-0-8218-5311-5 
	
\bibitem{Burger-Monod1}
	Burger, M.; Monod, N.
	\emph{Bounded cohomology of lattices in higher rank Lie groups.} 
	J. Eur. Math. Soc. (JEMS) 1 (1999), no. 2, 199-235. 
	
\bibitem{Burger-Monod2}
	Burger, M.; Monod, N.
	\emph{Continuous bounded cohomology and applications to rigidity theory}. 
	Geom. Funct. Anal. 12 (2002), no. 2, 219-280. 
	
\bibitem{Devyatov}
	Devyatov, R.
	\emph{Generically transitive actions on multiple flag varieties}.
	Int. Math. Res. Not. IMRN 2014, no. 11, 2972-2989. 
	
\bibitem{Dupont}
    	Dupont, J.L.
    	\emph{Bounds for characteristic numbers of flat bundles},
  	Algebraic topology, Aarhus 1978 (Proc. Sympos., Univ. Aarhus, Aarhus, 1978), 			Lecture Notes in Math., vol. 763, Springer, Berlin, 1979,
  	pp. 109--119.
	
\bibitem{Essert}
	Essert, J.
	\emph{Homological stability for classical groups}. 
	Israel J. Math. 198 (2013), no. 1, 169-204. 
	
\bibitem{Federer}
	Federer, H.
	\emph{Geometric measure theory}.
	Die Grundlehren der mathematischen Wissenschaften, Band 153, 
	Springer-Verlag, New York Inc., New York 1969 xiv+676 pp. 

\bibitem{Frigerio}
	Frigerio, R.
	\emph{Bounded cohomology of discrete groups}.
	Mathematical Surveys and Monographs, 227. 
	American Mathematical Society, 
	Providence, RI, 2017. 
	xvi+193 pp. 
	ISBN: 978-1-4704-4146-3 
	
\bibitem{GHV} 
	Greub, W.; Halperin S.; Vanstone R. 
	\emph{Connections, curvature and cohomology, Volume III: Cohomology of principal bundles and homogeneous spaces}. 
	Pure and Applied Mathematics, Vol. 47-III. 
	Academic Press [Harcourt Brace Jovanovich, Publishers], New York--London, 1976. xxi+593 pp. 	

\bibitem{Gromov}
	Gromov, M.
	\emph{Volume and bounded cohomology}.
	Inst. Hautes \'Etudes Sci. Publ. Math. No. 56 (1982), 5-99 (1983). 

\bibitem{Harer}
	Harer, J.L.
	\emph{Stability of the homology of the mapping class groups of orientable surfaces}.
	Ann. of Math. (2) 121 (1985), no. 2, 215-249.
	
\bibitem{Hatcher-AT}
	Hatcher, A.
	\emph{Algebraic topology}.
	Cambridge University Press, Cambridge, 2002. xii+544 pp. ISBN: 0-521-79160-X; 0-521-79540-0 

\bibitem{HatVogt}
	Hatcher, A.; Vogtmann, K.
	\emph{Homology stability for outer automorphism groups of free groups}. 
	Algebr. Geom. Topol. 4 (2004), 1253-1272. 
	
\bibitem{HatVogtWahl}
	Hatcher, A.; Vogtmann, K.; Wahl, N.
	\emph{Erratum to: "Homology stability for outer automorphism groups of free groups by Hatcher and Vogtmann."}
	Algebr. Geom. Topol. 6 (2006), 573-579. 
	
\bibitem{vdK}
	Van der Kallen, W.
	\emph{Homology stability for linear groups}.
	Invent. Math. 60 (1980), no. 3, 269-295. 

\bibitem{Koshi-Takesaki}
	Koshi, S.; Takesaki, M.
	\emph{An open mapping theorem on homogeneous spaces}. 
	J. Austral. Math. Soc. Ser. A 53 (1992), no. 1, 51-54. 
	
\bibitem{McDuffSalamon}
	McDuff, D.; Salamon, D.
	\emph{Introduction to symplectic topology}. 
	Third edition. Oxford Graduate Texts in Mathematics. 
	Oxford University Press, Oxford, 2017. xi+623 pp. ISBN: 978-0-19-879490-5; 978-0-19-879489-9.

\bibitem{Monod-Survey}
	Monod, N.
	\emph{An invitation to bounded cohomology}. 
	International Congress of Mathematicians. Vol. II, 1183-1211, 
	Eur. Math. Soc., 
	Z\"urich, 2006. 

\bibitem{Monod-Book} 
	Monod, N. 
	\emph{Continuous bounded cohomology of locally compact groups}. 
	Lecture Notes in Mathematics, 1758. 
	Springer-Verlag, 
	Berlin, 2001. 
	
\bibitem{Monod-sot}
	Monod, N.
	\emph{On the bounded cohomology of semi-simple groups, S-arithmetic groups and products}.
	J. Reine Angew. Math. 640 (2010), 167-202. 
	
\bibitem{Monod-Stab} 
	Monod, N.  
	\emph{Stabilization for $\SL_n$ in bounded cohomology}. 
	Discrete geometric analysis, 191-202, Contemp. Math., 347, 
	Amer. Math. Soc., 
	Providence, RI, 2004. 
	
\bibitem{Monod-Vanish} 
	Monod, N.  
	\emph{Vanishing up to the rank in bounded cohomology}. 
	Math. Res. Lett. 14 (2007), 
	no. 4, 681-687. 
	
\bibitem{Pieters}
	Pieters, H.
	\emph{The boundary model for the continuous cohomology of ${\rm Isom}^+(\bbH^n)$}.
	Preprint, 
	https://arxiv.org/abs/1507.04915
	
\bibitem{Popov}
	Popov, V.L.
	\emph{Generically multiple transitive algebraic group actions}. 
	Algebraic groups and homogeneous spaces, 481-523,
	Tata Inst. Fund. Res. Stud. Math., 19, Tata Inst. Fund. Res., Mumbai, 2007. 
	
\bibitem{Stas} 
    	Stasheff, J.D. 
    	\emph{Continuous cohomology of groups and classifying spaces}. 
    	Bull. Amer. Math. Soc. 84 (1978), no. 4, 513-530.

\end{thebibliography}
\end{document}